\documentclass{amsart}

\usepackage{bbm,amsmath,amsfonts,amssymb}
\usepackage[mathscr]{eucal}

\newtheorem{theorem}{Theorem}[section]
\newtheorem{lemma}[theorem]{Lemma}

\newtheorem{proposition}[theorem]{Proposition}
\newtheorem{definition}[theorem]{Definition}

\newtheorem{remark}[theorem]{Remark}

\numberwithin{equation}{section}

\def\bZ{{\mathbb Z}}

\def\bN{{\mathbb N}}

\def\bR{{\mathbb R}}

\def\cF{\mathcal{F}}

\def\cJ{\mathcal{J}}

\def\cM{\mathcal{M}}
\def\cN{\mathcal{N}}

\def\cP{\mathcal{P}}

\def\cS{\mathcal{S}}

\def\cX{\mathcal{X}}

\def\supp{{\rm{supp}\, }}
\def\eps{{\varepsilon}}
\def\Id{{\rm Id}}
\def\ONE{{\mathbbm 1}}
\def\tONE{{\tilde{\mathbbm 1}}}

\def\R{\mathbb{R}}

\def\XX{{\cX}}

\def\TT{T}
\def\SSS{S}

\def\cd{{c_\diamond}}
\def\cf{{c_\flat}}
\def\bb{\beta}

\def\ct{\tilde{c}}

\def\kap{\kappa}

\def\kaph{{\hat\kappa}}
\def\ch{{\hat c}}

\def\Cstar{{C^\star}}
\def\cstar{{c^\star}}

\def\KK{K}
\def\NN{N}

\def\BB{{\dot{B}}}
\def\tBB{{\dot{\tilde{B}}}}
\def\FF{{\dot{F}}}
\def\tFF{{\dot{\tilde{F}}}}

\def\bb{{\dot{b}}}
\def\tbb{{\dot{\tilde{b}}}}
\def\ff{{\dot{f}}}
\def\tff{{\dot{\tilde{f}}}}

\def\Aa{A}

\def\mm{N}
\def\sig{\sigma}
\def\AA{B}
\def\ad{{\mathbf{ad}}}
\def\sc{{c_*}}
\def\kk{{\ell}}
\def\LL{F}
\def\PP{\mathscr{P}}

\newcommand\norm[1]{\|#1\|}

\def\ddelta{\delta}

\def\MM{\mathscr{M}}

\def\fm{{\mathfrak m}}

\begin{document}

\title[Atomic and molecular decomposition of homogeneous spaces]
%{Atomic and molecular decomposition of Besov and Triebel-Lizorkin spaces
% associated to non-negative self-adjoint operators}
{Atomic and molecular decomposition of homogeneous spaces of distributions associated to non-negative self-adjoint operators}

\author{A. G. Georgiadis}
%\author{Athanasios G. Georgiadis}
\address{Department of Mathematical Sciences,
Aalborg University, Fredrik Bajers Vej 7G, DK-9220 Aalborg}
\email{nasos@math.aau.dk}

\author{G. Kerkyacharian}
%\author{Gerard Kerkyacharian}
\address{Laboratoire de Probabilit\'{e}s et Mod\`{e}les Al\'{e}atoires, CNRS-UMR 7599, and Crest}
%Universit\'{e} Paris VI et Universit\'{e} Paris VII, rue de Clisson, F-75013 Paris}
\email{kerk@math.univ-paris-diderot.fr}

\author{G. Kyriazis}
%\author{George Kyriazis}
\address{Department of Mathematics and Statistics,
University of Cyprus, 1678 Nicosia, Cyprus}
\email{kyriazis@ucy.ac.cy}

\author{P. Petrushev}
%\author{Pencho Petrushev}
\address{Department of Mathematics\\University of South Carolina\\
Columbia, SC 29208}
\email{pencho@math.sc.edu}

\subjclass[2010]{Primary 46E35, 58J35, 43A85 ; Secondary 42B25, 42B15, 42C15, 42C40}

\keywords{Algebra, almost diagonal operators, atomic decomposition, Besov spaces, frames, heat kernel, homogeneous spaces,
molecular decomposition, spectral multipliers, Triebel-Lizorkin spaces}

\thanks{Corresponding author: George Kyriazis,
Email: kyriazis@ucy.ac.cy}

\begin{abstract}
We deal with homogeneous Besov and Triebel-Lizorkin spaces in the setting of
a doubling metric measure space in the presence of a non-negative self-adjoint operator
whose heat kernel has Gaussian localization and the Markov property.
The class of almost diagonal operators on the associated sequence spaces is developed
and it is shown that this class is an algebra.
The boundedness of almost diagonal operators is utilized for establishing
smooth molecular and atomic decompositions for the above homogeneous Besov and Triebel-Lizorkin spaces.
Spectral multipliers for these spaces are established as well.
\end{abstract}

%\date{July 27, 2016}-Nasos
%\date{September 18, 2016}
\date{March 12, 2018}

\maketitle

\section{Introduction}\label{Introduction}

Homogeneous Besov and Triebel-Lizorkin spaces are developed in \cite{GKKP}
in the general setting of a metric measure space with the doubling property and in the presence
of a non-negative self-adjoint operator whose heat kernel has Gaussian localization and the Markov property.
Their inhomogeneous version was previously developed in \cite{CKP, KP}.
These spaces can be viewed as a natural generalization of the classical
Besov and Triebel-Lizorkin spaces on $\R^d$,
developed mainly by J. Peetre, H. Triebel, M. Frazier, and B. Jawerth,
see \cite{Peetre, Triebel-1, Triebel-2, FJ1, FJ2, FJW}.
Our goal here is to develop various aspects of the theory of
homogeneous Besov and Triebel-Lizorkin spaces in the general setting indicated above (see below),
including, almost diagonal operators on respective sequence spaces, %compactly supported frames,
atomic and molecular decompositions as well as spectral multipliers,
in analogy to the theory of Frazier and Jawerth \cite{FJ1, FJ2}.

We shall operate in the setting put forward in \cite{CKP, KP}, which we describe next:

%%%%%%%%%%%%% The setting

I. We assume that $(M,\rho,\mu)$ is a metric measure space satisfying the conditions:
$(M, \rho)$ is a locally compact  metric space with distance $\rho(\cdot, \cdot)$
and $\mu$ is a~positive Radon measure
such that the following {\em volume doubling condition} is valid
\begin{equation}\label{doubling-0}
0 < |B(x,2r)| \le c_0|B(x,r)|<\infty
\quad\hbox{for all $x \in M$ and $r>0$,}
\end{equation}
where $|B(x,r)|$ is the volume of the open ball $B(x,r)$ centered at $x$ of radius $r$ and $c_0>1$ is a constant.
From above it follows that
\begin{equation}\label{doubling}
|B(x,\lambda r)| \le c_0\lambda^d |B(x,r)|
\quad\hbox{for $x \in M$, $r>0$, and $\lambda >1$,}
\end{equation}
where $d:=\log_2 c_0 >0$ is a constant playing the role of a {\em dimension}.

We also assume that $\mu(M)=\infty$.

%%%%%%%%%%%

II. The main assumption is that the geometry of the space $(M,\rho,\mu)$ is related to
an essentially self-adjoint non-negative operator $L$ on $L^2(M, d\mu)$,
mapping real-valued to real-valued functions,
such that the associated semigroup $P_t=e^{-tL}$ consists of integral operators with
(heat) kernel $p_t(x,y)$ obeying the conditions:

\smallskip

\noindent
\begin{equation}\label{Gauss-local}
|p_t(x,y)|
\le \frac{ \Cstar\exp\{-\frac{\cstar\rho^2(x,y)}t\}}{\big[|B(x,\sqrt t)||B(y,\sqrt t)|\big]^{1/2}}
\quad\hbox{for} \;\;x,y\in M,\,t>0.
\end{equation}
%\begin{equation}\label{Gauss-local}
%|p_t(x,y)|
%\le \frac{ \Cstar\exp\{-\frac{\cstar\rho^2(x,y)}t\}}{\sqrt{\mu(B(x,\sqrt t))\mu(B(y,\sqrt t))}}
%\quad\hbox{for} \;\;x,y\in M,\,t>0.
%\end{equation}

\noindent
(b) {\em H\"{o}lder continuity:} There exists a constant $\alpha>0$ such that
\begin{equation}\label{lip}
\big|  p_t(x,y) - p_t(x,y')  \big|
\le \Cstar\Big(\frac{\rho(y,y')}{\sqrt t}\Big)^\alpha
\frac{\exp\{-\frac{\cstar\rho^2(x,y)}t \}}{\big[|B(x,\sqrt t)||B(y,\sqrt t)|\big]^{1/2}}
\end{equation}
for $x, y, y'\in M$ and $t>0$, whenever $\rho(y,y')\le \sqrt{t}$.

\smallskip

\noindent
(c) {\em Markov property:}
\begin{equation}\label{hol3}
\int_M p_t(x,y) d\mu(y)= 1
\quad\hbox{for $x\in M$ and $t >0$.}
\end{equation}
Above $\Cstar, \cstar>0$ are structural constants.

The following additional conditions on the geometry of $M$ are also stipulated:

\smallskip

\noindent
(d) {\em Noncollapsing condition:}
There exists a~constant $c_1>0$ such that
\begin{equation}\label{non-collapsing}
\inf_{x\in M}|B(x,1)|\ge c_1.
\end{equation}

\smallskip

\noindent
(e) {\em Reverse doubling condition:} There exists a constant
$c_2>1$ such that
\begin{equation}\label{RD}
|B(x,2r)|\ge c_2 |B(x,r)|
\quad\hbox{for $x \in M$ and $r>0$.}
\end{equation}

Condition (e) readily implies
\begin{equation}\label{GRD}
|B(x,\lambda r)|\ge c_3\lambda^{d^*}|B(x,r)|
\quad\hbox{for $x \in M$, $r>0$, and $\lambda>1$,}
\end{equation}
where $d^*:=\log_2 c_2\le d$ and $c_3=c_2^{-1}$.

Note that the reverse doubling condition (\ref{RD}) is not restrictive
because as shown in \cite[Proposition 2.2]{CKP} if $M$ is connected, then
(\ref{RD}) follows by the doubling condition (\ref{doubling-0}).

The above setting appears naturally in the general
framework of strictly local regular Dirichlet spaces with a complete intrinsic metric.
In~particular, this setting covers the cases of
Lie groups or homogeneous spaces with polynomial volume growth,
complete Riemannian manifolds with Ricci curvature bounded from below and satisfying
the volume doubling condition.
It also contains the classical setting on $\R^n$.
For details, see \cite{CKP}.

Analysis on metric measure spaces of homogeneous type (satisfying the doubling property) 
goes back to the celebrated work of Coifman and Weiss \cite{CW,CW2}. 
Spaces of functions or distributions associated with operators are studied during the last fifteen years. 
The literature on the subject is extensive but as a small sample we refer to \cite{DY,DY2,HLMMY} for Hardy spaces 
and \cite{BDY} for Besov spaces. 
The reverse doubling property is quite common (see e.g. \cite{HMY,YZ}), 
while more general Gaussian bounds can be used instead of (\ref{Gauss-local}) \cite{Grig,LYY}. 
Two-sided Gaussian estimates are occasionally used \cite{Barlow,Grig2}; 
in fact it is established that they imply the H\"{o}lder continuity (\ref{lip}). 
For more articles in the area we refer to \cite{Bo,CKP,Pese5,KP,LYY,YY} and the references therein.

\smallskip

This article is a followup of \cite{GKKP}, where the homogeneous Besov and Triebel-Lizorlin spaces
in the general setting described above are introduced and studied.
Our goal is to generalize a substantial part of the theory of Frazier and Jawerth \cite{FJ1, FJ2}
in the general setting of this article.
%This includes the development of almost diagonal operators on the associated sequence spaces,
%the establishment of molecular and atomic decompositions of homogeneous Besov and Triebel-Lizorlin spaces,
%and Mihlin type spectral multipliers for these spaces.

The main point in the present article is to show that in the general setting described above
it is possible to develop atomic and molecular decomposition of homogeneous Besov and Triebel-Lizorkin spaces
and Mihlin type multipliers
in almost complete generality as in the classical case on $\R^n$.
As an application, we cover new settings such as the ones on Lie groups and Riemannian manifolds.

The organization of the paper is as follows.
In Sections~\ref{Background} - \ref{sec:hom-spaces} we place all needed preliminaries from \cite{CKP, KP, GKKP},
including, smooth functional calculus, distributions, frames,
and frame decomposition of the homogeneous Besov and Triebel-Lizorkin spaces
$\BB^s_{pq}$, $\tBB^s_{pq}$, $\FF^s_{pq}$, $\tFF^s_{pq}$.
In Section~\ref{sec:almost-diag}, we develop almost diagonal operators on the associated sequence spaces,
improving on results of \cite{DKKP}, in analogy to the classical case on $\R^n$,
developed by Frazier and Jawerth \cite{FJ1, FJ2}.
In particular, we show that the almost diagonal operators form an algebra and
are bounded on the respective $\bb$- and $\ff$-sequence spaces.
%As an application we derive a decomposition in terms of compactly supported frames.
%
In analogy to the classical case on $\R^n$ (see \cite{FJ1, FJ2})
we introduce in Section~\ref{sec:molecules} smooth molecules for the spaces
$\BB^s_{pq}$, $\tBB^s_{pq}$, $\FF^s_{pq}$, $\tFF^s_{pq}$
and establish results for molecular decomposition of these spaces similar to the ones from \cite{FJ1, FJ2}.
We use these results and the compactly supported frames, essentialy developed in \cite{DKKP},
to establish atomic decomposition of the spaces as well (\S\ref{sec:atomic-decomp}).
In Section~\ref{sec:multipliers}, we use the molecular decomposition to obtain
Mihlin type spectral multipliers for the spaces
$\BB^s_{pq}$, $\tBB^s_{pq}$, $\FF^s_{pq}$, and $\tFF^s_{pq}$.
The atomic and molecular decompositions of inhomogeneous Besov and Triebel-Lizorkin spaces
are briefly discussed in Section~\ref{sec:inhomogeneous-case}.
Section~\ref{sec:appendix} is an appendix where we place the proofs of some claims from previous sections.

\smallskip

\noindent
{\em Notation}:
Throughout we shall denote
$|E|:= \mu(E)$
and $\ONE_E$ will stand for the characteristic function of  $E\subset M$,
$\|\cdot\|_p=\|\cdot\|_{L^p}:=\|\cdot\|_{L^p(M, d\mu)}$.
The Schwartz class on $\R$ will be denoted by $\cS(\R)$.
Positive constants will be denoted by $c$, $C$, $c_1$, $c'$, $\dots$ and will be allowed to vary
at every occurrence.
The notation $a\sim b$ will stand for $c_1\le a/b\le c_2$.
We shall also use the standard notation
$a\wedge b:= \min\{a, b\}$ and
$a\vee b:= \max\{a, b\}$.

\section{Background}\label{Background}

In this section we collect all basic ingredients for our theory, developed in \cite{CKP, GKKP, KP}.

\subsection{Functional calculus}\label{subsec:func-calc}

Let $E_\lambda$, $\lambda \ge 0$, be the spectral resolution associated with
the non-negative self-adjoint operator $L$ from our setting (\S\ref{Introduction}).
As $L$ maps real-valued to real-valued functions, for any real-valued, measurable and
bounded function $f$ on $\R_+$ the operator $f(L)$, defined by
$f(L):=\int_0^\infty f(\lambda)dE_\lambda$,
is bounded on $L^2(M)$, self-adjoint, and maps real-valued functions to real-valued functions.
Furthermore, if $f(L)$ is an integral operator, then its kernel $f(L)(x, y)$ is real-valued and
$f(L)(y, x)=f(L)(x, y)$,
in particular, $p_t(x, y)\in \R$ and $p_t(y,x) = p_t(x, y)$.

\smallskip

The {\bf finite speed propagation property} plays an important role in this study:
\begin{equation}\label{finite-speed}
\big\langle \cos(t\sqrt{L})f_1, f_2 \big\rangle=0,
\quad 0< \ct t<r,
\quad \ct:=\frac{1}{2\sqrt{\cstar}},
\end{equation}
for all open sets $U_j \subset M$, $f_j\in L^2(M)$, $\supp f_j\subset U_j$,
$j=1, 2$, where $r:=\rho(U_1, U_2)$.

This property is a consequence of the Gaussian localization of the heat kernel
and implies the following localization result for the kernels of operators of
the form $f(\delta\sqrt{L})$ whenever $\hat f$ is band limited, see \cite{KP}.
Here $\hat f(\xi):=\int_\R f(t)e^{-it\xi}dt$.

%%%%%%%%% Proposition

\begin{proposition}\label{prop:finite-sp}
Let $f$ be even, $\supp \hat f \subset [-A, A]$ for some $A>0$,
and $\hat f\in W^m_1$ for some $m>d$, i.e. $\|\hat f^{(m)}\|_1 <\infty$.
Then for any $\delta>0$ and $x, y\in M$
\begin{equation}\label{finite-speed-2}
f(\delta\sqrt{L})(x, y) = 0
\quad\hbox{if}\quad
%\ct \delta A < \rho(x, y).
\rho(x, y) > \ct \delta A.
\end{equation}
\end{proposition}

We shall need the following result from the
smooth functional calculus induced by the heat kernel,
developed in \cite{CKP,KP}.

%%%%%%%% Theorem

\begin{theorem}\label{thm:S-local-kernels}%\cite{KP}
Suppose $f\in C^\mm(\bR)$, $\mm\ge d+1$, $f$ is real-valued and even, and
$$
\hbox{
$|f^{(\nu)}(\lambda)|\le A_\mm(1+|\lambda|)^{-r}$ for $\lambda\in\bR$ and $0\le \nu\le \mm$, where $r > \mm+d$.
}
$$
%and $f^{(2\nu+1)}(0)=0$ for $\nu\ge 0$ such that $2\nu+1 \le \mm$.
%
Then $f(\ddelta \sqrt L)$, $\ddelta>0$, is an integral operator with kernel $f(\ddelta \sqrt L)(x, y)$
satisfying
\begin{equation}\label{local-ker}
\big|f(\ddelta \sqrt L)(x, y)\big|
\le \frac{cA_\mm \big(1+\ddelta^{-1}\rho(x, y)\big)^{-\mm}}
{\big(|B(x, \delta)||B(y, \delta)|\big)^{1/2}}
\le \frac{c'A_\mm \big(1+\ddelta^{-1}\rho(x, y)\big)^{-\mm+d/2}}
{|B(x, \delta)|}
%\quad\hbox{and}
\end{equation}
and
\begin{equation}\label{lip-ker}
\big|f(\ddelta \sqrt L)(x, y)-f(\ddelta \sqrt L)(x, y')\big|
\le \frac{cA_\mm \big(\frac{\rho(y, y')}{\ddelta}\big)^\alpha\big(1+\ddelta^{-1}\rho(x, y)\big)^{-\mm}}
{\big(|B(x, \delta)||B(y, \delta)|\big)^{1/2}}
\end{equation}
whenever $\rho(y, y')\le \ddelta$.
Here $\alpha>0$ is from $(\ref{lip})$
and %$c_m= cA_m$ with
$c, c'>0$ are constants depending only on $r$, $\mm$, and
the structural constants $c_0, \Cstar, \cstar$, $\alpha$.
%the structural constant of the setting.

Moreover, $\int_M f(\ddelta \sqrt L)(x, y)d\mu(y)=f(0)$.
\end{theorem}

In the construction of frames we utilize operators of the form
$\varphi(\delta \sqrt{L})$
generated by cutoff functions $\varphi$ specified in the following

%%%%%%%%%%% Definition

\begin{definition}\label{cutoff-d1}
A real-valued function $\varphi\in C^\infty (\R_+)$ is said to be an admissible cutoff function if $\varphi \ne 0$,
 $\supp \varphi \subset [0, 2]$, and $\varphi^{(m)}(0)=0$ for $m\ge 1$.
Furthermore, $\varphi$ is said to be admissible of
type $(a)$, $(b)$ or $(c)$ if $\varphi$ is admissible and in addition obeys the respective condition:

$(a)$ $\varphi(t) = 1$, $t\in [0, 1]$,

$(b)$ $\supp \varphi \subset [1/2, 2]$ or

$(c)$ $\supp \varphi \subset [1/2, 2]$ and $\sum_{j\in\bZ} |\varphi(2^{-j}t)|^2=1$ for $t\in (0,\infty)$.
\end{definition}

The kernels of operators of the form $\varphi(\delta \sqrt{L})$ with
sub-exponential space localization
will be the main building blocks in constructing frames.

%%%%%%% Theorem

\begin{theorem}\label{thm:band-sub-exp} $($\cite{KP}$)$
For any $0<\eps<1$ there exists a cutoff function $\varphi$ of any type, $(a)$ or $(b)$ or $(c)$,
such that for any $\delta >0$ and $m\ge0$
\begin{equation}\label{band-sub-exp-m}
|[L^m\varphi(\delta \sqrt{L})](x, y)|
\le \frac{c_2\delta^{-2m}\exp\Big\{-\kap\big(\frac{\rho(x, y)}{\delta}\big)^{1-\eps}\Big\}}
{\big(|B(x, \delta)||B(y, \delta)|\big)^{1/2}},
\quad x, y\in M,
\end{equation}
where $c, \kap >0$ depend only on $\eps$, $m$ and
the constants $c_0, \Cstar, \cstar$ from $(\ref{doubling-0})-(\ref{lip}).$
\end{theorem}

\begin{remark}\label{rem:calculus}
Note that $[L^m\varphi(\delta \sqrt{L})](x, y)$ in $(\ref{band-sub-exp-m})$
is the kernel of the operator $L^m\varphi(\delta \sqrt{L})$,
however, it can be considered as $L^m$ acting on the kernel $\varphi(\delta \sqrt{L})(\cdot, y)$
or $L^m$ acting on $\varphi(\delta \sqrt{L})(x, \cdot)$ as well.
In fact the result is the same: For any $x, y\in M$
\begin{equation}\label{L-m-phi}
[L^m\varphi(\delta \sqrt{L})](x, y)= L^m[\varphi(\delta \sqrt{L})(\cdot, y)](x)
= L^m[\varphi(\delta \sqrt{L})(x, \cdot)](y).
\end{equation}
Furthermore, if $f, g$ are real-valued functions on $\R$, obeying the hypotheses of Theorem ~\ref{thm:S-local-kernels}
and $F(\sqrt{L}):=f(\sqrt{L})g(\sqrt{L})$, then
\begin{equation}\label{fg-F}
F(\sqrt{L})(x, y)= f(\sqrt{L})[g(\sqrt{L})(\cdot, y)](x)
= f(\sqrt{L})[g(\sqrt{L})(x,\cdot)](y),
\quad \forall x,y\in M.
\end{equation}
The above claims follow from the more general result in \cite[Proposition~2.7]{GKKP}.
\end{remark}

%%%%%%% Proposition
%
%\begin{proposition}\label{prop:F-HG}
%Assume $F$ and $G$ satisfy the hypotheses of Theorem~\ref{thm:S-local-kernels} with $m\ge 3d/2+1$
%and let $H$ be a real-valued measurable function on $\R_+$ such that
%\begin{equation}\label{F-HG-1}
%F(\lambda)=H(\lambda)G(\lambda)
%\quad\hbox{for almost all $\lambda\in \R_+$.}
%\end{equation}
%Then $F(\sqrt{L})$ and $G(\sqrt{L})$ are self-adjoint bounded on $L^2(M)$ operators,
%and $H(\sqrt{L})$ is a~self-adjoint operator $($defined densely in $M)$
%such that for all $x\in M$ we have
%$G(\sqrt{L})(x, \cdot)\in D(H(\sqrt{L}))$
%and
%\begin{equation}\label{F-HG-2}
%F(\sqrt{L})(x, y) = H(\sqrt{L})\big[G(\sqrt{L})(x, \cdot)\big](y)
%\quad \hbox{for a.a. $y \in M$.}
%%\;\; \hbox{for $\mu$-a.a. $y\in M$ and $\mu$-a.a. $x\in M$.}
%\end{equation}
%Above $D(H(\sqrt{L}))$ stands for the domain of $H(\sqrt{L})$
%and $F(\sqrt{L})(x, y)$, $G(\sqrt{L})(x, y)$ are the kernels of the operators $F(\sqrt{L})$, $G(\sqrt{L})$.
%\end{proposition}
%See \cite[Proposition 2.7]{GKKP}.

\subsection{Some properties related to the geometry of the underlying space}\label{subsec:geometry}

To compare the volumes of balls with different centers $x, y\in M$ and the same radius $r$
we shall use the inequality
%\begin{equation}\label{D2}
%|B(x, r)| \le c_0\big(1+ \rho(x,y)/r\big)^d  |B(y, r)|,
%\quad x, y\in M, \; r>0.
%\end{equation}
\begin{equation}\label{D2}
|B(x, r)| \le c_0\Big(1+ \frac{\rho(x,y)}{r}\Big)^d  |B(y, r)|,
\quad x, y\in M, \; r>0.
\end{equation}
As $B(x,r) \subset B(y, \rho(y,x) +r)$ the above inequality is immediate from (\ref{doubling}).

\smallskip

The following simple inequalities are established in \cite{GKKP,KP}:
\begin{lemma}If $\sigma >d$ and $\ddelta>0$, then for any $x\in M$
\begin{equation}\label{tech-1}
\int_M \big(1+\ddelta^{-1}\rho(x, u)\big)^{-\sigma} d\mu(u) \le c|B(x, \ddelta)|.
%\quad \forall x\in M
\end{equation}
\end{lemma}

\begin{lemma}\label{lem:gen-ineq}
Let $\sig_1, \sig_2> d$ and $\delta_1, \delta_2>0$, and
\begin{equation}\label{def-I}
I:=\int_M \frac{d\mu(u)}{\big(1+\delta_1^{-1}\rho(x, u))^{\sig_1} \big(1+\delta_2^{-1}\rho(y, u))^{\sig_2}}
\end{equation}
Then for any $x, y\in M$
\begin{equation}\label{int-est-1}
I \le \frac{c|B(x, \delta_1)|}{\big(1+\delta_2^{-1}\rho(x, y))^{\sig_2}}
+ \frac{c|B(y, \delta_2)|}{\big(1+\delta_1^{-1}\rho(x, y))^{\sig_1}}
\end{equation}
and consequently
\begin{equation}\label{int-est-2}
I \le \frac{c|B(x, \delta_1)|}{\big(1+\delta_{\max}^{-1}\rho(x, y))^{(\sig_1-d)\wedge\sig_2}}
\quad\hbox{and}\quad
I \le \frac{c|B(y, \delta_2)|}{\big(1+\delta_{\max}^{-1}\rho(x, y))^{\sig_1\wedge(\sig_2-d)}}.
\end{equation}
Here $\delta_{\max}:= \delta_1\vee\delta_2$ and the constant $c>0$ depends only on $\sig_1$, $\sig_2$, $d$, and $c_0$.
\end{lemma}

\subsection*{Maximal \boldmath $\delta$-nets}

The construction of frames in our setting relies on a sequence of $\delta$-nets. 
Such $\delta$-nets on manifolds have been used earlier in \cite{Pese3} 
and in the current framework in \cite{CKP,GKKP,KP}. 
Note that in \cite{LYY} the frames where built in an alternative way 
by using the Christ's ``dyadic cubes" \cite{Chr}.

By definition $\cX\subset M$ is a $\delta$-net on $M$ ($\delta >0$)
if $\rho(\xi, \eta) \ge \delta$, $\forall \xi, \eta\in \cX$,
and $\cX\subset M$ is a {\em maximal $\delta$-net} on $M$ if $\cX$ is a $\delta$-net on $M$
that cannot be enlarged.

Some basic properties of maximal $\delta$-nets will be needed (see \cite[Proposition 2.5]{CKP}):
{\em A maximal $\delta$-net on $M$ always exists and
if $\cX$ is a maximal $\delta$-net on $M$, then
\begin{equation}\label{net-prop}
M=\cup_{\xi\in\cX} B(\xi, \delta)
\quad\hbox{and}\quad
B(\xi, \delta/2)\cap B(\eta, \delta/2)=\emptyset
\quad\hbox{if}\;\; \xi\ne\eta, \; \xi, \eta\in\cX.
\end{equation}
Furthermore, $\cX$ is countable
and there exists a~disjoint partition $\{A_\xi\}_{\xi\in\cX}$ of $M$ consisting
of measurable sets such that
\begin{equation}\label{B-Axi-B}
B(\xi, \delta/2) \subset A_\xi \subset B(\xi, \delta), \quad \xi\in\cX.
\end{equation}
}
%For future use we introduce the following notation for a given maximal $\delta$-net $\cX$ on $M$:
%\begin{equation}\label{B-xi}
%\AA_\xi:=B(\xi, \delta), \quad \xi\in\cX.
%\end{equation}

The next lemma is a discrete counterpart of Lemma~\ref{lem:gen-ineq}; its proof is deferred to the appendix.

\begin{lemma}\label{lem:est-discr-sum}
Let $\sigma>d$ and $0<\delta \le \delta_1\le \delta_2$.
Suppose $\cX\subset M$ is a $\delta$-net on $M$.
Then
\begin{equation}\label{est-discr-sum}
\sum_{\xi\in\cX}\frac{1}{\left(1+\delta_1^{-1}\rho(x,\xi)\right)^\sigma\left(1+\delta_2^{-1}\rho(y,\xi)\right)^\sigma}
\le \frac{c (\delta_1/\delta)^d}{\left(1+\delta_2^{-1}\rho(x,y)\right)^\sigma},
\quad \forall x, y\in M,
\end{equation}
where the constant $c>0$ depends only on $\sigma$, $d$, and the constant $c_0$ from $(\ref{doubling-0})$.
\end{lemma}

\subsection{Maximal inequality}\label{max-operator}

We shall use the version $\cM_t$ ($t>0$) of the maximal operator defined by
\begin{equation}\label{def:max-op}
\cM_tf(x):=\sup_{B\ni x}\left(\frac1{|B|}\int_B |f|^t\, d\mu \right)^{1/t},
\quad x\in M,
\end{equation}
where the sup is over all balls  $B\subset M$ such that $x\in B$.

With $\mu$ being a Radon measure on $M$ obeying the doubling condition (\ref{doubling})
the general theory of maximal operators applies
and the Fefferman-Stein vector-valued maximal inequality holds \cite{Stein}:
If $0<p<\infty, 0<q\le\infty$, and
$0<t<\min\{p,q\}$ then for any sequence of functions
$\{f_\nu\}$ on $M$
\begin{equation}\label{max-ineq}
\Big\|\Bigl(\sum_{\nu}|\cM_tf_\nu(\cdot)|^q\Bigr)^{1/q} \Big\|_{L^p}
\le c\Big\|\Bigl(\sum_{\nu}| f_\nu(\cdot)|^q\Bigr)^{1/q}\Big\|_{L^p}.
\end{equation}

\subsection{Spectral spaces}\label{spectral-spaces}

As before $E_\lambda$, $\lambda\ge 0$, will be the spectral resolution associated with
the self-adjoint positive operator $L$ on $L^2:=L^2(M, d\mu)$.
We let $F_\lambda$, $\lambda\ge 0$, denote the spectral resolution associated with $\sqrt L$,
i.e. $F_\lambda=E_{\lambda^2}$.
Then for any measurable and bounded function $f$ on $\R_+$ the operator $f(\sqrt L)$ is defined by
$f(\sqrt L)= \int_0^\infty f(\lambda) d F_\lambda$ on $L^2$.
For the spectral projectors we have
$E_\lambda = \ONE_{[0, \lambda]}(L) := \int_0^\infty \ONE_{[0,\lambda]}(u) dE_u$
and
\begin{equation}\label{spect-projector}
F_\lambda = \ONE_{[0,  \lambda]}(\sqrt L)
:=\int_0^\infty  \ONE_{[0,\lambda]}(u)dF_u
=\int_0^\infty  \ONE_{[0,\lambda]}(\sqrt u)dE_u.
\end{equation}
For any compact $K \subset [0, \infty)$ the spectral space $\Sigma^p_K$ is defined by
$$
\Sigma^p_K
:= \{f\in L^p: \theta (\sqrt L)f = f \hbox{ for all }
\theta \in C^\infty_0(\bR_+), \;  \theta \equiv 1 \hbox{ on } \;  K\}.
$$
In general, given a space $Y$ of measurable functions on $M$ we set
$$
\Sigma_\lambda = \Sigma_\lambda(Y)
:= \{f\in Y: \theta (\sqrt L)f = f \hbox{ for all }
\theta \in C^\infty_0(\bR_+), \;  \theta \equiv 1 \hbox{ on } \;  [0, \lambda]\}.
$$

For the use of band limited functions on similar settings, such as Lie groups, we refer the reader to \cite{Pese,Pese2}.
\section{Distributions}\label{sec:distributions}

Homogeneous Besov and Triebel-Lizorkin spaces
associated with the operator $L$ are spaces of distributions modulo generalized polynomials,
and are introduced in \cite{GKKP}. Here we collect some basic facts from \cite{GKKP,KP}.

\subsection{Basic facts}\label{subsec:basic-facts}

In the setting of this article the class of \textit{test functions} $\cS$
is defined (see \cite{KP})  as the set of all complex-valued functions
$\phi\in  \cap_{m\ge 1} D(L^m)$ such that
\begin{equation}\label{norm-S}
\cP_{m}(\phi)
:= \sup_{x\in M} \big(1+\rho(x, x_0)\big)^m \max_{0\le \nu\le m}|L^\nu\phi(x)| < \infty,
\quad \forall m \ge 0.
\end{equation}
Here $x_0\in M$ is selected arbitrarily and fixed once and for all.
Note that $\cS$ is a complete locally convex space with topology generated by
the above sequence of norms, i.e. $\cS$ is a Fr\'{e}chet space, see \cite{RS}.

Observe also that if $\phi\in\cS$, then $\overline{\phi} \in \cS$,
which follows from the fact that $\overline{L\phi}=L \overline{\phi}$,
for $L$ maps real-valued to real-valued functions.

The space $\cS'$ of \textit{distributions} on $M$ is defined as the set of all
continuous linear functionals on $\cS$
and the action of $f\in \cS'$ on $\overline{\phi}\in \cS$ will be denoted by
$\langle f, \phi\rangle:= f(\overline{\phi})$,
which is consistent with the inner product on $L^2(M)$.

Let us clarify the action of operators of the form $\varphi(\sqrt L)$ on  $\cS'$.
%for even and real-valued functions in $\cS(\R)$.
Observe that if the function $\varphi\in \cS(\R)$, the Schwarz class on $\R$, is real-valued and even, then
from Theorem~\ref{thm:S-local-kernels} it follows that
$\varphi(\sqrt{L})(x, \cdot)\in \cS$
and $\varphi(\sqrt{L})(\cdot, y)\in \cS$.
Moreover, it is easy to see that  $\varphi(\sqrt L)$ maps continuously $\cS$  into $\cS$.

%%%%%%%%%%% Definition

\begin{definition}\label{def:phi-distr}
We define $\varphi(\sqrt L)f$ for any $f\in \cS'$ by
\begin{equation}\label{phi-distr}
\langle \varphi(\sqrt L)f, \phi \rangle := \langle f, \varphi(\sqrt L)\phi \rangle
\quad\hbox{for}\quad \phi \in \cS.
\end{equation}
\end{definition}

From above it follows that, $\varphi(\sqrt L)$ maps continuously $\cS'$  into $\cS'$.
Furthermore, if $\varphi, \psi\in \cS(\R)$ are real-valued and even, then
\begin{equation}\label{comute}
\varphi(\sqrt{L})\psi(\sqrt{L})f =  \psi(\sqrt{L})\varphi(\sqrt{L}) f, \quad \forall f\in\cS'.
\end{equation}

\subsection{Distributions modulo generalized polynomials}\label{subsec:distr-mod-pol}

We recall first the definition of generalized polynomials associated with the operator $L$.

\noindent
{\bf Generalized polynomials.} In the setting of this article,
we define the set $\PP_m$ of ``generalized polynomials" of degree $m$ $( m \ge 1)$ by
\begin{equation}\label{def-polyn}
%\Pi_m
\PP_m := \{g\in \cS': L^m g=0\}
\end{equation}
and set $\PP:= \cup_{m\ge 1} \PP_m$.
Clearly, $g\in \PP_m$ if and only if $\langle g, L^m\phi\rangle = 0$ for all $\phi\in \cS$.

%The homogeneous Besov and Triebel-Lizorkin spaces will consist of
%tempered distribution modulo generalized polynomials.

We define an equivalence $f\sim g$ on $\cS'$ by
$$
f\sim g \Longleftrightarrow f-g\in \PP.
$$
We denote by $\cS'/\PP$ the set of all equivalent classes in $\cS'$.
To avoid unnecessary complicated notation we shall make no difference between
any two elements $f_1, f_2$ belonging to one and the same equivalence class in $\cS'/\PP$.

Note that the null space of $L$ contains no nontrivial test functions:

\begin{proposition}\label{prop:null-space}
Let $\cN(L^k):=\{f\in D(L^k): L^k f=0\},\;\forall k\ge0$. Then
\begin{equation}\label{null-1}
\cN(L)\cap L^2(M) =\{0\}
\quad\hbox{and hence}\quad
\cN(L^k)\cap L^2(M) =\{0\}, \;\; \forall k\in \bN.
\end{equation}
\end{proposition}

\noindent
{\bf The classes \boldmath $\cS_\infty$ and $\cS'_\infty$.}
Denote by $\cS_\infty$ the set of all functions $\phi\in\cS$ such that for
every $k\ge 1$ there exists $\omega_k \in \cS$ such that $\phi = L^k\omega_k$,
that is, $L^{-k}\phi\in\cS$ for all $k\ge 1$.
Note that from Proposition~\ref{prop:null-space} it follows that
$\omega_k$ above is unique and hence $L^{-k}\phi$ is well defined.

The topology in $\cS_\infty$ is defined by the sequence of norms
\begin{equation}\label{norm-S-infty}
\cP_{m}^\star(\phi)
:= \sup_{x\in M} \big(1+\rho(x, x_0)\big)^m \max_{-m\le \nu\le m}|L^\nu\phi(x)|,
\quad m \ge 0.
\end{equation}
We denote by $\cS'_\infty$ the set of all continuous linear functional on $\cS_\infty$.
As before the the action of $f\in \cS'_\infty$ on $\overline{\phi}\in \cS_\infty$ will be denoted by
$\langle f, \phi\rangle$.
Apparently, for any $f\in \cS'_\infty$ there exist constants $m\in \bZ_+$ and $c>0$ such that
\begin{equation}\label{distribution-1-infty}
|\langle f, \phi\rangle| \le c\cP_m^\star(\phi), \quad \forall \phi\in\cS_\infty.
\end{equation}

Several remarks are in order:

(1)
Let $\theta \in \cS(\R)$ %the Schwarz class on $\R$,
be real-valued and
$\theta^{(\nu)}(0)=0$ for $\nu=0, 1, \dots$.
Then for any $k\ge 1$ we have $\lambda^{-2k}\theta(\lambda)\in \cS(\R)$,
which implies that $L^{-k}\theta(\sqrt{L})\phi \in \cS$ for each $\phi\in\cS$
and hence $\theta(\sqrt{L})\phi \in \cS_\infty$, $\forall \phi\in \cS$.

\smallskip

(2)
Clearly, if $\phi\in\cS_\infty$, then $L^k\phi\in\cS_\infty$ and $L^{-k}\phi\in\cS_\infty$, $\forall k\ge0$.

\smallskip

(3)
It is important to note that $\cS_\infty$ is a Fr\'{e}chet space, since it is complete.

\smallskip

(4) If $\varphi \in \cS(\R)$ be even and real-valued, then
\begin{equation}\label{prop-1}
L^{-k}\varphi(\sqrt{L})\phi = \varphi(\sqrt{L})L^{-k}\phi , \quad \forall \phi\in \cS_\infty, \; \forall k\ge 1,
\end{equation}
and hence
\begin{equation}\label{prop-2}
\varphi(\sqrt{L})\phi \in \cS_\infty, \quad \forall \phi\in \cS_\infty.
\end{equation}
Moreover, $\varphi(\sqrt{L})$ maps $\cS_\infty$ into $\cS_\infty$ continuously.

(5)
The action of operators of the form $\varphi(\sqrt L)$ on $\cS'_\infty$,
where $\varphi\in\cS(\R)$ is real-valued and even,
needs some further clarification:

%%%%%%%%%%% Definition

\begin{definition}\label{def:phi-distr-infty}
We define $\varphi(\sqrt L)f$ for any $f\in \cS'_\infty$ by
\begin{equation}\label{phi-distr-infty}
\langle \varphi(\sqrt L)f, \phi \rangle := \langle f, \varphi(\sqrt L)\phi \rangle
\quad\hbox{for}\quad \phi \in \cS_\infty.
\end{equation}
\end{definition}

From (4) above it follows that, $\varphi(\sqrt L)$ maps continuously $\cS'_\infty$  into $\cS'_\infty$.
In addition, if $\varphi, \psi\in \cS(\R)$ are real-valued and even, then
\begin{equation}\label{comute-infty}
\varphi(\sqrt{L})\psi(\sqrt{L})f =  \psi(\sqrt{L})\varphi(\sqrt{L}) f, \quad \forall f\in\cS'_\infty.
\end{equation}

%%%%%%%%%%%%% Proposition

\begin{proposition}\label{prop:distr-infty}
Suppose $\varphi\in\cS(\R)$ is real-valued and even and $\varphi^{(\nu)}(0)=0$ for $\nu=0, 1, \dots$.
Then for any $f\in \cS'_\infty$
\begin{equation}\label{phif-infty}
\varphi(\sqrt L)f(x) = \langle  f , \varphi(\sqrt L)(x,\cdot)\rangle, \quad x\in M.
\end{equation}
Moreover, $\varphi(\sqrt L)f$ is a continuous and slowly growing function.
\end{proposition}

%%%%%%% Proposition

\begin{proposition}\label{prop:identify}
We have the following identification:
\begin{equation}\label{identification}
\cS'/\PP = \cS'_\infty.
\end{equation}
\end{proposition}

\smallskip

From Proposition~\ref{prop:identify} it follows that for a sequence
$\{f_j\}\subset \cS'/\PP$ and $f\in \cS'/\PP$ we have
\begin{equation}\label{converge}
f_j \to f \;\;\hbox{in}\;\;  \cS'/\PP
\quad\hbox{if and only if}\quad
\langle f_j, \phi\rangle \to \langle f, \phi\rangle, \quad \forall \phi\in \cS_\infty.
\end{equation}

%The proof of this proposition is a straightforward adaptation of the proof
%of Proposition~2.6 in \cite{DKKP2} and will be omitted.

The main decomposition result takes the form:

%%%%%% Theorem

\begin{theorem}\label{thm:decomp-SP}
Let $\Psi \in C^\infty(\R_+)$,
$\supp \Psi \subset [b^{-1}, b]$ with $b >1$, $\Psi$ real-valued, and
\begin{equation}\label{dec-unity}
\sum_{j\in \bZ}\Psi(b^{-j}\lambda) =1 \quad \hbox{for} \quad \lambda \in (0, \infty).
\end{equation}
Then for any $f\in \cS'/\PP$
\begin{equation}\label{decomp-SP1}
f=\sum_{j\in \bZ}\Psi(b^{-j}\sqrt L)f \;\;\mbox{ in }\; \cS'/\PP,
\end{equation}
that is, for any $f\in \cS'_\infty$
\begin{equation}\label{decomp-SP2}
\lim_{n, m\to \infty}\sum_{j=-n}^m \big\langle \Psi(b^{-j}\sqrt L)f, \phi\big\rangle
= \langle f, \phi\rangle,
%\quad\hbox{for all}\quad
\quad\forall \phi\in\cS_\infty.
\end{equation}

\end{theorem}

\begin{remark}
In the case when $M=\R^d$ and $L=-\Delta$ $($the Laplacian$)$
the distributions modulo generalized polynomials $\cS'/\PP$
introduced in \S\ref{subsec:distr-mod-pol} are
just the classical tempered distributions modulo polynomials on $\R^d$.
Therefore, our general setting covers the classical case on $\R^d$.
\end{remark}

\section{Frames}\label{sec:frames}

In the setting of this article frames are constructed in \cite{CKP,GKKP,KP}. 
For frames on compact homogeneous manifolds see \cite{Pese4}.

We next recall the construction of frames.

\smallskip

\noindent
{\bf Construction of Frame \# 1.}
We first apply Theorem~\ref{thm:band-sub-exp}
for the construction of a real-valued cutoff function $\Phi$ with the following properties:
$\Phi\in C^\infty (\bR_+)$,
$\Phi(u)=1$ for $u\in[0, 1]$,
$0\le \Phi \le 1$, and
$\supp \Phi \subset [0, b]$, where $b > 1$ is a constant, see \cite{KP}.
Set
\begin{equation}\label{def:Psi}
\Psi(u):=\Phi(u)-\Phi(bu).
\end{equation}
Observe that $\Psi\in C^\infty (\bR_+)$ and
$\supp \Psi \subset [b^{-1}, b]$.
By Theorem~\ref{thm:band-sub-exp} $\Psi(\delta\sqrt{L})$
is an integral operator with kernel $\Psi(\delta\sqrt{L})(x, y)$
of sub-exponential localization, that is,
\begin{equation}\label{sub-exp-Phi}
|\Psi(\delta\sqrt{L})(x, y)|
\le \frac{\cd \exp\big\{-\kappa\big(\frac{\rho(x, y)}{\delta}\big)^\beta\big\}}
{\big(|B(x, \delta)||B(y, \delta)|)^{1/2}},
\quad \forall x, y\in M.
\end{equation}
%with
%\begin{equation}\label{def-E}
%E(x, y; \delta, \kap)
%:= \frac{\exp\big\{-\kappa\big(\frac{\rho(x, y)}{\delta}\big)^\beta\big\}}
%{\big(|B(x, \delta)||B(y, \delta)|)^{1/2}}.
%\end{equation}
Here $0<\beta<1$ is an arbitrary constant (as close to $1$ as we wish)
and $\kappa>0$ and $\cd >1$ are constants depending only on
$\beta$, $b$, and the constants $c_0, \Cstar, \cstar$ from (\ref{doubling-0})-(\ref{lip}).
Furthermore, for any $m\ge 1$
\begin{equation}\label{m-sub-exp-Phi}
|[L^m\Psi(\delta \sqrt{L})](x, y)|
\le \frac{c_m\delta^{-2m} \exp\big\{-\kappa\big(\frac{\rho(x, y)}{\delta}\big)^\beta\big\}}
{\big(|B(x, \delta)||B(y, \delta)|)^{1/2}},
\quad \forall x, y\in M.
\end{equation}
Set
\begin{equation}\label{def-Psi-j}
\Psi_j(u):=\Psi(b^{-j}u),\;\; j\in \bZ.
\end{equation}
Clearly, $\Psi_j\in C^\infty (\bR_+)$, $0\le \Psi_j \le 1$,
$\supp \Psi_j \subset [b^{j-1}, b^{j+1}]$, $j\in \bZ$, and
$$
\sum_{j\in \bZ}\Psi_j(u) = 1 \quad\hbox{ for $u\in (0, \infty)$.}
$$
Therefore, by Theorem~\ref{thm:decomp-SP} for any $f\in \cS'/\PP$
\begin{equation}\label{repres-Psi-j}
f= \sum_{j\in \bZ} \Psi_j(\sqrt L)f
\quad\hbox{(convergence in $\cS'/\PP$).}
\end{equation}

The sampling Theorem~4.2 from \cite{CKP} will play an important r\^{o}le in this construction.
In particular, this theorem yields the following

\begin{proposition}\label{prop:sampling}
For any $\varepsilon>0$
there exists a constant $\gamma$ $(0<\gamma<1)$ such that
for any maximal $\delta-$net $\cX$ on $M$ with $\delta:=\gamma \lambda^{-1}$, $\lambda>0$,
and a companion disjoint partition $\{A_\xi\}_{\xi\in\XX}$ of $M$ as in \S\ref{subsec:geometry}
consisting of measurable sets such that
$B(\xi, \delta/2) \subset  A_\xi \subset B(\xi,\delta)$, $\xi \in \XX$,
we have
\begin{equation}\label{sampling-L2}
(1-\eps)\| f \|_2^2 \le \sum_{\xi \in \XX} |A_\xi|| f(\xi)|^2 \le (1+\eps)\| f \|_2^2,
\quad \forall f\in \Sigma_\lambda^2.
\end{equation}
\end{proposition}

At this point, we introduce a constant $0<\varepsilon <1$ that will be specified later on.
We use the above proposition to produce for each $j\in\bZ$ a maximal $\delta_j$-net $\cX_j$ on $M$
with $\delta_j:=\gamma b^{-j-2}$
and a disjoint partition $\{A_\xi\}_{\xi\in\XX_j}$ of $M$ such that
\begin{equation}\label{sampling-Xj}
(1-\eps)\| f \|_2^2 \le \sum_{\xi \in \XX_j} |A_\xi|| f(\xi)|^2 \le (1+\eps)\| f \|_2^2,
\quad \forall f\in \Sigma_{b^{j+2}}^2.
\end{equation}
Set
$\cX:=\cup_{j\in \bZ}\cX_j$,
where equal points from different sets $\cX_j$ will be regarded as distinct elements of $\cX$,
and hence $\cX$ can be used as an index set.

Frame \# 1 $\{\psi_{\xi}\}_{\xi\in\cX}$ is defined by
\begin{equation}\label{def-frame}
\psi_{\xi} (x):= |A_\xi |^{1/2} \Psi_j(\sqrt L)(x,\xi),
\quad \xi \in \XX_j,\; j\in \bZ.
\end{equation}

\subsection*{Construction of Frame \# 2}

A dual frame $\{\tilde\psi_{\xi}\}$ was constructed similarly as in \cite{KP} with properties
similar to the properties of $\{\psi_{\xi}\}$.

The first step in this construction is to introduce a cutoff function
\begin{equation}\label{def:Gamma}
\Gamma(u)=\Phi(b^{-2}u)-\Phi(bu),
\end{equation}
where $\Phi$ is from the construction of Frame \#1.
Clearly, $\supp \Gamma\subset[b^{-1},b^3]$ and $\Gamma=1$ on $[1, b^2]$,
implying $\Gamma(u)\Psi_1(u)=\Psi_1(u)$.

The construction of Frame \# 2 hinges on the following

%%%%%%%%% Lemma

\begin{lemma}\label{lem:instrument}$($\cite{CKP}$)$
There exists a constant $0<\varepsilon<1$ such that the following claim holds true.
Given $\lambda>0$, let $\XX$ be a~maximal $\delta-$net on $M$,
where
$\delta:=\gamma\lambda^{-1}b^{-3}$
with $\gamma$ the constant from Proposition~\ref{prop:sampling},
and suppose $\{A_\xi\}_{\xi\in\XX}$ is a companion disjoint partition of $M$
consisting of measurable sets such that
$B(\xi, \delta/2) \subset  A_\xi \subset B(\xi,\delta)$, $\xi \in \XX$ $(\S\ref{subsec:geometry})$.
%just as in \S\ref{sec:max-d-nets}.
%
Set $\omega_\xi := (1+\varepsilon)^{-1}|A_\xi| \sim  |B(\xi, \delta)|$.
Then there exists a linear operator $\TT_\lambda: L^2(M)\to L^2(M)$
of the form $\TT_\lambda = \Id + \SSS_\lambda$ such that

$(a)$
$$
\|f \|_2\le \|\TT_\lambda f\|_2 \le \frac 1{1-2\varepsilon}\|f \|_2,
\quad \forall f \in L^2.
$$

$(b)$
$\SSS_\lambda$ is an integral operator with kernel
$\SSS_\lambda(x, y)$ verifying
\begin{equation}\label{local-S}
|\SSS_\lambda(x,y)|
\le \frac{c\exp\big\{-\frac{\kappa}{2}\big(\lambda\rho(x, y)\big)^\beta\big\}}
{\big(|B(x, \lambda^{-1})||B(y, \lambda^{-1})|)^{1/2}},
%\le  cE(x,y; \lambda^{-1}, \kappa/2),
\quad \forall x,y\in M.
\end{equation}

$(c)$
$\SSS_\lambda (L^2)\subset \Sigma_{[\lambda b^{-1}, \lambda b^3]}^2$.

$(d)$ For any $f\in L^2(M)$ such that $ \Gamma(\lambda^{-1} \sqrt L)f =f$ we have
\begin{equation}\label{instr-1}
f(x) =  \sum_{\xi \in \XX}  \omega_\xi  f(\xi)  \TT_\lambda [\Gamma_\lambda(\cdot, \xi)](x),
\quad \forall x\in M,
\end{equation}
where $\Gamma_\lambda(\cdot, \cdot)$ is the kernel of the operator
$\Gamma_\lambda:=\Gamma(\lambda^{-1} \sqrt L)$ with $\Gamma$ from $(\ref{def:Gamma})$.
\end{lemma}

We use the above lemma to select the constant $\varepsilon$ ($0<\varepsilon <1$)
that was used in the construction of Frame \#1.

Let $\cX_j$ and $\{A_\xi\}_{\xi\in\cX_j}$ be as in the definition of Frame \#1.
Denote briefly $\Gamma_{\lambda_{j}}=\Gamma(b^{-j+1}\sqrt{L})$
for $j\in \bZ$ with $\lambda_{j}:=b^{j-1}$,
and let $\TT_{\lambda_j} = \Id + \SSS_{\lambda_j}$
be the operator from Lemma~\ref{lem:instrument}, applied with $\lambda=\lambda_j$.
The dual frame $\{\tilde\psi_{\xi}\}_{\xi\in\cX}$ is defined by
\begin{equation}\label{def-dual}
\tilde\psi_{\xi}(x):=c_{\varepsilon}|A_{\xi}|^{1/2}T_{\lambda_{j}}\big[\Gamma_{\lambda_{j}}(\cdot,\xi)\big](x),
\quad \xi\in\cX_{j},\; j\in\bZ,
\quad c_{\varepsilon}:=(1+\varepsilon)^{-1}.
\end{equation}

In the next theorem we record the main properties of
$\{\psi_{\xi}\}_{\xi\in\cX}$ and $\{\tilde\psi_{\xi}\}_{\xi\in\cX}$.

%%%%%%%%%%%%% Theorem

\begin{theorem}\label{thm:frames}
$(a)$ {\rm Representation:}
For any $f\in\cS'/\PP$,
\begin{equation}\label{rep-L2}
f = \sum_{\xi \in \XX} \langle f, \tilde\psi_{\xi}\rangle \psi_{\xi}
= \sum_{\xi \in \XX} \langle f, \psi_{\xi}\rangle \tilde\psi_{\xi}
\quad\hbox{in}\;\; \cS'/\PP.
\end{equation}

$(b)$ {\rm Space localization:} For any $0<\kaph<\kap/2$, $m \in \bZ$, and any $\xi\in\XX_j$, $j\in\bZ$,
\begin{equation}\label{prop-tpsi-1}
|L^m\psi_{\xi} (x)|, |L^m\tilde\psi_{\xi} (x)|
\le c_m b^{2jm}|B(\xi, b^{-j})|^{-1/2} \exp\big\{-\kaph(b^j\rho(x, \xi))^\beta\big\}.
\end{equation}

$(c)$ {\rm Spectral localization:}
$\psi_{\xi}\in \Sigma_{[b^{j-1}, b^{j+1}]}^p$
and
$\tilde\psi_{\xi}\in \Sigma_{[b^{j-2}, b^{j+2}]}^p$ for $\xi\in \XX_j$, $j\in\bZ$,
$0<p\le\infty$.

$(d)$ {\rm Norms:} For any $\xi\in\XX_j$, $j\in\bZ$,
\begin{equation}\label{frame-norms}
\|\psi_{\xi}\|_p \sim
\|\tilde\psi_{\xi}\|_p \sim |B(\xi, b^{-j})|^{\frac 1p-\frac 12}
\quad\hbox{for} \;\; 0< p \le \infty.
\end{equation}

$(e)$ {\rm Frame:}
The system $\{\tilde\psi_{\xi}\}$ as well as $\{\psi_{\xi}\}$
is a frame for $L^2$, namely, there exists a constant $c>0$ such that
\begin{equation}\label{frame-tpsi}
 c^{-1}\|f\|_2^2 \le
\sum_{\xi \in \XX} |\langle f, \tilde\psi_{\xi}\rangle|^2
\leq  c\| f\|_2^2,
\quad \forall f\in L^2.
\end{equation}
\end{theorem}

\section{Homogeneous Besov and Triebel-Lizorkin spaces}\label{sec:hom-spaces}

Homogeneous Besov and Triebel-Lizorkin spaces in the setting of this article are developed in \cite{GKKP}.
Next, we recall the definition of these spaces and and some basic results on them.

\smallskip

\noindent
{\bf Definition of homogeneous Besov and Triebel-Lizorkin spaces.}
To deal with possible anisotropic geometries we introduced in \cite{GKKP} two types of
homogeneous Besov (B) and Triebel-Lizorkin (F) spaces:

(i)  Classical homogeneous B-spaces  $\BB_{pq}^{s}=\BB_{pq}^{s}(L)$ and F-spaces $\FF_{pq}^{s}=\FF_{pq}^{s}(L)$,
and

(ii)  Nonclassical homogeneous B-spaces $\tBB_{pq}^{s}=\tBB_{pq}^{s}(L)$ and F-spaces $\tFF_{pq}^{s}=\tFF_{pq}^{s}(L)$.

Let the function
$\varphi\in C^\infty(\bR_+)$
satisfy
\begin{equation}\label{cond_varphi}
\supp \varphi \subset   [1/2, 2], \;\;
|\varphi(\lambda)| \ge c>0 \;\hbox{ for } \lambda\in [2^{-3/4}, 2^{3/4}].
\end{equation}
Then
$\sum_{j\in \bZ} |\varphi(2^{-j}\lambda)| \ge c >0$
for $\lambda \in  \R_+$.
Set $\varphi_j(\lambda):= \varphi(2^{-j}\lambda)$ for $j\in \bZ$.

%%%%%%%%%%%% Definition

\begin{definition}\label{def-B-spaces}
Let $s \in \R$ and $0<p,q \le \infty$.

$(i)$ The Besov space  $\BB_{pq}^{s}=\BB_{pq}^{s}(L)$
is defined as the set of all $f \in \cS'/\PP$ such that
\begin{equation}\label{def-Besov-space1}
\|f\|_{\BB_{pq}^{s}} :=
\Big(\sum_{j\in\bZ} \Big(2^{js}
\|\varphi_j(\sqrt{L}) f(\cdot)\|_{L^p}
\Big)^q\Big)^{1/q} <\infty.
\end{equation}

$(ii)$ The Besov space  $\tBB_{pq}^{s}= \tBB_{pq}^{s}(L)$ is defined as the set
of all $f \in \cS'/\PP$ such that
\begin{equation}\label{def-Besov-space2}
\|f\|_{\tBB_{pq}^{s}} :=
\Big(\sum_{j\in\bZ} \Big(
\| |B(\cdot, 2^{-j})|^{-s/d}
\varphi_j(\sqrt{L}) f(\cdot)\|_{L^p}
\Big)^q\Big)^{1/q} <\infty.
\end{equation}
%Above the $\ell^q$-norm is replaced by the sup-norm if $q=\infty$.
\end{definition}

%%%%%%%%% Definition

\begin{definition}\label{def-F-spaces}
Let $s\in \R$, $0<p<\infty$, and $0<q \le \infty$.

$(a)$
The Triebel-Lizorkin space  $\FF_{pq}^{s}= \FF_{pq}^{s}(L)$
is defined as the set of all $f \in \cS'/\PP$
such that
\begin{equation}\label{def-F-space1}
\|f\|_{\FF_{pq}^s} :=
\Big\|\Big(\sum_{j\in \bZ} \Big(
2^{js}
|\varphi_j(\sqrt L) f(\cdot)|
\Big)^q\Big)^{1/q}\Big\|_{L^p} <\infty.
\end{equation}

$(b)$
The Triebel-Lizorkin space  $\tFF_{pq}^{s}= \tFF_{pq}^{s}(L)$
is defined as the set of all $f \in \cS'/\PP$
such that
\begin{equation}\label{def-F-space2}
\|f\|_{\tFF_{pq}^s} :=
\Big\|\Big(\sum_{j\in \bZ} \Big(
|B(\cdot, 2^{-j})|^{-s/d}
|\varphi_j(\sqrt L) f(\cdot)|
\Big)^q\Big)^{1/q}\Big\|_{L^p} <\infty.
\end{equation}
\end{definition}
Above in both definitions the $\ell^q$-norm is replaced by the sup-norm if $q=\infty$.

\smallskip

Several remarks regarding the homogeneous Besov and Triebel-Lizorkin spaces are in order.

(1) The above definitions of
the spaces $\BB_{pq}^s$, $\tBB_{pq}^s$, $\FF_{pq}^s$, and $\tFF_{pq}^s$ are independent of
the particular selection of the function $\varphi\in C^\infty(\R_+)$ obeying $(\ref{cond_varphi})$.

(2) In the definitions of the $\BB_{pq}^s$, $\tBB_{pq}^s$, $\FF_{pq}^s$, and $\tFF_{pq}^s$
spaces above the role of the constant $2$ can be played by an arbitrary $\beta>1$,
then e.g. $2^{js}$ in $(\ref{def-Besov-space1})$ and $(\ref{def-F-space1})$ will be replaced by $\beta^{js}$.
and then the resulting norms are equivalent
to the ones from Definitions~\ref{def-B-spaces} and \ref{def-F-spaces}.

(3) The space $\cS_\infty$ is continuously embedded in each of the spaces
$\BB_{pq}^s$, $\tBB_{pq}^s$, $\FF_{pq}^s$, and $\tFF_{pq}^s$ and each of the last
is continuously embedded in $\cS'/\PP=\cS_\infty '$.

(4) Each of the spaces $\BB_{pq}^s$, $\tBB_{pq}^s$, $\FF_{pq}^s$, and $\tFF_{pq}^s$
is continuously embedded in $\cS'/\PP$, that is,
there exist constants $m\ge 0$ and $c>0$, depending on $s, p, q$, such that
\begin{equation}\label{cont-embed-B}
|\langle f, \phi\rangle| \le c\|f\|_{\BB_{pq}^s}\cP_m^\star(\phi),
\quad \forall f\in \BB_{pq}^s, \;\; \forall \phi\in \cS_\infty,
\end{equation}
and similar inequalities hold for $\tBB_{pq}^s$, $\FF_{pq}^s$, and $\tFF_{pq}^s$.

(5)  By a standard argument the above assertion readily implies that
the spaces $\BB_{pq}^s$, $\tBB_{pq}^s$, $\FF_{pq}^s$, and $\tFF_{pq}^s$ are complete
and hence they are quasi-Banach spaces $($Banach spaces if $p, q \ge 1$$)$.

\subsection*{Frame decomposition of homogeneous Besov and Triebel-Lizorkin spaces}

One of the main results in \cite{GKKP} asserts that the homogeneous Besov and Triebel-Lizorkin spaces
in the setting of this article
can be characterized in terms of respective sequence norms of the frame coefficients of distributions,
using the frames
$\{\psi_\xi\}_{\xi\in\cX}$, $\{\tilde\psi_\xi\}_{\xi\in\cX}$
from \S\ref{sec:frames}.
As is \S\ref{sec:frames}
$\cX:= \cup_{j\in \bZ} \cX_j$ will denote the sets of the centers of the frame elements and
$\{A_\xi\}_{\xi\in\cX_j}$ will be the associated partitions of $M$.

In the following we first recall the definition of the homogeneous sequence spaces
$\bb_{pq}^s$, $\tbb_{pq}^s$, and $\ff_{pq}^s$, $\tff_{pq}^s$,
associated with the $\BB$- and $\FF$-spaces,
and then give the frame characterization of the $\BB$- and $\FF$-spaces.

%%%%%%%%%%%%%%%% Definition

\begin{definition}\label{def:b-spaces}
Let $s\in \R$ and $0<p,q \le \infty$.

$(a)$
$\bb_{pq}^s$
is defined as the space of all complex-valued sequences
$a:=\{a_{\xi}\}_{\xi\in \cX}$ such that
\begin{equation}\label{def-tilde-berpq}
\|a\|_{\bb_{pq}^s}
:=\Bigl(\sum_{j\in \bZ}b^{jsq}
\Bigl[\sum_{\xi\in \cX_j}\Big(|B(\xi, b^{-j})|^{1/p-1/2}|a_\xi|\Big)^p
\Bigr]^{q/p}\Bigr)^{1/q} <\infty.
\end{equation}

$(b)$
$\tbb_{pq}^s$
is defined as the space of all complex-valued sequences
$a:=\{a_{\xi}\}_{\xi\in \cX}$ such that
\begin{equation}\label{def-berpq}
\|a\|_{\tbb_{pq}^s}
:=\Bigl(\sum_{j\in \bZ}
\Bigl[\sum_{\xi\in \cX_j}\Big(|B(\xi, b^{-j})|^{-s/d+1/p-1/2}|a_\xi|\Big)^p
\Bigr]^{q/p}\Bigr)^{1/q} <\infty.
\end{equation}
%Above as usual the $\ell^p$ or $\ell^q$ norm is replaced by the $\sup$-norm if
%$p=\infty$ or $q=\infty$.
\end{definition}

%%%%%%%%%%%%%% Definition

\begin{definition}\label{def-f-spaces}
Suppose $s\in \R$, $0<p<\infty$, and $0<q\le\infty$.

$(a)$
$\ff_{pq}^s$
is defined as the space of all complex-valued sequences
$a:=\{a_{\xi}\}_{\xi\in \cX}$ such that
\begin{equation}\label{def-f-space}
\|a\|_{\ff_{pq}^s}
:=\Big\|\Big(\sum_{j\in\bZ}b^{jsq}\sum_{\xi \in \cX_j}
\big[|a_{\xi}|\tONE_{A_\xi}(\cdot)\big]^q
\Big)^{1/q}
\Big\|_{L^p} <\infty.
\end{equation}

$(b)$
$\tff_{pq}^s$
is defined as the space of all complex-valued sequences
$a:=\{a_{\xi}\}_{\xi\in \cX}$ such that
\begin{equation}\label{def-tf-space}
\|a\|_{\tff_{pq}^s}
:=\Big\|\Big(\sum_{\xi \in \cX}
\big[|A_\xi|^{-s/d}|a_{\xi}|\tONE_{A_\xi}(\cdot)\big]^q
\Big)^{1/q}
\Big\|_{L^p} <\infty.
\end{equation}
Here $\tONE_{A_\xi}:=|A_\xi|^{-1/2}\ONE_{A_\xi}$
with $\ONE_{A_\xi}$ being the characteristic function of $A_\xi$.
\end{definition}

Above as usual the $\ell^p$ or $\ell^q$ norm is replaced by the $\sup$-norm if
$p=\infty$ or $q=\infty$.

The ``analysis" and ``synthesis" operators are defined by
\begin{equation}\label{anal_synth_oprts2}
S_{\tilde\psi}: f\rightarrow \{\langle f, \tilde\psi_\xi\rangle\}_{\xi \in \cX}
\quad\text{and}\quad
T_{\psi}: \{a_\xi\}_{\xi \in \cX}\rightarrow \sum_{\xi\in \cX}a_\xi\psi_\xi.
\end{equation}
Here the roles of $\{\psi_\xi\}$ and $\{\tilde\psi_\xi\}$ can be interchanged.

%%%%%%%%%%%%%%%% Theorem

\begin{theorem}\label{thm:B-character}
Let $s \in \R$ and  $0< p,q\le \infty$.
%
%\noindent
$(a)$
The operators
$S_{\tilde\psi}: \BB_{pq}^s \rightarrow  \bb_{pq}^s$ and
$T_{\psi}: \bb_{pq}^s \rightarrow \BB_{pq}^s$
are bounded and $T_{\psi}\circ S_{\tilde\psi}=Id$ on $\BB_{pq}^s$.
Consequently,  for $f\in \cS'/\PP$ we have $f\in \BB_{pq}^s$ if and
only if $\{\langle f, \tilde\psi_\xi\rangle\}_{\xi \in \cX}\in \bb_{pq}^s$.
Moreover, if $f\in \BB_{pq}^s$, then
$
\|f\|_{\BB_{pq}^s} \sim  \|\{\langle f,\tilde\psi_\xi\rangle\}\|_{\bb_{pq}^s}
$
and % under the reverse doubling condition $(\ref{rdc})$
\begin{equation}\label{Bnorm-equivalence-1}
\|f\|_{\BB_{pq}^s} % \sim  \|\{\langle f,\psi_\xi\rangle\}\|_{\tb_{pq}^s}
\sim \Big(\sum_{j\in\bZ} b^{jsq}\Bigl[\sum_{\xi\in \cX_j}
\|\langle f,\tilde\psi_{\xi}\rangle\psi_{\xi}\|_p^p\Bigr]^{q/p}\Bigr)^{1/q}.
\end{equation}

\noindent
$(b)$
The operators
$S_{\tilde\psi}: \tBB_{pq}^s \rightarrow  \tbb_{pq}^s$ and
$T_{\psi}: \tbb_{pq}^s \rightarrow \tBB_{pq}^s$
are bounded and $T_{\psi}\circ S_{\tilde\psi}=Id$ on $\tBB_{pq}^s$.
%Consequently,  for $f\in \cD'$ we have
%$f\in \tB_{pq}^s$ if and only if $\{\langle f, \psi_\xi\rangle\}_{\xi \in \cX}\in \tbb_{pq}^s$.
Hence,
$f\in \tBB_{pq}^s \Longleftrightarrow \{\langle f, \tilde\psi_\xi\rangle\}_{\xi \in \cX}\in \tbb_{pq}^s$.
Furthermore, if $f\in \tBB_{pq}^s$, then
$\|f\|_{\tBB_{pq}^s} \sim  \|\{\langle f,\tilde\psi_\xi\rangle\}\|_{\tbb_{pq}^s}
$
and % under the reverse doubling condition $(\ref{rdc})$
\begin{equation}\label{tBnorm-equiv-1}
\|f\|_{\tBB_{pq}^s} % \sim  \|\{\langle f,\psi_\xi\rangle\}\|_{\tbb_{pq}^s}
\sim \Big(\sum_{j\in\bZ} \Bigl[\sum_{\xi\in \cX_j}
\Big(|B(\xi, b^{-j})|^{-s/d}
\|\langle f,\tilde\psi_{\xi}\rangle\psi_{\xi}\|_{p}\Big)^p\Bigr]^{q/p}\Bigr)^{1/q}.
\end{equation}
Above in $(a)$ and $(b)$ the roles of $\{\psi_\xi\}$ and $\{\tilde\psi_\xi\}$ can be interchanged.
\end{theorem}

%%%%%%%%%%%%% Theorem

\begin{theorem}\label{thm:F-character}
Let $s\in \R$, $0< p< \infty$ and $0<q\le \infty$.
$(a)$
The operators
$S_{\tilde\psi}:\FF_{pq}^s \rightarrow \ff_{pq}^s$ and $T_{\psi}:\ff_{pq}^s \rightarrow \FF_{pq}^s$
are bounded and $T_{\tilde\psi}\circ S_\psi=Id$ on $\FF_{pq}^s$.
Consequently, $f\in \FF_{pq}^s$ if and only if
$\{\langle f, \tilde\psi_\xi\rangle\}_{\xi \in \cX}\in \ff_{pq}^s$,
and if $f\in \FF_{pq}^s$, then
$\|f\|_{\FF_{pq}^s} \sim  \|\{\langle f, \tilde\psi_\xi\rangle\}\|_{\ff_{pq}^s}$.
Furthermore, %under the reverse doubling condition $(\ref{reverse-doubling})$
\begin{equation}\label{F-norm-equivalence-1}
\|f\|_{\FF_{pq}^s}
%\sim  \|\{\langle f, \tilde\psi_\xi\rangle\}\|_{\tf_{pq}^s}
%
\sim \Big\|\Big(
\sum_{j\in\bZ}b^{jsq}\sum_{\xi\in \cX_j}
\big[|\langle f, \tilde\psi_\xi \rangle||\psi_\xi(\cdot)|\big]^q
\Big)^{1/q}\Big\|_{L^p}.
\end{equation}

\noindent
$(b)$
The operators
$S_{\tilde\psi}:\tFF_{pq}^s \rightarrow \tff_{pq}^s$ and $T_{\psi}: \tff_{pq}^s \rightarrow \tFF_{pq}^s$
are bounded and $T_{\tilde\psi}\circ S_\psi=Id$ on $\tFF_{pq}^s$.
Hence, $f\in \tFF_{pq}^s$ if and only if
$\{\langle f, \tilde\psi_\xi\rangle\}_{\xi \in \cX}\in \tff_{pq}^s$,
and if $f\in \tFF_{pq}^s$, then
$\|f\|_{\tFF_{pq}^s} \sim  \|\{\langle f, \tilde\psi_\xi\rangle\}\|_{\tff_{pq}^s}$.
Furthermore, % under the reverse doubling condition $(\ref{reverse-doubling})$
\begin{equation}\label{F-norm-equivalence-2}
\|f\|_{\tFF_{pq}^s}
%\sim  \|\{\langle f, \tilde\psi_\xi\rangle\}\|_{\tf_{pq}^s}
%
\sim \Big\|\Big(
\sum_{\xi\in \cX}
\big[|B(\xi, b^{-j})|^{-s/d}|\langle f, \tilde\psi_\xi \rangle||\psi_\xi(\cdot)|\big]^q
\Big)^{1/q}\Big\|_{L^p}.
\end{equation}
As before the roles of $\psi_\xi$ and $\tilde\psi_\xi$ can be interchanged.
\end{theorem}

%\begin{remark}\label{rem:heat-ker-decom}
%Observe that heat kernel characterizations of the homogeneous Besov and Triebel-Lizorkin spaces
%similar to the ones in the inhomogeneous case $($\cite{KP}, Theorems~6.7 and 7.5$)$ are valid
%with almost identical proofs.
%We shall not elaborate on these results here.
%\end{remark}

\section{Almost diagonal operators}\label{sec:almost-diag}

As in the classical case on $\R^n$ (see \cite{FJ2}),
we shall introduce almost diagonal operators acting on the sequence homogeneous Besov and Triebel-Lizorkin spaces.
In fact, our definition for almost diagonal operators is a refinement of the one given in \cite{DKKP}
in the inhomogeneous case.

As in the definition of the sequence Besov and Triebel-Lizorkin spaces in \S\ref{sec:hom-spaces}
\begin{equation}\label{cX}
\cX:=\cup_{j\in\bZ}\cX_j
\end{equation}
will be the set of centers of the frame elements
$\psi_\xi$ and $\tilde\psi_\xi$, $\xi\in\cX$,
and
\begin{equation}\label{A-xi}
\{A_\xi\}_{\xi\in\cX_j}, \quad \cup_{\xi\in\cX_j}A_\xi =M, \quad j\in\bZ,
\end{equation}
%$\{A_\xi\}_{\xi\in\cX_j}$
will denote the companion disjoint partitions of $M$.
%Recall that $\cX_j$ is a maximal $\delta_j-$net on $M$ with $\delta_j=\gamma b^{-j-2}$
%and $B(\xi, \delta_j/2)\subset A_\xi \subset B(\xi, \delta_j)$ for $\xi\in\cX_j$.

\begin{remark}\label{rem:X-A-xi}
As indicated above the sets $\cX_j$, $j\in\bZ$, $\cX:=\cup_{j\in\bZ}\cX_j$ are from the definition of the frames
$\{\psi_\xi\}_{\xi\in\cX}$, $\{\tilde\psi_\xi\}_{\xi\in\cX}$ in \S\ref{sec:frames}.
Note that once the constant $\gamma>0$ is fixed $($see \S\ref{sec:frames}$)$,
each set $\cX_j$ is an arbitrary maximal $\delta_j-$net on $M$ with $\delta_j=\gamma b^{-j-2}$.
Therefore, there is no uniqueness in the selection of these sets.
In what follows we shall assume that once selected these sets are fixed, once and for all.

A similar observation is valid about the sets $A_\xi$, $\xi\in\cX_j$, from $(\ref{A-xi})$.
They form an arbitrary disjoint partition of $M$ consisting of measurable sets such that
\begin{equation}\label{B-A-B}
B(\xi, \delta_j/2)\subset A_\xi \subset B(\xi, \delta_j), \quad \xi\in\cX_j.
\end{equation}
Again there is no uniqueness.
We shall consider them fixed, once and for all.
\end{remark}

It will be convenient to us to use the notation %(see also (\ref{B-xi}))
\begin{equation}\label{def-ell}
\ell(\xi):=b^{-j}
\quad\hbox{and}\quad
\AA_\xi:= B(\xi, \delta_j)
\quad\mbox{for} \quad\xi\in\cX_j, \; j\in\bZ.
\end{equation}
Here $b>1$ is the constant from the construction of the frames in \S\ref{sec:frames}.
Observe that by (\ref{doubling}) and \eqref{B-A-B} it follows that  $|A_\xi|\sim |\AA_\xi|\sim |B(\xi, \ell(\xi))|$.

%%%%%%%%% Definition

\begin{definition}\label{def:almost-diagonal}
Let $\Aa$ be a linear operator acting on one of the spaces
$\tbb^s_{pq}$, $\tff^s_{pq}$, $\bb^s_{pq}$, $\ff^s_{pq}$,
with associated matrix $(a_{\xi\eta})_{\xi, \eta\in \cX}$.
Let also $\cJ:= d/\min \{1, p\}$ for $\bb^s_{pq}$, $\tbb^s_{pq}$ and
$\cJ:= d/\min \{1, p, q\}$ for $\ff^s_{pq}$, $\tff^s_{pq}$.
We say that the operator $\Aa$ is {\it almost diagonal} on the respective $b$- or $f$-space
if there exists $\delta > 0$ such that
$$
\sup_{\xi,\eta\in \cX}\frac{|a_{\xi \eta}|}{\omega_{\xi\eta}(\delta)}<\infty,
$$
where
\begin{align}
\omega_{\xi\eta}(\delta)
&:=
\biggl(\frac{\ell(\xi)}{\ell(\eta)}\biggr)^{s}\biggl(\frac{|\AA_\xi|}{|\AA_\eta|}\biggr)^{1/2} \biggl(1+\frac{\rho(\xi,\eta)}
{\max \{\ell(\xi),\ell(\eta)\}}\biggr)^{-\cJ-\delta} \label{def-omega-d-1}
\\
&\times \min
\biggl\{\biggl(\frac{\ell(\xi)}{\ell(\eta)}\biggr)^{\delta},
     \biggl(\frac{\ell(\eta)}{\ell(\xi)}\biggr)^{\cJ+\delta}\biggr\} \notag
\end{align}
in the case of the spaces $\bb^s_{pq}$ or $\ff^s_{pq}$, and
\begin{align}
\omega_{\xi\eta}(\delta)
&:=
\biggl(\frac{|\AA_\xi|}{|\AA_\eta|}\biggr)^{s/d+1/2} \biggl(1+\frac{\rho(\xi,\eta)}
{\max \{\ell(\xi),\ell(\eta)\}}\biggr)^{-\cJ-\delta} \label{def-omega-d-2}
\\
&\times \min
\biggl\{\biggl(\frac{\ell(\xi)}{\ell(\eta)}\biggr)^{\delta},
     \biggl(\frac{\ell(\eta)}{\ell(\xi)}\biggr)^{\cJ+\delta}\biggr\} \notag
\end{align}
in the case of $\tbb^s_{pq}$ or $\tff^s_{pq}$.
\end{definition}

\subsection{Boundedness of almost diagonal operators}\label{subsec:bound-alm-diag-op}

We next show that the almost diagonal operators are bounded on
$\bb^s_{pq}$, $\tbb^s_{pq}$, $\ff^s_{pq}$, or $\tff^s_{pq}$, respectively.
More precisely, with the notation
\begin{equation}\label{AlmDiag-Norm}
\|\Aa\|_\delta:=\sup_{\xi,\eta\in \cX}\frac{|a_{\xi \eta}|}{\omega_{\xi\eta}(\delta)}
\end{equation}
the following result holds:

%%%%%%%% Theorem

\begin{theorem}\label{thm:AlmDiag}
Suppose  $s\in\R$, $0<q\le \infty$, and $0<p<\infty$ $(0<p\le\infty$ in the case of $\bb$-spaces$)$
and let $\|\Aa\|_\delta<\infty $ $($in the sense of Definition~\ref{def:almost-diagonal}$)$
for some $\delta>0$.
Then there exists a constant $c>0$ such that for any sequence
$h:=\{h_\xi\}_{\xi\in \cX}\in \bb^s_{pq}$
\begin{equation}\label{b-AlmDiag}
\|\Aa h\|_{\bb^s_{pq}}\le c\|\Aa\|_\delta \|h\|_{\bb^s_{pq}},
\end{equation}
and the same holds true with $\bb^s_{pq}$ replaced by
$\tbb^s_{pq}$, $\ff^s_{pq}$, or $\tff^s_{pq}$.
\end{theorem}

The proof of this theorem will be carried out similarly as the proof of Theorem~3.3
in \cite{FJ2} or Theorem~4.4 in \cite{DKKP}. We place it in the appendix.

\subsection{The algebra of almost diagonal operators}\label{algebra-alm-diag-op}

Let $s\in\R$, $0<p<\infty$ $(0<p\le\infty$ in the case of $\bb$-spaces$)$, and $0<q\le \infty$ be fixed.
We denote by $\ad^{s}_{pq}$ %or simply by $\ad$
the class of almost diagonal operators
on $\bb^{s}_{pq}$, $\ff^{s}_{pq}$, $\tbb^s_{pq}$, or $\tff^s_{pq}$, equipped with the norm
\begin{equation}\label{def:operator-norm}
\|\Aa\|_{\ad^{s}_{pq}}:=\inf_{\varepsilon>0}\|\Aa\|_{\varepsilon},
\end{equation}
where
$\|\Aa\|_{\varepsilon}:=\sup_{\xi,\eta}|a_{\xi\eta}|/\omega_{\xi\eta}(\varepsilon)$,
see (\ref{def-omega-d-1})-(\ref{AlmDiag-Norm}).

This is a nondecreasing function of $\varepsilon$ and, therefore, $\|\Aa\|_{\ad^{s}_{pq}}$ is indeed a norm.

Our next goal is to prove that the class $\ad^{s}_{pq}$ is an algebra under composition.

\begin{theorem}\label{thm:algebra}
Let $s\in\R$, $0<p<\infty$, and $0<q\le\infty$ $(0<p\le\infty$ in the case of $\bb$-spaces$)$.
Then for the respective $\bb$- and $\ff$-spaces the following claims hold:

$(i)$ If $A,B\in\ad^{s}_{pq}$, then $A\circ B\in\ad^{s}_{pq}$.

$(ii)$ For any $\varepsilon >0$ there exists $\delta>0$ such that
if $A\in\ad^{s}_{pq}$ and $\|I-A\|_{\varepsilon}<\delta$, then
$A$ is invertible and $A^{-1}\in\ad^{s}_{pq}$.
\end{theorem}

We shall carry out the proof of this theorem in the spirit of the proof of Theorem~9.1 in \cite{FJ2}.
We need some additional notation.
For any $\beta,\gamma>0$ and $\xi,\eta\in\cX$  we set:

$(i)$ in the case of $\bb^s_{pq}$ and  $\ff^s_{pq}$,
\begin{align}
\omega_{\xi\eta}(\beta,\gamma)
&:=
\biggl(\frac{\ell(\xi)}{\ell(\eta)}\biggr)^{s}
\biggl(\frac{|\AA_\xi|}{|\AA_\eta|}\biggr)^{1/2} \biggl(1+\frac{\rho(\xi,\eta)}
{\max \{\ell(\xi),\ell(\eta)\}}\biggr)^{-\cJ-\beta} \label{def-omega-detail-1}
\\
&\times \min
\biggl\{\biggl(\frac{\ell(\xi)}{\ell(\eta)}\biggr)^{\gamma},
     \biggl(\frac{\ell(\eta)}{\ell(\xi)}\biggr)^{\cJ+\gamma}\biggr\}, \notag
\end{align}

$(ii)$ in the case of $\tbb^s_{pq}$ and $\tff^s_{pq}$,
\begin{align}
\omega_{\xi\eta} (\beta,\gamma)
&:=
\biggl(\frac{|\AA_\xi|}{|\AA_\eta|}\biggr)^{s/d+1/2} \biggl(1+\frac{\rho(\xi,\eta)}
{\max \{\ell(\xi),\ell(\eta)\}}\biggr)^{-\cJ-\beta}  \label{def-omega-detail-2}
\\
&\times \min
\biggl\{\biggl(\frac{\ell(\xi)}{\ell(\eta)}\biggr)^{\gamma},
     \biggl(\frac{\ell(\eta)}{\ell(\xi)}\biggr)^{\cJ+\gamma}\biggr\}. \notag
\end{align}

Furthermore, given $\beta,\gamma_1,\gamma_2>0$ and $\xi,\eta\in\cX$, we set
\begin{equation}\label{def-W}
W_{\xi\eta}(\beta,\gamma_1,\gamma_2)
:=\sum_{\zeta\in\cX}\omega_{\xi\zeta}(\beta,\gamma_1)\omega_{\zeta\eta}(\beta,\gamma_2).
\end{equation}
The following lemma will be instrumental in the proof of Theorem~\ref{thm:algebra}.

\begin{lemma}\label{lem:algebra4}
Let $\beta,\gamma_1,\gamma_2>0$ be such that
$\gamma_1\neq\gamma_2$ and $\beta<\gamma_1+\gamma_2$. Then for any $\xi,\eta\in\cX$
\begin{equation}\label{algd}
W_{\xi\eta}(\beta,\gamma_1,\gamma_2)\le c \omega_{\xi\eta}(\beta,\gamma_1\wedge\gamma_2),
\end{equation}
where the constant $c>0$ depends on $\beta$, $\gamma_1$, $\gamma_2$, $\cJ$,
and the constant $c_0$ from $(\ref{doubling-0})$.
\end{lemma}

\begin{proof}
We shall only carry out the proof for the spaces $\bb^{s}_{pq}$ and $\ff^{s}_{pq}$.
The proof for the spaces $\tbb^s_{pq}$ and $\tff^s_{pq}$ is similar.

Assume $\ell(\xi)\le \ell(\eta)$. Clearly, by (\ref{def-omega-detail-1}) - (\ref{def-W})
\begin{align*}
W_{\xi\eta}(\beta,\gamma_1,\gamma_2)
&=\sum_{\zeta\in\cX}
\biggl(\frac{\ell(\xi)}{\ell(\eta)}\biggr)^{s}\biggl(\frac{|\AA_\xi|}{|\AA_\eta|}\biggr)^{1/2}
\\
&\times \biggl(1+\frac{\rho(\xi,\zeta)}
{\max \{\ell(\xi),\ell(\zeta)\}}\biggr)^{-\cJ-\beta}
\biggl(1+\frac{\rho(\zeta,\eta)}
{\max \{\ell(\zeta),\ell(\eta)\}}\biggr)^{-\cJ-\beta}
\\
&\times \min\biggl\{\biggl(\frac{\ell(\xi)}{\ell(\zeta)}\biggr)^{\gamma_1},
     \biggl(\frac{\ell(\zeta)}{\ell(\xi)}\biggr)^{\cJ+\gamma_1}\biggr\}
\min\biggl\{\biggl(\frac{\ell(\zeta)}{\ell(\eta)}\biggr)^{\gamma_2},
     \biggl(\frac{\ell(\eta)}{\ell(\zeta)}\biggr)^{\cJ+\gamma_2}\biggr\}.
\end{align*}
Let $\ell(\xi)=b^{-j}$, $\ell(\eta)=b^{-\nu}$, $\nu\le j$.
Then
$$
W_{\xi\eta}(\beta,\gamma_1,\gamma_2)
=\biggl(\frac{\ell(\xi)}{\ell(\eta)}\biggr)^{s}\biggl(\frac{|\AA_\xi|}{|\AA_\eta|}\biggr)^{1/2}
\left(\Sigma_1+\Sigma_2+\Sigma_3\right),
$$
where
\begin{align*}
\Sigma_1 =\sum_{m=j+1}^{\infty}\sum_{\zeta\in\cX_m}
&\biggl(1+\frac{\rho(\xi,\zeta)}
{\max \{\ell(\xi),\ell(\zeta)\}}\biggr)^{-\cJ-\beta}
\biggl(1+\frac{\rho(\zeta,\eta)}
{\max \{\ell(\zeta),\ell(\eta)\}}\biggr)^{-\cJ-\beta}
\\
&\times
\min \biggl\{\biggl(\frac{\ell(\xi)}{\ell(\zeta)}\biggr)^{\gamma_1},
     \biggl(\frac{\ell(\zeta)}{\ell(\xi)}\biggr)^{\cJ+\gamma_1}\biggr\}
\min \biggl\{\biggl(\frac{\ell(\zeta)}{\ell(\eta)}\biggr)^{\gamma_2},
     \biggl(\frac{\ell(\eta)}{\ell(\zeta)}\biggr)^{\cJ+\gamma_2}\biggr\}
\end{align*}
and $\Sigma_2=\sum\limits_{m=\nu}^{j}\sum\limits_{\zeta\in\cX_m}\cdots$,
$\Sigma_3=\sum\limits_{m=-\infty}^{\nu-1}\sum\limits_{\zeta\in\cX_m} \cdots$
for the same quantity.
Applying (\ref{est-discr-sum}) we get
\begin{align*}\Sigma_1 &=\sum_{m=j+1}^{\infty}\sum_{\zeta\in\cX_m}
\frac{b^{(\cJ+\gamma_1)(j-m)}b^{\gamma_2(\nu-m)}}{\left(1+b^{j}\rho(\xi,\zeta)\right)^{\cJ+\beta}
\left(1+b^{\nu}\rho(\zeta,\eta)\right)^{\cJ+\beta}}
\\
&\le c\left(1+b^{\nu}\rho(\xi,\eta)\right)^{-\cJ-\beta}
\sum_{m=j+1}^{\infty}b^{(\cJ+\gamma_1)(j-m)}b^{\gamma_2(\nu-m)}b^{d(m-j)}
\\
&\le c\left(1+b^{\nu}\rho(\xi,\eta)\right)^{-\cJ-\beta}b^{\gamma_2(\nu-j)},
\end{align*}
where for the last inequality we used that $\gamma_1+\gamma_2>0$ and $\cJ\ge d$.

We use again (\ref{est-discr-sum}) to obtain
\begin{align*}\Sigma_2 &=\sum_{m=\nu}^{j}\sum_{\zeta\in\cX_m}
\frac{b^{\gamma_1(m-j)}b^{\gamma_2(\nu-m)}}{\left(1+b^{m}\rho(\xi,\zeta)\right)^{\cJ+\beta}
\left(1+b^{\nu}\rho(\zeta,\eta)\right)^{\cJ+\beta}}
\\
&\le c\left(1+b^{\nu}\rho(\xi,\eta)\right)^{-\cJ-\beta}
\sum_{m=\nu}^{j}b^{\gamma_1(m-j)}b^{\gamma_2(\nu-m)}.
\end{align*}
If $\gamma_1<\gamma_2$, then
$$
\sum_{m=\nu}^{j}b^{\gamma_1(m-j)}b^{\gamma_2(\nu-m)}
\le \sum_{m=\nu}^{\infty}b^{(\gamma_1-\gamma_2)m}b^{-\gamma_1 j}b^{\gamma_2 \nu}
\le cb^{\gamma_1(\nu-j)}.
$$
If $\gamma_1>\gamma_2$, then
$$
\sum\limits_{m=\nu}^{j}b^{\gamma_1(m-j)}b^{\gamma_2(\nu-m)}
\le \sum_{m=-\infty}^{j}b^{(\gamma_1-\gamma_2)m}b^{-\gamma_1 j}b^{\gamma_2 \nu}
\le cb^{\gamma_2(\nu-j)}.
$$
In both cases we get
$$
\Sigma_2 \le c\left(1+b^{\nu}\rho(\xi,\eta)\right)^{-\cJ-\beta}b^{(\gamma_1\wedge\gamma_2)(\nu-j)}.
$$

To estimate $\Sigma_3$ we use again (\ref{est-discr-sum}) and obtain

\begin{align*}\Sigma_3
&=\sum_{m=-\infty}^{\nu-1}\sum_{\zeta\in\cX_m}
\frac{b^{\gamma_1(m-j)}b^{(\cJ+\gamma_2)(m-\nu)}}{\left(1+b^{m}\rho(\xi,\zeta)\right)^{\cJ+\beta}
\left(1+b^{m}\rho(\zeta,\eta)\right)^{\cJ+\beta}}
\\
&\le c\sum\limits_{m=-\infty}^{\nu-1}\left(1+b^{m}\rho(\xi,\eta)\right)^{-\cJ-\beta}b^{\gamma_1(m-j)}b^{(\cJ+\gamma_2)(m-\nu)}.
\end{align*}
However, if $m<\nu$, then
$$
\left(1+b^{m}\rho(\xi,\eta)\right)^{-\cJ-\beta}
\le b^{(\cJ+\beta)(\nu-m)}\left(1+b^{\nu}\rho(\xi,\eta)\right)^{-\cJ-\beta}
$$
and due to $\gamma_1+\gamma_2>\beta$
\begin{align*}
&\sum_{m=-\infty}^{\nu-1}b^{\gamma_1(m-j)}b^{(\cJ+\gamma_2)(m-\nu)}b^{(\cJ+\beta)(\nu-m)}
\\
&=\sum_{m=-\infty}^{\nu-1}b^{(\gamma_1+\gamma_2-\beta)m}b^{-\gamma_1 j}b^{-\gamma_2 \nu}b^{\beta \nu}
\le cb^{\gamma_1(\nu-j)}.
\end{align*}
Hence
$$
\Sigma_3\le c\left(1+b^{\nu}\rho(\xi,\eta)\right)^{-\cJ-\beta}b^{\gamma_1(\nu-j)}.
$$

Putting the above estimates together we get for $\ell(\xi)\le \ell(\eta)$
$$
W_{\xi\eta}(\beta,\gamma_1,\gamma_2)\leq c
\biggl(\frac{\ell(\xi)}{\ell(\eta)}\biggr)^{s}\biggl(\frac{|\AA_\xi|}{|\AA_\eta|}\biggr)^{1/2}
\biggl(1+\frac{\rho(\xi,\eta)}
{\ell(\eta)}\biggr)^{-\cJ-\beta}
\biggl(\frac{\ell(\xi)}{\ell(\eta)}\biggr)^{\gamma_1\wedge\gamma_2}.
$$
Just in the same way one shows that if $\ell(\xi)>\ell(\eta)$, then
$$
W_{\xi\eta}(\beta,\gamma_1,\gamma_2)\leq c
\biggl(\frac{\ell(\xi)}{\ell(\eta)}\biggr)^{s}\biggl(\frac{|\AA_\xi|}{|\AA_\eta|}\biggr)^{1/2}
\biggl(1+\frac{\rho(\xi,\eta)}
{\ell(\xi)}\biggr)^{-\cJ-\beta}
\biggl(\frac{\ell(\eta)}{\ell(\xi)}\biggr)^{\cJ+\gamma_1\wedge\gamma_2}.
$$
The proof is complete.
\end{proof}

\smallskip

\noindent
{\em Proof of Theorem~\ref{thm:algebra}.}
(i)
Assume $A,B\in\ad^s_{pq}$ and let $\{a_{\xi\eta}\}_{\xi,\eta}$, $\{b_{\xi\eta}\}_{\xi,\eta}$
be their respective matrices.
Then there exist $\varepsilon_{a}$, $\varepsilon_{b}>0$ such that
$$
|a_{\xi\eta}|\le c\omega_{\xi\eta}(\varepsilon_a),
\quad
|b_{\xi\eta}|\le c\omega_{\xi\eta}(\varepsilon_b).
$$
Evidently, $\omega_{\xi\eta}(\varepsilon)$ is a nonincreasing function of $\varepsilon$
and hence we may assume that $\varepsilon_{a}>\varepsilon_{b}$.
Note that by the definitions it follows that
\begin{equation}\label{prop-omega}
\omega_{\xi\eta}(\beta,\beta)=\omega_{\xi\eta}(\beta),
\quad\hbox{and}\quad
\omega_{\xi\eta}(\varepsilon)\le \omega_{\xi\eta}(\beta,\gamma),
\;\;\hbox{if $0<\beta,\gamma\le\varepsilon$.}
\end{equation}
Denote by $\{c_{\xi\eta}\}_{\xi,\eta}$ the matrix of the composition $A\circ B$.
Applying Lemma~\ref{lem:algebra4} we get
\begin{align*}
|c_{\xi\eta}|&=\Big|\sum\limits_{\zeta\in\cX}a_{\xi\zeta}b_{\zeta\eta}\Big|
\le c\sum_{\zeta\in\cX}\omega_{\xi\zeta}(\varepsilon_a)\omega_{\zeta\eta}(\varepsilon_b)
\le c\sum_{\zeta\in\cX}
\omega_{\xi\zeta}(\varepsilon_b,\varepsilon_a)\omega_{\zeta\eta}(\varepsilon_b,\varepsilon_b)
\\
&=cW_{\xi\eta}(\varepsilon_b,\varepsilon_a,\varepsilon_b)\leq c^{*}\omega_{\xi\eta}(\varepsilon_b,\varepsilon_b)=c^{*}\omega_{\xi\eta}(\varepsilon_b)
\end{align*}
and the proof of (i) is complete.

(ii) Let $D:=I-A$ with matrix $\{d_{\xi\eta}\}_{\xi,\eta}$ and fix $\varepsilon >0$.
Assume $\|I-A\|_\varepsilon <\delta$ for some $\delta>0$,
implying
$|d_{\xi\eta}|\le \delta\omega_{\xi\eta}(\varepsilon)$.
Denote by $\{d_{\xi\eta}^{(n)}\}_{\xi,\eta}$ the matrix of $D^n$, $n\ge 1$.
Fix $0<\varepsilon_1 <\varepsilon$.
We claim that there exists a constant $c^*>1$, independent of $\delta$,  such that
for any $n\in\bN$
\begin{equation}\label{est-d-n}
|d_{\xi\eta}^{(n)}|\le (\delta c^{*})^n \omega_{\xi\eta}(\varepsilon_1),
\quad\forall \xi, \eta\in \cX.
\end{equation}
Indeed, from (\ref{prop-omega})
$\omega_{\xi\eta}(\varepsilon) \le \omega_{\xi\eta}(\varepsilon_1, \varepsilon)$
and just as in the proof of (i) we infer
$$
|d_{\xi\eta}^{(2)}| \le c^*\delta^2\omega_{\xi\eta}(\varepsilon_1) \le (c^*\delta)^2\omega_{\xi\eta}(\varepsilon_1).
$$
Suppose that (\ref{est-d-n}) holds for some $n\in\bN$.
Then from $|d_{\xi\eta}|\le \delta\omega_{\xi\eta}(\varepsilon)$ and (\ref{est-d-n})
it follows just as above that
$|d_{\xi\eta}^{(n+1)}| \le (\delta c^{*})^{n+1} \omega_{\xi\eta}(\varepsilon_1)$.
Therefore, (\ref{est-d-n}) holds for all $n\in\bN$.

Now, choose $\delta<1/c^{*}$. Then the Neumann series $\sum_{n\ge 0}D^n$ converges to
the operator $(I-D)^{-1}=A^{-1}$
and for its matrix $\{a^{-1}_{\xi\eta}\}_{\xi,\eta}$ it holds that
$$
|a^{-1}_{\xi\eta}|\le (1-\delta c^{*})^{-1}\omega_{\xi\eta}(\varepsilon_1).
$$
Therefore, $A^{-1}$ exists and is almost diagonal.
\qed

\subsection{Compactly supported frames}\label{subsec:construction}

Frames for inhomogeneous Besov and Triebel-Lizorkin spaces in the setting of this article
with compactly supported elements are developed in \cite{DKKP}.
We next show how this construction can be modified for homogeneous Besov and Triebel-Lizorkin spaces.

Let $\Psi$ be the compactly supported $C^\infty$ functions from the construction of Frame~\#~1, see (\ref{def:Psi}).
The first step is to construct a band limited function $\Theta$, which approximates
$\Psi$ in the specific sense given next.

%%%%%%%%%%%%% Proposition

\begin{proposition}\label{prop:constr-1}
For any $\eps>0$ and $\NN \ge\KK\ge 1$ there exists a function $\Theta\in C^\infty(\R)$ and $R>0$
such that $\Theta$ is even and real-valued,
$\supp \hat\Theta \subset [-R, R]$, and
\begin{equation}\label{approx-1}
|\Psi^{(\nu)}(u)- \Theta^{(\nu)}(u)| \le \frac{\eps|u|^\NN}{(1+|u|)^{2\NN}},
\quad u\in \R, \quad \nu=0, 1, \dots, \KK.
\end{equation}
Furthermore,
\begin{equation}\label{supp-hat-Theta}
\supp \cF(u^{-m}\Theta(u)) \subset [-R, R]
\quad\hbox{for $0\le m\le N$}
\end{equation}
with $\cF$ being the Fourier transform.
\end{proposition}

The constants $\NN, \KK$ and $\eps$ (sufficiently small) will be selected later on.
With these constants fixed, we use the functions $\Theta$ from Proposition~\ref{prop:constr-1}
to define the new frame.
%Similarly as in (\ref{def-Psi-j}) we set
%\begin{equation}\label{def-Theta-j}
%\Theta_j(u):=\Theta(b^{-j}u),\;\; j\in\bZ.
%\end{equation}
Let the sets $\cX_j$, $\cX:=\cup_{j\in\bZ}\cX_j$, and $\{A_\xi\}_{\xi\in \cX_j}$
be as in the definition of Frame \# 1.
We define a new system $\{\theta_{\xi}\}_{\xi\in\cX}$ by
\begin{equation}\label{new-frame}
\theta_{\xi} (x):= |A_\xi |^{1/2} \Theta(b^{-j}\sqrt{L})(x,\xi),
\quad \xi \in \XX_j,\; j\in\bZ.
\end{equation}

Observe that by the fact that $\supp \hat\Theta \subset [-R, R]$
and the final speed propagation property (Proposition~\ref{prop:finite-sp}) %(see \cite[Proposition~2.8]{KP})
it follows that each $\theta_\xi$ is compactly supported, more precisely
\begin{equation}\label{support}
\supp \theta_{\xi} \subset B(\xi, \ct Rb^{-j}), \quad \xi\in \cX_j,\; j\in\bZ.
\end{equation}
The construction of a dual frame $\{\tilde{\theta}_{\xi}\}_{\xi\in\cX}$ is more involved,
see below and for more details see \cite{DKKP}.

The basic result we need here is that the systems
$\{\theta_{\xi}\}_{\xi\in\cX},\;\{\tilde{\theta}_{\xi}\}_{\xi\in\cX}$
form a pair of frames for Besov and Triebel-Lizorkin spaces for
the following range $\Omega$ of indices determined by constants $s_0\geq 0, p_0,p_1,q_0>0$:
\begin{equation}\label{omega}
\Omega:=\{(s,p,q):|s|\leq s_0,\;p_0\leq p\leq p_1,\;q_0\leq q<\infty\}.
\end{equation}
We introduce the following notation:
$\cJ_0:= d/\min\{1, p_0\}$ in the case of $\BB$-spaces and
$\cJ_0:= d/\min\{1, p_0, q_0\}$ in the case of $\FF$-spaces.

%%%%%%% Theorem

\begin{theorem}\label{thm:main}
Suppose $(s,p,q)\in\Omega$ and let $\{\theta_\xi\}_{\xi\in\cX}$ be the system constructed
in $(\ref{new-frame})$, where
\begin{equation}\label{cond-NK}
K \geq s_0+\cJ_0+d/2+1
\quad\hbox{and}\quad
N \geq K+s_0+\cJ_0+3d/2+1
\end{equation}
Then for sufficiently small $\eps$ in the construction of $\{\theta_\xi\}_{\xi\in\cX}$
the following claims are valid:

$(a)$ The operator
$$
Tf:=\sum_{\xi\in\cX} \langle f, \tilde\psi_\xi\rangle \theta_\xi,
$$
is invertible on $\BB^s_{pq}$ and $T$, $T^{-1}$ are bounded on $\BB^s_{pq}$.
%The same holds true when $B^s_{pq}$ is replaced by $\tB^s_{pq}$, $F^s_{pq}$, or $\tF^s_{pq}$.

$(b)$ The system $\{\tilde\theta_\xi\}_{\xi\in\cX}$ defined by
$$
\tilde\theta_\xi := \sum_{\eta\in\cX} \overline{\langle T^{-1}\psi_\eta, \tilde\psi_\xi\rangle} \tilde\psi_\eta,
\quad\xi\in\cX,
$$
is a dual frame to $\{\theta_\xi\}_{\xi\in\cX}$ in the following sense:
For any $f\in \BB^s_{pq}$
\begin{equation}\label{frame-B}
f=\sum_{\xi\in\cX} \langle f, \tilde\theta_\xi\rangle \theta_\xi
\quad\hbox{and}\quad
\|f\|_{\BB^s_{pq}} \sim \|(\langle f, \tilde\theta_\xi\rangle)\|_{\bb^s_{pq}},
\end{equation}
where $\langle f, \tilde\theta_\xi\rangle$ is defined by
\begin{equation}\label{def-f-dual-B}
\langle f, \tilde\theta_\xi\rangle
:= \sum_{\eta\in\cX} \langle T^{-1}\psi_\eta, \tilde\psi_\xi\rangle \langle f, \tilde\psi_\eta\rangle
\end{equation}
and the convergence in $(\ref{frame-B})$ is in $\cS'/\PP$ and unconditional in $\BB^s_{pq}$.

Furthermore, $(a)$ and $(b)$ hold true when
$\BB^s_{pq}$ is replaced by $\tBB^s_{pq}$, $\FF^s_{pq}$, or $\tFF^s_{pq}$, and
$\bb^s_{pq}$  by $\tbb^s_{pq}$, $\ff^s_{pq}$, or $\tff^s_{pq}$, respectively.
\end{theorem}

The proof of this theorem is a straightforward adaptation of the proof of Theorem~4.2 in \cite{DKKP};
we omit it.

\section{Smooth molecular and atomic decompositions}\label{sec:molecules}

Families of smooth atoms and molecules on Homogeneous Besov and Tribel-Lizorkin spaces in the classical case on $\R^n$
are introduced and studied in \cite{FJ1,FJ2}. 
They not only provide convenient building blocks for various spaces of distributions,
but also can be used in establishing boundedness of operators \cite{Bownik,CGO,GT,Skrz, Tor,YSY}. 
%Their importance is that not only these can generalize the $\varphi$-transform, under weaker assumptions,
%but these can be further used in the boundedness of operators \cite{Bownik,CGO,GT,Skrz, Tor,YSY}.

\subsection{Smooth molecules for \boldmath $\BB$ and $\FF$-spaces}\label{subsec:molecules-BF}

In this section we generalize the boundedness of the operators $T_{\psi}$, $S_{\tilde{\psi}}$
from (\ref{anal_synth_oprts2}) (see Theorems \ref{thm:B-character} and \ref{thm:F-character})
by replacing $\{\psi_\xi\}$ and $\{\tilde\psi_\xi\}$ by families of smooth molecules.
We present the results for the $\FF^{s}_{pq}$ spaces, but they also hold for the $\BB^{s}_{pq}$ spaces.
We recall from Definition~\ref{def:almost-diagonal} that
$\cJ:= d/\min \{1, p\}$ for $\BB^s_{pq}$, $\tBB^s_{pq}$ and
$\cJ:= d/\min \{1, p, q\}$ for $\FF^s_{pq}$, $\tFF^s_{pq}$.
Also, we set
\begin{equation}\label{def-K-N}
K:=\lfloor(\cJ-s)/2\rfloor + 1 \;\;\hbox{ if $s\le \cJ$}
\quad\hbox{and}\quad
N:=\lfloor s/2\rfloor+1\;\; \hbox{if $s\ge 0$}
\end{equation}
($K$, $N$ will be needed only for the indicated values of $s$).
As before $\cX=\cup_{j\in\bZ}\cX_j$ will be the set centers of the frame elements
$\{\psi_\xi\}$ and $\{\tilde\psi_\xi\}$
and $\{A_\xi\}_{\xi\in\cX_j}$ will be the companion disjoint partitions of $M$,
see Remark~\ref{rem:X-A-xi}.
Recall that $\ell(\xi):=b^{-j}$ and $\AA_\xi:= B(\xi, \delta_j)$, $\delta_j=\gamma b^{-j-2}$, $\xi\in \cX_j$, $j\in\bZ$.

\subsection*{Definition of smooth synthesis molecules}

Let $s\in \R$, $0<p<\infty$, $0<q\le \infty$,
and let $\cJ, K,N$ be as above.
We say that $\{m_{\xi}\}_{\xi\in\cX}$ is a family of smooth synthesis molecules for $\FF^{s}_{pq}$,
if there exists $\MM>\cJ$  such that for any $\xi\in\cX$:

(i)
\begin{equation}\label{fssmbb}
|m_{\xi}(x)|\le \frac{|\AA_{\xi}|^{-1/2}}{(1+\ell(\xi)^{-1}\rho(x,\xi))^{\MM}}.
%\quad \xi\in\cX.
\end{equation}

(ii) If $s\ge 0$ it is assumed that $m_\xi\in D(L^N)$ and for $0<\nu\le N$
\begin{equation}\label{fssmb}
|L^{\nu}m_{\xi}(x)|\le \frac{\ell(\xi)^{-2\nu}|\AA_{\xi}|^{-1/2}}{(1+\ell(\xi)^{-1}\rho(x,\xi))^{\MM}}.
\end{equation}

(iii) In addition, if $s\le \cJ$, it is also assumed that there exists
a family of functions $\{b_{\xi}\}_{\xi\in\cX}$, $b_{\xi} \in D(L^K)$, such that
\begin{equation}\label{fssma}
m_{\xi}=L^{K}b_{\xi},
\end{equation}
and for  $0\le\nu<K$
\begin{equation}\label{fssmc}
|L^{\nu}b_{\xi}(x)|
\le \frac{\ell(\xi)^{2(K-\nu)}|\AA_{\xi}|^{-1/2}}{(1+\ell(\xi)^{-1}\rho(x,\xi))^{\MM}}.
\end{equation}

Note that (\ref{fssmb}) is void if $s<0$, and (\ref{fssma})-(\ref{fssmc}) are void if $s>\cJ$.

\subsection*{Definition of smooth analysis molecules}

Let $s,p,q$ and $\cJ, K,N$ be as above.
We say that $\{\tilde{m}_{\xi}\}_{\xi\in\cX}$ is a family of smooth analysis molecules for $\FF^{s}_{pq}$,
if there exists $\MM>\cJ$ such that for any $\xi\in\cX$:

(i)
\begin{equation}\label{fsambb}
|\tilde m_{\xi}(x)|
\le \frac{|\AA_{\xi}|^{-1/2}}{(1+\ell(\xi)^{-1}\rho(x,\xi))^{\MM+d}}.
\end{equation}

(ii) If $s\le \cJ$ it is assumed that $\tilde{m}\in D(L^K)$ and for $0\le \nu\le K$
\begin{equation}\label{fsamb}
|L^{\nu}\tilde{m}_{\xi}(x)|
\le \frac{\ell(\xi)^{-2\nu}|\AA_{\xi}|^{-1/2}}{(1+\ell(\xi)^{-1}\rho(x,\xi))^{\MM+d}}.
\end{equation}

(iii) In addition, if $s\ge 0$,
it is also assumed that there exists a family of functions
$\{\tilde b_{\xi}\}_{\xi\in\cX}$, $\tilde b_{\xi}\in D(L^N)$,
such that:
\begin{equation}\label{fsama}
\tilde{m}_{\xi}=L^{N}\tilde{b}_{\xi},
\end{equation}
and for $0\leq\nu\le N$
\begin{equation}\label{fsamc}
|L^{\nu}\tilde{b}_{\xi}(x)|
\le \frac{\ell(\xi)^{2(N-\nu)}|\AA_{\xi}|^{-1/2}}{(1+\ell(\xi)^{-1}\rho(x,\xi))^{\MM+d}}.
\end{equation}

As before condition (\ref{fsamb}) is void whenever $s>\cJ$, and (\ref{fsama})-(\ref{fsamc}) are void if $s<0$.

\begin{remark}
If $\{m_\xi\}_{\xi\in\cX}$ is a family of smooth synthesis $($or analysis$)$ molecules,
then we shall say that $m_\eta$ $(\eta\in\cX)$ is a molecule centered at $\eta$.
\end{remark}

\smallskip

The first step here is to establish the following:

\begin{lemma}\label{lem:frame-molec}
There exist constants $\sc, \ch >0$ such that
each of the frames $\{\sc\psi_{\xi}\}_{\xi\in\cX}$ and $\{\ch\tilde{\psi}_{\xi}\}_{\xi\in\cX}$
is a family of smooth synthesis and analysis molecules for $\FF^s_{pq}$.
\end{lemma}

\begin{proof}
We first show that there exists a constant $\sc>0$ such that $\{\sc\psi_{\xi}\}_{\xi\in\cX}$
is a family of smooth synthesis molecules.
Let
$$
m_\xi(x):=\sc \psi_{\xi}(x)=\sc |A_{\xi}|^{1/2}\Psi(b^{-j}\sqrt{L})(x,\xi),
\quad\hbox{(see (\ref{def-frame}))}
$$
where $\sc >0$ is a constant that will be selected later on.
From (\ref{prop-tpsi-1}) it readily follows that $m_\xi$ obeys (\ref{fssmbb})-(\ref{fssmb})
if the constant $\sc$ is sufficiently small.

Let now $s\le \cJ$ and $\xi\in\cX_{j}$.
Write (see (\ref{def-frame}) and Remark~\ref{rem:calculus})
$$
b_\xi(x):=\sc L^{-K}\psi_{\xi}(x)=\sc |A_{\xi}|^{1/2}[L^{-K}\Psi(b^{-j}\sqrt{L})](x,\xi).
%\quad\hbox{(see (\ref{def-frame}))}
$$
Then
$$
L^Kb_\xi(x)%=\sc |A_{\xi}|^{1/2}\Psi(b^{-j}\sqrt{L})(x,\xi)
=\sc \psi_\xi(x)= m_\xi(x).
$$
Assuming that $0\le\nu\le K$, $K=\lfloor (\cJ-s)/2\rfloor+1$, we set $g(u):=u^{2(\nu-K)}\Psi(u)$.
Clearly
$L^{\nu-K}\Psi(b^{-j}\sqrt{L})= b^{-2j(K-\nu)}g(b^{-j}\sqrt{L})$,
$g\in C^\infty(\R_+)$ and $\supp g \subset [b^{-1},b]$.
Then by Theorem~\ref{thm:S-local-kernels}, applied to $g$, it follows that for any $\MM>0$
\begin{align*}
|L^{\nu}b_{\xi}(x)|
\le \frac{c \sc |A_{\xi}|^{1/2}b^{-2j(K-\nu)}}{|B(\xi, b^{-j})|(1+b^j\rho(\xi, x))^\MM}
\le \frac{c \sc \ell(\xi)^{2(K-\nu)}|\AA_{\xi}|^{-1/2}}{(1+\ell(\xi)^{-1}\rho(x,\xi))^{\MM}},
\end{align*}
where the constant $c>0$ depends on $\MM$.
We fix $\MM>\cJ$.
By choosing the constant $\sc$ sufficiently small we conclude that
$\{m_\xi\}$ with $m_\xi:=\sc \psi_\xi$ is a family of smooth synthesis molecules.

Just as above one shows that there exists a constant $\sc>0$ such that
$\{\sc \psi_{\xi}\}_{\xi\in\cX}$ is a family of smooth analysis molecules.
We omit the details.

%%%%%%%%%%%%%%%%%%%%%%%

\smallskip

We next show that there exists a constant $\ch>0$ such that $\{\ch\tilde{\psi}_{\xi}\}_{\xi\in\cX}$
is a family of smooth analysis molecules.
Denote (see (\ref{def-dual}))
$$
\tilde{m}_\xi(x):=\ch\tilde{\psi}_{\xi}(x)
=\ch c_{\epsilon}|A_{\xi}|^{1/2}T_{\lambda_{j}}\left(\Gamma_{\lambda_{j}}(\cdot,\xi)\right)(x),
\quad \lambda_j:=b^{j-1},\;\; \xi\in\cX_j.
$$
where $\ch >0$ is a constant that will be selected later on.
From (\ref{prop-tpsi-1}) it readily follows that $\tilde{m}_\xi$ obeys (\ref{fsambb})-(\ref{fsamb})
if the constant $\ch$ is sufficiently small.

Let $s\ge 0$ and $\xi\in\cX_{j}$.
Set
\begin{equation}\label{def-b-xi}
\tilde{b}_\xi(x) :=\ch c_{\epsilon}|A_{\xi}|^{1/2}L^{-N}T_{\lambda_{j}}\left(\Gamma_{\lambda_{j}}(\cdot,\xi)\right)(x).
\end{equation}
Clearly,
$\tilde{m}_\xi = L^N \tilde{b}_\xi$
and it remains to show that for $0\le \nu \le N$ the function
$L^\nu \tilde{b}_\xi$ obeys (\ref{fsamc}).
Observe that the operators $L^{-N}$ and $T_{\lambda_j}$ do not necessarily commute.
We go round this obstacle just as in the proof of Theorem~4.3 in \cite{KP}.
We have
\begin{equation}\label{rep-Tl-Sl}
T_{\lambda_j}:= \Id + S_{\lambda_j}
\quad\hbox{and}\quad
S_{\lambda_j} := \sum_{k\ge 1} R_{\lambda_j} = R_{\lambda_j}(\Id + S_{\lambda_j}),
\end{equation}
where the operator $R_{\lambda_j}$ is from the proof of Lemma~4.2 in \cite{KP}.
In fact, we have
$R_{\lambda_j} = \Gamma_{\lambda_j}^2- V_{\lambda_j}$,
where $V_{\lambda_j}$ is the operator with kernel
\begin{equation}\label{def-V-lam}
V_{\lambda_j}(x, y):= \sum_{\eta\in\cX_j}\omega_\eta \Gamma_{\lambda_j}(x, \eta)\Gamma_{\lambda_j}(\eta, y).
\end{equation}
Here $\omega_\eta:=(1+\eps)^{-1}|A_\eta|$, just as in Lemma~\ref{lem:instrument}.
From (\ref{def-b-xi})-(\ref{rep-Tl-Sl}) we derive the following representation of $L^\nu \tilde{b}_\xi$:
\begin{align*}
L^\nu \tilde{b}_\xi(x)
= \ch c_\eps |A_\xi|^{1/2}\Big(L^{\nu-N}\Gamma_{\lambda_j}(x, \xi)
&+ L^{\nu-N} R_{\lambda_j}[\Gamma_{\lambda_j}(\cdot, \xi)](x)\\
&+ L^{\nu-N} R_{\lambda_j} S_{\lambda_j}[\Gamma_{\lambda_j}(\cdot, \xi)](x)\Big).
\end{align*}
Let $h(u):= u^{2(\nu-N)}\Gamma(u)$.
We have
$L^{\nu-N}\Gamma_{\lambda_j} = \lambda_j^{-2(N-\nu)}h(\lambda_j^{-1}\sqrt{L})$,
and clearly $h\in C^\infty(\R_+)$ and $\supp h \subset [b^{-1},b^3]$.
We now apply Theorem~\ref{thm:S-local-kernels} to $h(\lambda_j^{-1}\sqrt{L})$ to conclude that
for any $\sigma >0$ there exists a constant $c_\sigma$ such that
the kernel of the operator $L^{\nu-N}\Gamma_{\lambda_j}$ obeys
\begin{equation*}
|[L^{\nu-N}\Gamma_{\lambda_j}](x, y)|
\le \frac{c_\sigma \lambda_j^{-2(N-\nu)}}{|B(y, \lambda_j)|(1+\lambda_j\rho(x, y))^\sigma}.
%\le \frac{c b^{-2j(N-\nu)}}{|B(y, b^{-j})|(1+b^j\rho(x, y))^\sigma}.
\end{equation*}
We choose $\sigma:=\MM+4d$, where $\MM>\cJ$ is fixed.
Then
\begin{equation}\label{est-L-Gamma}
|[L^{\nu-N}\Gamma_{\lambda_j}](x, y)|
%\le \frac{c_\sigma \lambda_j^{-2(N-\nu)}}{|B(y, \lambda_j)|(1+\lambda_j\rho(x, y))^\sigma}
\le \frac{c b^{-2j(N-\nu)}}{|B(y, b^{-j})|(1+b^j\rho(x, y))^{\MM+4d}}.
\end{equation}
Also, by Theorem~\ref{thm:S-local-kernels}
\begin{equation}\label{est-Gamma}
|\Gamma_{\lambda_j}(x, y)|
%\le \frac{c_\sigma}{|B(\xi, \lambda_j)|(1+\lambda_j\rho(\xi, x))^\sigma}
\le \frac{c}{|B(y, b^{-j})|(1+b^j\rho(x, y))^{\MM+4d}}.
\end{equation}
Estimates (\ref{est-L-Gamma})-(\ref{est-Gamma}) along with (\ref{def-V-lam}) and Lemma~\ref{lem:est-discr-sum} yield
$$
|[L^{\nu-N}V_{\lambda_j}](x, y)|
\le \frac{c b^{-2j(N-\nu)}}{|B(y, b^{-j})|(1+b^j\rho(x, y))^{\MM+4d}}.
$$
On the other hand, (\ref{est-L-Gamma})-(\ref{est-Gamma}) and Lemma~\ref{lem:gen-ineq} imply
$$
|[L^{\nu-N}\Gamma_{\lambda_j}^2](x, y)|
\le \frac{c b^{-2j(N-\nu)}}{|B(y, b^{-j})|(1+b^j\rho(x, y))^{\MM+3d}}.
$$
Therefore, the kernel of the operator $R_{\lambda_j}$ satisfies
\begin{equation}\label{est-R-lam}
|[L^{\nu-N}R_{\lambda_j}](x, y)|
\le \frac{c b^{-2j(N-\nu)}}{|B(y, b^{-j})|(1+b^j\rho(x, y))^{\MM+3d}}.
\end{equation}
We apply inequality (\ref{int-est-1}) twice using (\ref{est-R-lam}), (\ref{local-S}), and (\ref{est-Gamma})
to obtain
\begin{equation*}
|L^{\nu-N} R_{\lambda_j} S_{\lambda_j}[\Gamma_{\lambda_j}(\cdot, \xi)](x)|
\le \frac{c b^{-2j(N-\nu)}}{|B(\xi, b^{-j})|(1+b^j\rho(x, \xi))^{\MM+d}}.
\end{equation*}
By the same token we get a similar estimate for
$|L^{\nu-N} R_{\lambda_j}[\Gamma_{\lambda_j}(\cdot, \xi)](x)|$.
These estimates and (\ref{est-L-Gamma}) imply
\begin{align*}
|L^{\nu}\tilde{b}_{\xi}(x)|
\le \frac{c \ch |A_{\xi}|^{1/2}b^{-2j(N-\nu)}}{|B(\xi, b^{-j})|(1+b^j\rho(\xi, x))^{\MM+d}}
\le \frac{c \ch \ell(\xi)^{2(N-\nu)}|\AA_{\xi}|^{-1/2}}{(1+\ell(\xi)^{-1}\rho(x,\xi))^{\MM+d}},
\end{align*}
Therefore, inequality (\ref{fsamc}) holds if the constant $\ch$ is sufficiently small.
Consequently, $\{\tilde{m}_\xi\}$ with $\tilde{m}_\xi:=\ch \tilde{\psi}_\xi$
is a family of analysis molecules.

Exactly as above one shows that there exists a constant $\ch>0$ such that
$\{\ch \tilde{\psi}_\xi\}_{\xi\in\cX}$ is a family of synthesis molecules.
We omit the details.
\end{proof}

%---------------------------------------------------------

\begin{lemma}\label{lem:molecule-ad}
Suppose $\{m_{\xi}\}$ and $\{\tilde{m}_{\xi}\}$ are families of smooth synthesis and analysis molecules
for $\FF^{s}_{pq}$, respectively, and let $A$ be the operator with matrix
$$
(a_{\xi\eta})_{\xi,\eta\in\cX}=(\langle m_{\eta},\tilde{m}_{\xi}\rangle)_{\xi,\eta\in\cX}.
$$
Then there exist constants $c,\delta>0$ such that
\begin{equation}\label{est-a-xi-eta}
|a_{\xi\eta}|\le c \omega_{\delta}(\xi,\eta), \quad \forall \xi,\eta\in \cX,
\end{equation}
where $\omega_{\delta}(\xi,\eta)$ is from $(\ref{def-omega-d-1})$.
Therefore, $A$ is almost diagonal on $\ff^s_{pq}$ and by Theorem~\ref{thm:AlmDiag} the operator $A$ is bounded on $\ff^s_{pq}$.
\end{lemma}

\begin{proof}
Under the hypothesis of the lemma, two cases present themselves here.

{\em Case 1:} $\ell(\xi)\ge \ell(\eta)$.
We consider two subcases depending on whether $s\le \cJ$ or $s>\cJ$.

Let $s\le \cJ$. By the definition of synthesis molecules (\S\ref{subsec:molecules-BF})
there exists a function $b_{\eta}\in D(L^K)$, $K:=\lfloor(\cJ-s)/2\rfloor+1$, such that $m_{\eta}=L^{K}b_{\eta}$
and
\begin{equation}\label{CI2}
|b_{\eta}(x)|\le \ell(\eta)^{2K}|\AA_{\eta}|^{-1/2}\Big(1+\frac{\rho(x,\eta)}{\ell(\eta)}\Big)^{-\MM}.
\end{equation}
On the other hand, from (\ref{fsamb})
\begin{equation}\label{CI3}
|L^{K}\tilde{m}_{\xi}(x)|
\le \ell(\xi)^{-2K}|\AA_{\xi}|^{-1/2}\Big(1+\frac{\rho(x,\xi)}{\ell(\xi)}\Big)^{-\MM-d}.
\end{equation}
Clearly,
$
a_{\xi\eta}
=\langle L^{K}b_{\eta},\tilde{m}_{\xi}\rangle
=\langle b_{\eta},L^{K}\tilde{m}_{\xi}\rangle
$
and using (\ref{CI2})-(\ref{CI3}) and (\ref{int-est-2}) we obtain
\begin{align*}
|a_{\xi\eta}|
&\le |\AA_{\eta}|^{-1/2}|\AA_{\xi}|^{-1/2}\Big(\frac{\ell(\eta)}{\ell(\xi)}\Big)^{2K}
\int_{M}\Big(1+\frac{\rho(x,\eta)}{\ell(\eta)}\Big)^{-\MM}\Big(1+\frac{\rho(\xi,x)}{\ell(\xi)}\Big)^{-\MM-d}d\mu(x)
\notag\\
&\le c|\AA_{\eta}|^{-1/2}|\AA_{\xi}|^{-1/2}\Big(\frac{\ell(\eta)}{\ell(\xi)}\Big)^{2K}
|B(\xi, \ell(\xi))|\Big(1+\frac{\rho(\xi,x)}{\ell(\xi)}\Big)^{-\MM}.
%\\
%&\le c \Big(\frac{\ell(\eta)}{\ell(\xi)}\Big)^{2K}\Big(\frac{|\AA_{\xi}|}{|\AA_{\eta}|}\Big)^{1/2}
%\Big(1+\frac{\rho(\xi,\eta)}{\ell(\xi)}\Big)^{-\MM}.
\end{align*}
Here for the last inequality we  used  that $|B(\xi,\ell(\xi))|\sim |\AA_\xi|$.
Hence
\begin{equation}\label{est-a-xi-et}
|a_{\xi\eta}| \le c \Big(\frac{\ell(\eta)}{\ell(\xi)}\Big)^{2K}\Big(\frac{|\AA_{\xi}|}{|\AA_{\eta}|}\Big)^{1/2}
\Big(1+\frac{\rho(\xi,\eta)}{\ell(\xi)}\Big)^{-\MM}.
\end{equation}
In the light of (\ref{def-omega-d-1}) the above implies (\ref{est-a-xi-eta})
for any $0< \delta\le \min\{2K-\cJ+s, \MM-\cJ\}$.

If $s>\cJ$, we use that $a_{\xi\eta}=\langle m_{\eta}, \tilde{m}_{\xi} \rangle$,
the fact that $m_{\eta}$, $\tilde{m}_{\xi}$ satisfy (\ref{fssmbb}), (\ref{fsambb}),
and (\ref{int-est-2}) to obtain
\begin{equation}\label{nomoments}
\begin{aligned}
|a_{\xi\eta}|
& \le |\AA_{\eta}|^{-1/2}|\AA_{\xi}|^{-1/2}
\int_{M}\Big(1+\frac{\rho(x,\eta)}{\ell(\eta)}\Big)^{-\MM}\Big(1+\frac{\rho(\xi,x)}{\ell(\xi)}\Big)^{-\MM-d}d\mu(x)
\\
&\le c\Big(\frac{|\AA_{\xi}|}{|\AA_{\eta}|}\Big)^{1/2}\Big(1+\frac{\rho(\xi,\eta)}{\ell(\xi)}\Big)^{-\MM}.
\end{aligned}
\end{equation}
It is easy to see that this implies (\ref{est-a-xi-eta})
for $0<\delta \le \min\{s-\cJ, \MM-\cJ\}$.

\smallskip

{\em Case 2:} $\ell(\eta)> \ell(\xi)$. Here we consider two subcases: $s\ge 0$ or $s<0$.

Let $s\ge 0$. By the definition of analysis molecules there exists
%a function
$\tilde{b}_{\xi}\in D(L^N)$, $N:=\lfloor s/2\rfloor+1$, such that
$\tilde{m}_{\xi}=L^{N}\tilde{b}_{\xi}$
and $\tilde{b}_{\xi}$ obeys
\begin{equation}\label{est-t-b}
|\tilde{b}_\xi(x)|\le \ell(\xi)^{2N}|\AA_{\xi}|^{-1/2}\Big(1+\frac{\rho(x,\xi)}{\ell(\xi)}\Big)^{-\MM-d}.
\end{equation}
Furthermore, by (\ref{fssmb})
\begin{equation}\label{est-L-N-m}
|L^N m_{\eta}(x)|
\le \ell(\eta)^{-2N}|\AA_{\eta}|^{-1/2}\Big(1+\frac{\rho(x,\eta)}{\ell(\eta)}\Big)^{-\MM}.
\end{equation}
Clearly,
$a_{\xi\eta}=\langle m_{\eta}, L^N\tilde{b}_{\xi}\rangle
=\langle L^Nm_{\eta}, \tilde{b}_{\xi}\rangle$.
Then using  (\ref{est-t-b})-(\ref{est-L-N-m}), and  (\ref{int-est-2}) we obtain just as in Case 1
\begin{align}\label{a-xi-eta-est}
|a_{\xi\eta}|
&\le\Big(\frac{\ell(\xi)}{\ell(\eta)}\Big)^{2N}|\AA_{\eta}|^{-1/2}|\AA_{\xi}|^{-1/2}
\int_{M}\Big(1+\frac{\rho(x,\eta)}{\ell(\eta)}\Big)^{-\MM}\Big(1+\frac{\rho(\xi,x)}{\ell(\xi)}\Big)^{-\MM-d}d\mu(x)
\\
&\le c \Big(\frac{\ell(\xi)}{\ell(\eta)}\Big)^{2N}\left(\frac{|\AA_{\xi}|}{|\AA_{\eta}|}\right)^{1/2}
\Big(1+\frac{\rho(\xi,\eta)}{\ell(\eta)}\Big)^{-\MM}.\notag
\end{align}
This estimate and the fact that $N=\lfloor s/2\rfloor+1$ readily imply (\ref{est-a-xi-eta}) for an arbitrary
$0<\delta \le\min\{2N-s, \MM-\cJ\}$.

If $s<0$, we use that $a_{\xi\eta}=\langle m_{\eta}, \tilde{m}_{\xi} \rangle$,
the fact that
$m_{\eta}$, $\tilde{m}_{\xi}$ obey (\ref{fssmbb}), and (\ref{fsambb}) to obtain
\begin{equation}\label{nomoments-2}
\begin{aligned}
|a_{\xi\eta}|&
\le |\AA_{\eta}|^{-1/2}|\AA_{\xi}|^{-1/2}
\int_{M}\Big(1+\frac{\rho(x,\eta)}{\ell(\eta)}\Big)^{-\MM}\Big(1+\frac{\rho(\xi,x)}{\ell(\xi)}\Big)^{-\MM-d}d\mu(x)
\\
&\le c\Big(\frac{|\AA_{\xi}|}{|\AA_{\eta}|}\Big)^{1/2}\Big(1+\frac{\rho(\xi,\eta)}{\ell(\eta)}\Big)^{-\MM},
\end{aligned}
\end{equation}
implying (\ref{est-a-xi-eta})
%
%$$
%|a_{\xi\eta}|\le c w_{\xi\eta}(\delta),
%$$
for any $0<\delta \le \min\{-s, \MM-\cJ\}$.

Finally, choosing $\delta >0$ sufficiently small we arrive at (\ref{est-a-xi-eta}).
\end{proof}
%\smallskip
%
%{\em Case 3:} $\ell(\xi)=\ell(\eta)$.
%Here we only have to treat the case when
%$\ell(\xi)=\ell(\eta)=1$ since for $\ell(\xi)<1$
%the result follows directly  from the proofs of any of the Cases 1, 2.
%Indeed, assume that $\ell(\xi)=\ell(\eta)=1$ .
%Then $a_{\xi\eta}=\langle m_{\eta},\tilde{m}_{\xi}\rangle$
%and the result follows identically as in (\ref{nomoments}) for any $0<\delta <\min\{-s, \MM-\cJ\}$.
%\end{proof}

After this preparation we come to the main assertions in this section.

\begin{theorem}\label{thm:molec-synthesis}$($Smooth molecular synthesis$)$
If
$\{m_{\xi}\}_{\xi\in \cX}$ is a family of smooth synthesis molecules for $\FF^{s}_{pq}$,
then for any sequence
$t=\{t_{\xi}\}_{\xi\in \cX} \in \ff^s_{pq}$
\begin{equation}\label{molec-synthesis}
\Big\|\sum_{\xi\in\cX}t_{\xi}m_{\xi}\Big\|_{\FF^{s}_{pq}}\le c\|t\|_{\ff^{s}_{pq}},
\end{equation}
where the constant $c>0$ is independent of $\{m_{\xi}\}$ and $\{t_{\xi}\}$.
\end{theorem}

\begin{proof}
By (\ref{rep-L2}) we have
$
m_{\eta}=\sum_{\xi\in\mathcal{X}}\langle m_{\eta},\tilde{\psi}_{\xi}\rangle\psi_{\xi},
$
where the convergence is in $\cS'/\PP$.
By Lemmas~\ref{lem:frame-molec}-\ref{lem:molecule-ad} it follows that the operator $A$ with matrix
$$
(a_{\xi\eta})_{\xi,\eta\in\cX} := (\langle m_{\eta},\tilde{\psi}_{\xi}\rangle)_{\xi,\eta\in\cX}
$$
is almost diagonal on $\ff^s_{pq}$ and hence, by Theorem~\ref{thm:AlmDiag}, $A$ is bounded on $\ff^s_{pq}$.
On the other hand, by Theorem~\ref{thm:F-character} the synthesis operator
$T_{\psi}:\ff^{s}_{pq}\rightarrow \FF^{s}_{pq}$
from (\ref{anal_synth_oprts2}) is also bounded.
Observe that
\begin{align*}
T_{\psi}At
=&\sum_{\xi\in\mathcal{X}}(At)_{\xi}\psi_{\xi}
=\sum_{\xi\in\mathcal{X}}\sum_{\eta\in\mathcal{X}}a_{\xi\eta}t_{\eta}\psi_{\xi}
\\
&=\sum_{\eta\in\mathcal{X}}\Big(\sum_{\xi\in\mathcal{X}}a_{\xi\eta}\psi_{\xi}\Big)t_{\eta}
=\sum_{\eta\in\mathcal{X}}m_{\eta}t_{\eta}=:f.
\end{align*}
Now, using the boundedness of the operators $T_{\psi}$ and $A$ we infer
\begin{align*}
\|f\|_{\FF^{s}_{pq}}= \|T_{\psi}At\|_{\FF^s_{pq}}
\le c\|At\|_{\ff^s_{pq}} \le c\|t\|_{\ff^s_{pq}},
\end{align*}
which completes the proof.
\end{proof}

%%%%%%%%%%%%%%

\begin{theorem}\label{thm:molec-analysis} $($Smooth molecular analysis$)$
If $\{\tilde{m}_{\xi}\}$ is a family of smooth analysis molecules for $\FF^{s}_{pq}$,
then for any $f\in\FF^{s}_{pq}$
\begin{equation}\label{molec-analysis}
\|\{\langle f,\tilde{m}_{\xi}\rangle\}\|_{\ff^{s}_{pq}}\le c\|f\|_{\FF^{s}_{pq}},
\end{equation}
where $c>0$ is a constant independent of $f$ and $\{\tilde{m}_{\xi}\}$.
Here $\langle f,\tilde{m}_{\xi}\rangle$ is defined by
\begin{equation}\label{def-analysis}
\langle f,\tilde{m}_{\xi}\rangle
:= \sum_{\eta\in\cX}\langle \tilde{m}_{\xi}, \psi_\eta \rangle \langle f,\tilde\psi_{\eta}\rangle,
\end{equation}
where the series converges absolutely.
\end{theorem}

\begin{proof}
Let $A$ and $\hat{A}$ be the operators with matrices
$$
(a_{\xi\eta})_{\xi,\eta\in\cX} := (\langle \tilde{m}_{\xi}, \psi_\eta \rangle)_{\xi,\eta\in\cX}
\quad\hbox{and}\quad
(\hat{a}_{\xi\eta})_{\xi,\eta\in\cX} :=(|a_{\xi\eta}|)_{\xi,\eta\in\cX}.
$$
By Lemmas~\ref{lem:frame-molec}-\ref{lem:molecule-ad} it follows that
there exist constants $c, \delta>0$ such that
$$
|\langle \tilde{m}_{\xi}, \psi_\eta \rangle| \le c\omega_{\xi\eta}(\delta),
\quad \forall \xi, \eta\in\cX.
$$
Then Theorem~\ref{thm:AlmDiag} implies that both operators $A$ and $\hat{A}$
are bounded on $\ff^s_{pq}$.

On the other hand, by Theorem~\ref{thm:F-character}, for any $f\in\FF^s_{pq}$
the sequence $\{\langle f,\tilde\psi_{\eta}\rangle\}$
belongs to $\ff^s_{pq}$ and hence $\{|\langle f,\tilde\psi_{\eta}\rangle|\}\in\ff^s_{pq}$.
From this and the boundedness of the operator $\hat{A}$ it follows that
$$
\sum_{\eta\in\cX}|\langle \tilde{m}_{\xi}, \psi_\eta \rangle| |\langle f,\tilde\psi_{\eta}\rangle| <\infty,
\quad \forall \xi\in\cX.
$$
Thus the absolute convergence of the series in (\ref{def-analysis}) is established.

Let $f\in\FF^s_{pq}$. By the boundedness of the operator $A: \ff^s_{pq}\to \ff^s_{pq}$ and
the analysis operator $T_{\tilde\psi}: \FF^s_{pq}\to \ff^s_{pq}$ from (\ref{anal_synth_oprts2})
(Theorem~\ref{thm:F-character}), and (\ref{def-analysis}) it follows that
\begin{align*}
\|\{\langle f,\tilde{m}_{\xi}\rangle\}\|_{\ff^s_{pq}}
= \|AS_{\tilde\psi}f\|_{\ff^s_{pq}}
\le c\|S_{\tilde\psi}f\|_{\ff^s_{pq}}
\le c\|f\|_{\FF^s_{pq}},
\end{align*}
which confirms (\ref{molec-analysis}).
\end{proof}

\subsection{Smooth molecules for \boldmath $\tBB$ and $\tFF$-spaces}\label{subsec:molecules-tilde-BF}

In this section we establish results analogous to the ones from \S\ref{subsec:molecules-BF}
for the $\tFF^{s}_{pq}$ spaces.
Similar results hold as well for the Besov spaces $\tBB^{s}_{pq}$,
which we shall not treat here.
There are a lot of similarities between these results for the $\FF$-spaces and $\tFF$-spaces.
Therefore, we shall put the emphasis on the new features.
Thus, unlike the case of $\FF$-spaces here we also use %assume
the reverse doubling condition (\ref{RD}).

For $s\le \cJ d/d^*$  %$\cJ\ge sd^*/d$
with $\cJ:= d/\min \{1, p, q\}$ and $d^*$ the constant form (\ref{GRD}),
we define
\begin{equation}\label{def-K}
K:=\lfloor(\cJ-s)/2\rfloor+1, \hbox{ if $s<0$}
\quad\hbox{and}\quad
K:=\lfloor(\cJ-sd^*/d)/2\rfloor+1, \hbox{ if $s\ge 0$.}
\end{equation}
As before we set $N:=\lfloor s/2 \rfloor +1$, if $s\ge 0$.

\smallskip

\noindent
{\bf Definition of smooth synthesis molecules.}
Assume $s\in \R$, $0<p<\infty$, $0<q\le \infty$,
and let $\cJ, K,N$ be as above.
We say that $\{m_{\xi}\}_{\xi\in\cX}$ is a family of smooth synthesis molecules for
$\tFF^{s}_{pq}$, if there exists $\MM>\cJ+|s|$, such that:

(i)
\begin{equation}\label{fssm2a}
|m_{\xi}(x)|\le \frac{|\AA_{\xi}|^{-1/2}}{(1+\ell(\xi)^{-1}\rho(x,\xi))^{\MM}}.
\end{equation}

(ii) For $s\ge 0$ and $0\le\nu\le N,$
\begin{equation}\label{fssm2b}
|L^{\nu}m_{\xi}(x)|\le \frac{\ell(\xi)^{-2\nu}|\AA_{\xi}|^{-1/2}}{(1+\ell(\xi)^{-1}\rho(x,\xi))^{\MM}}.
\end{equation}

(iii) In addition, if $0\le s\le \cJ d/d^*$ %$\cJ\ge sd^*/d$,
it is assumed that there exists a family of functions
$\{b_{\xi}\}_{\xi\in\cX}$, $b_{\xi}\in D(L^K)$,
such that
\begin{equation}\label{fssm2c}
m_{\xi}=L^{K}b_{\xi},
\end{equation}
and for $0\le\nu\le K$
\begin{equation}\label{fssm2d}
|L^{\nu}b_{\xi}(x)|\le \frac{\ell(\xi)^{2(K-\nu)}|\AA_{\xi}|^{-1/2}}{(1+\ell(\xi)^{-1}\rho(x,\xi))^{\MM}}.
\end{equation}

Note that (\ref{fssm2b}) is void if $s<0$ and (\ref{fssm2c})-(\ref{fssm2d}) are void
if $s>\cJ d/d^*$. % $\cJ<sd^*/d$

\smallskip

\noindent
{\bf Definition of smooth analysis molecules.}
Let $s,p,q$ and $\cJ,K,N$ be as above.
We say that $\{\tilde{m}_{\xi}\}_{\xi\in\cX}$ is a family of smooth analysis molecules for $\tFF^{s}_{pq}$,
if there exist $\MM>\cJ+|s|$, such that:

(i)
\begin{equation}\label{fsam2a}
|\tilde{m}_{\xi}(x)|\le \frac{|\AA_{\xi}|^{-1/2}}{(1+\ell(\xi)^{-1}\rho(x,\xi))^{\MM+d}}.
\end{equation}

(ii) For $s\le\cJ d/d^*$ %$\cJ\ge sd^*/d$
and $0\le\nu\le K,$
\begin{equation}\label{fsam2b}
|L^{\nu}\tilde{m}_{\xi}(x)|\le \frac{\ell(\xi)^{-2\nu}|\AA_{\xi}|^{-1/2}}{(1+\ell(\xi)^{-1}\rho(x,\xi))^{\MM+d}}.
\end{equation}

(iii) In addition, if $s\ge 0$,
it is assumed that there exists a family of distributions
$\{\tilde{b}_{\xi}\}_{\xi\in\cX}$, $\tilde{b}_{\xi}\in D(L^N)$, such that
\begin{equation}\label{fsam2c}
\tilde{m}_{\xi}=L^{N}\tilde{b}_{\xi},
\end{equation}
and for $0\le\nu\le N$,
\begin{equation}\label{fsam2d}
|L^{\nu}\tilde{b}_{\xi}(x)|
\le \frac{\ell(\xi)^{2(N-\nu)}|\AA_{\xi}|^{-1/2}}{(1+\ell(\xi)^{-1}\rho(x,\xi))^{\MM+d}}.
\end{equation}

Note that (\ref{fsam2b}) is void if $s>\cJ d/d^*$
and (\ref{fsam2c})-(\ref{fsam2d}) are void if $s<0$.

As in \S\ref{subsec:molecules-BF} two lemmas will be needed.

\begin{lemma}\label{lem:frame-molec-t}
There exist constants $\sc, \ch >0$ such that
each of the frames $\{\sc\psi_{\xi}\}_{\xi\in\cX}$ and $\{\ch\tilde{\psi}_{\xi}\}_{\xi\in\cX}$
is a family of smooth synthesis and analysis molecules for $\tFF^s_{pq}$.
\end{lemma}

The proof of this lemma is almost identical to the proof of Lemma~\ref{lem:frame-molec};
only the ranges for $s$ are different, which is not essential for the proof. We omit it.

\begin{lemma}\label{lem:molecule-ad-t}
Suppose $\{m_{\xi}\}$ and $\{\tilde{m}_{\xi}\}$ are families of smooth synthesis and analysis molecules
for $\tFF^{s}_{pq}$, respectively, and let $A$ be the operator with matrix
$$
(a_{\xi\eta})_{\xi,\eta\in\cX}=(\langle m_{\eta},\tilde{m}_{\xi}\rangle)_{\xi,\eta\in\cX}.
$$
Then there exist constants $c,\delta>0$ such that
\begin{equation}\label{est-a-xi-eta-t}
|a_{\xi\eta}|\le c \omega_{\delta}(\xi,\eta), \quad \forall \xi,\eta\in \cX,
\end{equation}
where $\omega_{\delta}(\xi,\eta)$ is from $(\ref{def-omega-d-2})$.
Therefore, $A$ is almost diagonal on $\tff^s_{pq}$ and by Theorem~\ref{thm:AlmDiag} the operator $A$ is bounded on $\tff^s_{pq}$.
\end{lemma}

\begin{proof}
This proof will follow in the footsteps of the proof of Lemma~\ref{lem:molecule-ad}.
We shall use some of the estimates derived in the proof of Lemma~\ref{lem:molecule-ad} as well.
Under the assumptions of the lemma, we consider two cases.

{\em Case 1:} $\ell(\xi)\ge\ell(\eta)$.
There are two subcases to be considered depending of whether
$s\le \cJ d/{d^*}$ or $s> \cJ d/{d^*}$.
%$\cJ\geq sd^*/d$ and $\cJ<sd^*/d$.

Let $s\le \cJ d/{d^*}$. %$\cJ\geq sd^{*}/d$.
From the definition of synthesis molecules, there exists a function $b_{\eta}\in D(L^K)$
such that $m_{\eta}=L^{K}b_{\eta}$.
Thus
$|a_{\xi\eta}|=|\langle L^{K}b_{\eta},\tilde{m}_{\xi}\rangle|=|\langle b_{\eta},L^{K}\tilde{m}_{\xi}\rangle|.$
Now just as in the proof of (\ref{est-a-xi-et}) we obtain
\begin{equation}\label{a-est}
|a_{\xi\eta}| \le c\Big(\frac{\ell(\eta)}{\ell(\xi)}\Big)^{2K}\Big(\frac{|\AA_{\xi}|}{|\AA_{\eta}|}\Big)^{1/2}
\Big(1+\frac{\rho(\xi,\eta)}{\ell(\xi)}\Big)^{-\MM}.
\end{equation}
Let $s>0$. By (\ref{D2}) and (\ref{GRD}) it follows that
\begin{align*}
|B(\eta, \ell(\xi))| \le c_0\Big(1+\frac{\rho(\xi, \eta)}{\ell(\xi)}\Big)^d|B(\xi, \ell(\xi))|,
\quad
|B(\eta, \ell(\xi))| \ge c_3\Big(\frac{\ell(\xi)}{\ell(\eta)}\Big)^{d^*}|B(\eta, \ell(\eta))|,
\end{align*}
implying
\begin{equation}\label{B-B-est}
\Big(\frac{|\AA_{\eta}|}{|\AA_{\xi}|}\Big)^{s/d}
\leq c\Big(\frac{\ell(\eta)}{\ell(\xi)}\Big)^{sd^*/d}
\Big(1+\frac{\rho(\xi,\eta)}{\ell(\xi)}\Big)^{s}.
\end{equation}
%On the other hand, by (\ref{D2})
%\begin{equation}\label{doubl-est}
%\Big(1+\frac{\rho(\xi,\eta)}{\ell(\xi)}\Big)^{-d}\le c \frac{|\AA_{\xi}|}{|\AA_{\eta}|}.
%\end{equation}
%
Combining (\ref{a-est}) and (\ref{B-B-est}) we obtain
\begin{align*}
|a_{\xi\eta}|
\le c\Big(\frac{\ell(\eta)}{\ell(\xi)}\Big)^{2K+sd^{*}/d}\Big(\frac{|\AA_{\xi}|}{|\AA_{\eta}|}\Big)^{1/2+s/d}
\Big(1+\frac{\rho(\xi,\eta)}{\ell(\xi)}\Big)^{-\MM+s}
\le c  \omega_{\xi\eta}(\delta),
\end{align*}
whenever $0<\delta\le\min\{2K+sd^{*}/d-\cJ, \MM-s-\cJ\}$.

Let $s<0$. Using (\ref{doubling}) and (\ref{D2}) we get
$$ %\begin{equation}\label{B-B-est-2}
|B(\xi, \ell(\xi))|
\le c\Big(1+\frac{\rho(\xi, \eta)}{\ell(\xi)}\Big)^d\Big(\frac{\ell(\xi)}{\ell(\eta)}\Big)^d|B(\eta, \ell(\eta))|
$$ %\end{equation}
implying
\begin{equation}\label{ball-comp}
\Big(1+\frac{\rho(\xi,\eta)}{\ell(\xi)}\Big)^{s}
\le c \Big(\frac{\ell(\eta)}{\ell(\xi)}\Big)^{s} \Big(\frac{|\AA_{\xi}|}{|\AA_{\eta}|}\Big)^{s/d}.
\end{equation}
This coupled with (\ref{a-est}) leads to
\begin{equation*}
\begin{aligned}
|a_{\xi\eta}|
&\le c\Big(\frac{\ell(\eta)}{\ell(\xi)}\Big)^{2K+s}\Big(\frac{|\AA_{\xi}|}{|\AA_{\eta}|}\Big)^{1/2+s/d}
\Big(1+\frac{\rho(\xi,\eta)}{\ell(\xi)}\Big)^{-\MM-s}
\le c  \omega_{\xi\eta}(\delta),
\end{aligned}
\end{equation*}
whenever $0<\delta\le\min\{2K+s-\cJ, \MM+s-\cJ\}$.

Assume $s>\cJ d/{d^*}$ (hence $s>0$). Just as in the proof of (\ref{nomoments}) we obtain
$$
|a_{\xi\eta}|\le c\Big(\frac{|\AA_{\xi}|}{|\AA_{\eta}|}\Big)^{1/2}\Big(1+\frac{\rho(\xi,\eta)}{\ell(\xi)}\Big)^{-\MM}.
$$
From this and (\ref{B-B-est}) we obtain
\begin{align*}
|a_{\xi\eta}|
\le c\Big(\frac{\ell(\eta)}{\ell(\xi)}\Big)^{s d^{*}/d}\Big(\frac{|\AA_{\xi}|}{|\AA_{\eta}|}\Big)^{1/2+s/d}
\Big(1+\frac{\rho(\xi,\eta)}{\ell(\xi)}\Big)^{-\MM+s}
\le c \omega_{\xi\eta}(\delta),
\end{align*}
whenever $0<\delta\le\min\{sd^{*}/d-\cJ, \MM-s-\cJ\}$.

\smallskip

{\em Case 2:} $\ell(\xi)<\ell(\eta)$.
We consider two subcases: $s\ge 0$ or $s<0$.

Let $s\ge 0$. From the definition of analysis molecules, there exists a distribution $\tilde{b}_{\xi}$
such that $\tilde{m}_{\xi}=L^{N}\tilde{b}_{\xi}$, where $N=\lfloor s/2\rfloor +1$.
Thus  $a_{\xi\eta}=\langle L^{N}m_{\eta},\tilde{b}_{\xi}\rangle$ and
just as in (\ref{a-xi-eta-est}) it follows that
\begin{equation}\label{a-est-2}
|a_{\xi\eta}|\le c\Big(\frac{\ell(\xi)}{\ell(\eta)}\Big)^{2N}
\Big(\frac{|\AA_{\xi}|}{|\AA_{\eta}|}\Big)^{1/2}
\Big(1+\frac{\rho(\xi,\eta)}{\ell(\eta)}\Big)^{-\MM}.
\end{equation}
As in the proof of (\ref{ball-comp}) we get
\begin{equation}
\Big(1+\frac{\rho(\xi,\eta)}{\ell(\eta)}\Big)^{-s}\le c \Big(\frac{\ell(\eta)}{\ell(\xi)}\Big)^{s} \Big(\frac{|\AA_{\xi}|}{|\AA_{\eta}|}\Big)^{s/d}
\end{equation}
and combining this with (\ref{a-est-2}) it follows that
\begin{align*}
|a_{\xi\eta}|\le c\Big(\frac{\ell(\xi)}{\ell(\eta)}\Big)^{2N-s}
\Big(\frac{|\AA_{\xi}|}{|\AA_{\eta}|}\Big)^{1/2+s/d}
\Big(1+\frac{\rho(\xi,\eta)}{\ell(\eta)}\Big)^{-\MM+s}
\le c  \omega_{\xi\eta}(\delta),
\end{align*}
whenever $0<\delta\le\min\{2N-s, \MM-\cJ-s \}$.

Let $s<0$. As in the proof of (\ref{nomoments-2}) we obtain
\begin{equation*}
|a_{\xi\eta}|
\le c\Big(\frac{|\AA_{\xi}|}{|\AA_{\eta}|}\Big)^{1/2}\Big(1+\frac{\rho(\xi,\eta)}{\ell(\eta)}\Big)^{-\MM}.
\end{equation*}
On the other hand, as in the proof of (\ref{B-B-est}) we get
\begin{equation}\label{B-B-est-3}
\Big(1+\frac{\rho(\xi,\eta)}{\ell(\eta)}\Big)^{s}
\le c \Big(\frac{|\AA_{\xi}|}{|\AA_{\eta}|}\Big)^{s/d} \Big(\frac{\ell(\eta)}{\ell(\xi)}\Big)^{sd^*/d}.
\end{equation}
From the above two inequalities it follows that
\begin{equation*}
\begin{aligned}
|a_{\xi\eta}|\le c\Big(\frac{\ell(\xi)}{\ell(\eta)}\Big)^{-sd^*/d}
\Big(\frac{|\AA_{\xi}|}{|\AA_{\eta}|}\Big)^{1/2+s/d}
\Big(1+\frac{\rho(\xi,\eta)}{\ell(\eta)}\Big)^{-\MM-s}
\le c  \omega_{\xi\eta}(\delta),
\end{aligned}
\end{equation*}
whenever $0<\delta\le \min\{-sd^*/d, \MM+s-\cJ\}$.

Choosing $\delta$ sufficiently small the above estimates of $|a_{\xi\eta}|$ imply (\ref{est-a-xi-eta-t}).
\end{proof}

The next two theorems contain the main results of this section for $\tBB$ and $\tFF$-spaces.

\begin{theorem}\label{thm:molec-synthesis-t}$($Smooth molecular synthesis$)$
If $\{m_{\xi}\}_{\xi\in \cX}$ is a family of smooth synthesis molecules for $\tFF^{s}_{pq}$,
then for any sequence $\{t_{\xi}\}_{\xi\in \cX} \in \tff^s_{pq}$
\begin{equation}\label{molec-synthesis-t}
\Big\|\sum_{\xi\in\cX}t_{\xi}m_{\xi}\Big\|_{\tFF^{s}_{pq}}\le c\|t\|_{\tff^{s}_{pq}},
\end{equation}
where the constant $c>0$ is independent of $\{m_{\xi}\}$ and $\{t_{\xi}\}$.
\end{theorem}

\begin{theorem}\label{thm:molec-analysis-t}$($Smooth molecular analysis$)$
If $\{\tilde{m}_{\xi}\}_{\xi\in \cX}$ is a family of smooth analysis molecules for $\tFF^{s}_{pq}$,
then for any $f\in \tFF^{s}_{pq}$
\begin{equation}\label{molec-analysis-t}
\|\{\langle f,\tilde{m}_{\xi}\rangle\}\|_{\tff^{s}_{pq}}
\le c\|f\|_{\tFF^{s}_{pq}},
\end{equation}
where $c>0$ is a constant independent of $f$ and $\{\tilde{m}_{\xi}\}$.
As before $\langle f,\tilde{m}_{\xi}\rangle$ is defined by
\begin{equation}\label{def-analysis-t}
\langle f,\tilde{m}_{\xi}\rangle
:= \sum_{\eta\in\cX}\langle \tilde{m}_{\xi}, \psi_\eta \rangle \langle f,\tilde\psi_{\eta}\rangle,
\end{equation}
where the series converges absolutely.
\end{theorem}

These theorems follow from Lemmas~\ref{lem:frame-molec-t} - \ref{lem:molecule-ad-t}
exactly as Theorems~\ref{thm:molec-synthesis} - \ref{thm:molec-analysis}
follow from Lemmas~\ref{lem:frame-molec} - \ref{lem:molecule-ad}.
We omit the details.

\subsection{Smooth atomic decomposition}\label{sec:atomic-decomp}

Here we focus on decompositions, where the building blocks are compactly supported smooth functions - smooth atoms.
We shall only consider atomic decompositions of $\FF$-spaces but, as before, the results hold for $\BB$-spaces as well.

\smallskip

\noindent
{\bf Definition of smooth atoms.}
We say that $\{a_{\xi}\}_{\xi\in\cX}$ is a family of smooth atoms for $\FF^{s}_{pq},$
if there exist integers
\begin{equation}\label{def-KtK}
K\ge(\lfloor(\cJ-s)/2\rfloor +1)_+
\quad\hbox{and}\quad
\tilde{K}\ge (\lfloor s/2\rfloor +2)_+
\end{equation}
and a family of functions $\{b_{\xi}\}_{\xi\in\cX}$, $b_\xi\in D(L^K)$, such that
for any $\xi\in\cX_j$, $j\in\bZ$,
\begin{equation}\label{fsaa}
a_{\xi}=L^{K}b_{\xi},
\end{equation}
\begin{equation}\label{fsab}
|L^{n}a_{\xi}(x)| \le \ell(\xi)^{-2n}|\AA_{\xi}|^{-1/2}
\quad\hbox{for $0\le n\le\tilde{K}$,}
\end{equation}
\begin{equation}\label{fsac}
|L^{\nu}b_{\xi}(x)|\le\ell(\xi)^{2(K-\nu)}|\AA_{\xi}|^{-1/2}
\quad\hbox{for $0\le \nu\le K$, and}
\end{equation}
\begin{equation}\label{fsad}
\supp L^\nu b_\xi \subset c\AA_{\xi}
\quad\hbox{for $0\le \nu\le K$,}
\end{equation}
where $c>0$ is a constant independent of $\xi$.

Clearly a family $\{a_\xi\}$ of smooth atoms is a family of smooth synthesis molecules.

%%%%%%% Theorem

\begin{theorem}\label{thm:atomicdecomposition}
Let $s\in\R$ and $0<p, q<\infty$.
Then for every $f\in \FF^{s}_{pq}$
there exist a family of smooth atoms $\{a_\xi\}_{\xi\in\cX}$
and a sequence $\{t_\xi\}_{\xi\in\cX}$ of complex numbers such that
\begin{equation}\label{est:atom-1}
f=\sum\limits_{\xi\in\cX}t_{\xi}a_{\xi}
\quad\text{and}\quad
\|t\|_{\ff^{s}_{pq}}\le c\|f\|_{\FF^{s}_{pq}},
\end{equation}
where the series converges in $\cS'/\PP$ and in the norm of $\FF^{s}_{pq}$.

Conversely, for every family of smooth atoms $\{a_{\xi}\}_{\xi\in\cX}$
\begin{equation}\label{est:atom-2}
\Big\|\sum_{\xi\in\cX}t_\xi a_\xi\Big\|_{\FF^s_{pq}} \le c\|t\|_{\ff^{s}_{pq}}.
\end{equation}
\end{theorem}

\begin{proof}
Since a family of smooth atoms is also a family of smooth synthesis molecules,
(\ref{est:atom-2}) follows readily by Theorem~\ref{thm:molec-synthesis}.

To prove the first part of the theorem, we shall use the compactly supported frames
$\{\theta_\xi\}$ from \S\ref{subsec:construction}.
In the construction of $\{\theta_\xi\}$ and $\{\tilde\theta_\xi\}$ in \S\ref{subsec:construction}
we impose in addition the condition
that the constant $K$ in (\ref{cond-NK}) be larger than the constants $K$ and $\tilde K$ from (\ref{def-KtK}).
We also choose the parameters $s_0, p_0, p_1, q_0$ so that
$(s, p, q)\in\Omega$ with $\Omega$ from (\ref{omega}).
From Theorem \ref{thm:main} we have for any $f\in \FF^{s}_{pq}$
$$
f=\sum_{\xi\in\cX} \langle f, \tilde\theta_\xi\rangle \theta_\xi
\quad \hbox{(convergence in $\cS'/\PP$ and in $\FF^{s}_{pq}$)}
$$
%with the convergence in $\cS'/\PP$ and in $\FF^{s}_{pq}$,
and
$
\|f\|_{\FF^{s}_{pq}} \sim \|\{\langle f, \tilde\theta_\xi\rangle\}\|_{\ff^{s}_{pq}}.
$
Therefore, it only remains to show that there exists a~constant $\sc>0$ such that
$\{\sc \theta_{\xi}\}_{\xi\in\cX}$ is a family of smooth atoms.
%
%From (\ref{support}) we have
%$$
%\supp \theta_{\xi} \subset B(\xi,cRb^{-j})=c\AA_{\xi}.
%$$
By definition
$$
\theta_{\xi}(x):=|A_{\xi}|^{1/2}\Theta(b^{-j}\sqrt{L})(x,\xi),
\quad\xi\in\cX_j, j\in \bZ.
$$
For each $\xi\in\cX_j$, $j\in\bZ$, we set
$$
a_\xi(x):= \sc \theta_{\xi}(x) = \sc |A_{\xi}|^{1/2}\Theta(b^{-j}\sqrt{L})(x,\xi)
$$
and
$$
b_{\xi}(x):= \sc |A_{\xi}|^{1/2}L^{-K}\Theta(b^{-j}\sqrt{L})(x,\xi),
$$
where $\sc>0$ is a constant to be determined.
Evidently, $a_\xi=L^K b_\xi$.
Consider the function
$$
g(u):=u^{-2(K-\nu)}\Theta(u), \quad 0\le \nu\le K.
$$
Clearly, $g\in \cS(\R)$, $g$ is real-valued and even, and
$$
g(b^{-j}\sqrt{L})= b^{2j(K-\nu)}L^{\nu-K}\Theta(b^{-j}\sqrt{L}).
$$
Hence
\begin{equation}\label{Lb-est}
L^\nu b_{\xi}(x)= \sc |A_{\xi}|^{1/2} b^{-2j(K-\nu)} g(b^{-j}\sqrt{L})(x, \xi).
\end{equation}
From (\ref{supp-hat-Theta})
$
\supp \hat{g} = \supp \cF(u^{-2(K-\nu)}\Theta(u)) \subset [-R, R]
$
and, therefore, using Proposition~\ref{prop:finite-sp} it follows that
$$
\supp L^\nu b_{\xi} \subset B(\xi, \ct Rb^{-j}) \subset cB_\xi.
$$
On the other hand, from Theorem~\ref{thm:S-local-kernels} it follows that
\begin{equation*}
\big|g(b^{-j} \sqrt L)(x, \xi)\big| \le c|B(\xi, b^{-j})|^{-1}
\end{equation*}
and on account of (\ref{Lb-est}) we obtain
$$
|L^{\nu}b_{\xi}(x)|\le c\sc\ell(\xi)^{2(k-\nu)}|\AA_{\xi}|^{-1/2},
\quad 0\le\nu\le K.
$$
Just as above we obtain as a consequence of Theorem~\ref{thm:S-local-kernels} that
$$
|L^{n}a_{\xi}(x)| \le c\sc\ell(\xi)^{-2n}|\AA_{\xi}|^{-1/2},
\quad 0\le n\le \tilde{K}.
$$
Finally, choosing the constant $\sc$ sufficiently small it follows that
$\{\sc \theta_{\xi}\}_{\xi\in\cX}$ is a family of smooth atoms.
\end{proof}

\section{Spectral multipliers}\label{sec:multipliers}

We next utilize almost diagonal operators (\S\ref{sec:almost-diag}) and smooth molecules (\S\ref{sec:molecules})
to establish the boundedness of spectral multipliers of Mihlin type on homogeneous Triebel-Lizorkin spaces.

\begin{theorem}\label{thm:multipliers}
Let $s\in\R$, $0<p<\infty$, and $0<q\le\infty$.
Suppose $m\in C^\kk(\R)$ for some $\kk>\cJ+d/2$ with %$\cJ:=\frac{d}{\min\{1, p, q\}}$,
$\cJ:=d/\min\{1, p, q\}$,
$m$ is even and real-valued,
%$m^{(2\nu+1)}(0)=0$ for $\nu\ge0$, $2\nu+1\le \kk$,
and
\begin{equation}\label{multiplier}
\sup_{\lambda\in\R_+}|\lambda^{\nu}m^{(\nu)}(\lambda)|<\infty,
\quad 0\le\nu\le \kk.
\end{equation}
Then the operator $m(\sqrt{L})$ is bounded on $\FF^{s}_{pq}$, that is,
\begin{equation*}
\|m(\sqrt{L})f\|_{\FF^{s}_{pq}} \le c \|f\|_{\FF^{s}_{pq}},
\quad\forall f\in\FF^{s}_{pq}.
\end{equation*}
Here $m(\sqrt{L})f$ for $f\in\FF^{s}_{pq}$ is defined by
\begin{equation}\label{multipl-1}
m(\sqrt{L})f := \sum_{\xi\in\cX}\langle f, \tilde{\psi}_\xi \rangle m(\sqrt{L})\psi_\xi
\qquad \hbox{$($convergence in $\cS'/\PP$$)$,}
\end{equation}
that is,
\begin{equation}\label{multipl-2}
\langle m(\sqrt{L})f, \phi\rangle
:= \sum_{\xi\in\cX}\langle f, \tilde{\psi}_\xi \rangle \langle m(\sqrt{L})\psi_\xi, \phi\rangle,
\quad \forall \phi\in \cS_\infty,
\end{equation}
where the series converges absolutely.
The motivation for the above definition is the fact that for any $f\in\cS'/\PP$ one has
$f=\sum_{\eta\in\cX}\langle f,\tilde\psi_\eta \rangle\psi_\eta$ in $\cS'/\PP$
$($Theorem~\ref{thm:frames}$)$.
\end{theorem}

\begin{proof}
Let $f\in\FF^{s}_{pq}$, $s\in\R$, $0<p<\infty$, and $0<q\le\infty$.
We first show that the series in (\ref{multipl-1}) converges absolutely
and hence $m(\sqrt{L})f$ is well defined.
Suppose the point $x_0$ from the definition of distributions in \S\ref{subsec:basic-facts}
belongs to $A_{\xi_0}$ for some $\xi_0\in\cX_0$.
We claim that any test function $\phi\in\cS_\infty$ is a constant multiple of
a smooth analysis and synthesis molecule centered at $\xi_0$ (\S\ref{subsec:molecules-BF}).
Indeed, by (\ref{norm-S}) it follows that for any $\nu \ge 0$ and $\sigma>0$
$$
|L^\nu\phi(x)| \le c\big(1+\rho(x, x_0)\big)^{-\sigma} \le c\big(1+\rho(x, \xi_0)\big)^{-\sigma}
$$
and the claim follows.

We next show that there exists a constant $\cf>0$ such that
$\{\cf m(\sqrt{L})\psi_\xi\}_{\xi\in\cX}$ is a family of smooth synthesis molecules.
We shall carry out the proof of this claim just as in the proof of the first part of Lemma~\ref{lem:frame-molec}.
Write
$$
\fm_\xi(x):=\cf m(\sqrt{L})\psi_\xi(x)= \cf |A_{\xi}|^{1/2}[m(\sqrt{L})\Psi(b^{-j}\sqrt{L})](x,\xi), \quad \xi\in\cX_j,
$$
where we used (\ref{def-frame}) and Remark~\ref{rem:calculus}.
Let $g(u):= u^{2\nu}m(b^{j}u)\Psi(u)$ for an arbitrary $\nu\ge 0$.
Clearly,
$$
L^\nu m(\sqrt{L})\Psi(b^{-j}\sqrt{L})= b^{2\nu j}g(b^{-j}\sqrt{L}),
$$
$g\in C^{\kk}(\R_+)$,
$\supp g\subset [b^{-1}, b]$,
and for $0\le n\le \kk$
\begin{align*}
|g^{(n)}(u)| \le c\max_{0\le r\le n} |m^{(r)}(b^ju)|b^{jr}
\le c\max_{0\le r\le \ell}\sup_{\lambda\in\R_+}|\lambda^{\nu}m^{(\nu)}(\lambda)| \le c<\infty.
\end{align*}
Then by Theorem~\ref{thm:S-local-kernels}, applied to $g$, it follows that
\begin{align*}
|[L^\nu m(\sqrt{L})\Psi(b^{-j}\sqrt{L})](x, \xi)|
\le \frac{c b^{\nu j}|B(\xi, b^{-j})|^{-1}}{(1+b^j\rho(x,\xi))^{\kk-d/2}}.
\end{align*}
Since $\kk >\cJ+d/2$ we may choose $\MM$ so that
$\cJ<\MM\le \ell -d/2$.
Now, using that $|A_\xi| \sim |B(\xi, b^{-j})| \sim |B_\xi|$ for $\xi\in\cX_j$
we arrive at
\begin{align*}
|L^\nu\fm_\xi(x)|
=|\cf L^\nu m(\sqrt{L})\psi_\xi(x)|
\le \frac{c \cf \ell(\xi)^{-2\nu}|B_\xi|^{-1/2}}{(1+\ell(\xi)^{-1}\rho(x,\xi))^{\MM}},
\quad \xi\in\cX_j, \;\; 0\le \nu\le N.
\end{align*}
This shows that $\{\fm_\xi\}$ %$\cf m(\sqrt{L})\psi_\xi$
obey (\ref{fssmbb})-(\ref{fssmb})
if the constant $\cf$ is sufficiently small.

%%%%%%%%%%%%%%%

Assume $s\le \cJ$ and let $\xi\in\cX_{j}$.
Define
$$
b_\xi(x):=\cf L^{-K}m(\sqrt{L})\psi_{\xi}(x)=\cf |A_{\xi}|^{1/2}[L^{-K}m(\sqrt{L})\Psi(b^{-j}\sqrt{L})](x,\xi).
%\quad\hbox{(see (\ref{def-frame}))}
$$
Hence
$$
L^Kb_\xi(x)=\cf m(\sqrt{L})\psi_\xi(x)= \fm_\xi(x).
$$
Assuming that $0\le\nu\le K$, $K=\lfloor (\cJ-s)/2\rfloor+1$,
we consider the following function
$h(u):=u^{2(\nu-K)}m(b^ju)\Psi(u)$.
Clearly
$$
L^{\nu-K}m(\sqrt{L})\Psi(b^{-j}\sqrt{L})= b^{-2j(K-\nu)}h(b^{-j}\sqrt{L}),
$$
$h\in C^\kk(\R_+)$ and $\supp g \subset [b^{-1},b]$.
Furthermore, for $0\le n\le \kk$
\begin{align*}
|h^{(n)}(u)| \le c\max_{0\le r\le n} |m^{(r)}(b^ju)|b^{jr}
\le c\max_{0\le r\le \ell}\sup_{\lambda\in\R_+}|\lambda^{\nu}m^{(\nu)}(\lambda)| \le c<\infty.
\end{align*}
As before we choose $\MM$ so that $\cJ<\MM\le \ell -d/2$.
Then by Theorem~\ref{thm:S-local-kernels}, applied to $h$, we infer
\begin{align*}
|L^{\nu}b_{\xi}(x)|
\le \frac{c \cf |A_{\xi}|^{1/2}b^{-2j(K-\nu)}}{|B(\xi, b^{-j})|(1+b^j\rho(\xi, x))^{\kk-d/2}}
\le \frac{c \cf \ell(\xi)^{2(K-\nu)}|\AA_{\xi}|^{-1/2}}{(1+\ell(\xi)^{-1}\rho(x,\xi))^{\MM}}.
\end{align*}
This shows that $\fm_\xi:=\cf m(\sqrt{L})\psi_\xi$ verifies (\ref{fssmc})
if $\cf\le c^{-1}$.
Therefore, if the constant $\cf$ is sufficiently small
$\{\fm_\xi\}$ is a family of smooth synthesis molecules.

%%%%%%%%%%

Given $\phi\in\cS_\infty$ write
$$
d_{\xi_0\eta}:=\langle m(\sqrt{L})\psi_\eta, \phi\rangle.
$$
%We use the fact that $\{\fm_\eta\}_{\xi\in\cX}$ with
From above we know that there exists a constant $\cf>0$ such that $\{\cf m(\sqrt{L})\psi_\eta\}_{\eta\in\cX}$ is
a family of smooth synthesis molecules for $\ff^s_{pq}$,
and $\phi$ is a smooth analysis molecule for $\ff^s_{pq}$.
Then applying Lemma~\ref{lem:molecule-ad}
we conclude that there exist constants $c, \delta>0$ such that
$$
|d_{\xi_0\eta}| \le c\omega_\delta(\xi_0, \eta), \quad \forall \eta\in\cX,
$$
where $\omega_\delta(\xi, \eta)$ is defined in (\ref{def-omega-d-1}).

On the other hand, by Theorem~\ref{thm:F-character}, for any $f\in\FF^s_{pq}$
the sequence $\{\langle f,\tilde\psi_{\eta}\rangle\}$
belongs to $\ff^s_{pq}$ and hence $\{|\langle f,\tilde\psi_{\eta}\rangle|\}\in\ff^s_{pq}$.
From this and the boundedness of the operator with matrix $\{\omega_\delta(\xi, \eta)\}$ on $\ff^s_{pq}$
it follows that
$$
\sum_{\eta\in\cX}|\langle m(\sqrt{L})\psi_\eta, \phi\rangle| |\langle f,\tilde\psi_{\eta}\rangle| <\infty.
$$
Thus the absolute convergence of the series in (\ref{multipl-2}) is established.

By Theorem~\ref{thm:F-character} it follows that to prove the theorem it suffices to show that
for any $f\in \FF^s_{pq}$
\begin{equation}\label{multipl-est}
\big\|\big(\langle m(\sqrt{L})f,\psi_\xi \big\rangle\big)\big\|_{\ff^s_{pq}}
\le c \|(\langle f,\tilde\psi_\xi \rangle)\|_{\ff^s_{pq}}.
\end{equation}
Let $f\in \FF^s_{pq}$. Then by (\ref{multipl-2})
\begin{equation}\label{mL-f-psi}
\langle m(\sqrt{L})f, \psi_\xi \rangle
=\sum_{\eta\in\cX}\langle f,\tilde\psi_\eta \rangle \langle m(\sqrt{L})\psi_\eta, \psi_\xi\rangle,
\quad \xi\in\cX.
\end{equation}
Let $A$ be the operators with matrix
$$
(a_{\xi\eta})_{\xi,\eta\in\cX} := (\langle m(\sqrt{L})\psi_\eta, \psi_\xi\rangle)_{\xi,\eta\in\cX}.
$$
From above we know that there exists a constant $\cf>0$  such that
$\{\cf m(\sqrt{L})\psi_\eta\}_{\eta\in\cX}$ is a family of smooth synthesis molecules for $\ff^s_{pq}$.
Also, by Lemma~\ref{lem:frame-molec} there exists a constant $\sc>0$ such that
$\{\sc\psi_\xi\}_{\xi\in\cX}$ is a family of smooth analysis molecules.
Then by Lemma~\ref{lem:molecule-ad} and Theorem~\ref{thm:AlmDiag} it follows that
the operator $A$ is bounded on $\ff^s_{pq}$.

Let $f\in\FF^s_{pq}$. By the boundedness of the operator $A: \ff^s_{pq}\to \ff^s_{pq}$
%and the analysis operator $T_{\tilde\psi}: \FF^s_{pq}\to \ff^s_{pq}$ from (\ref{anal_synth_oprts2})
%(Theorem~\ref{thm:F-character}),
and (\ref{mL-f-psi}) we infer
\begin{align*}
\|(\langle m(\sqrt{L})f, \psi_\xi \rangle)\|_{\ff^s_{pq}}
= \|A(\langle f,\tilde\psi_\eta \rangle)\|_{\ff^s_{pq}}
\le c\|(\langle f,\tilde\psi_\eta \rangle)\|_{\ff^s_{pq}},
\end{align*}
%\begin{align*}
%\|(\langle m(\sqrt{L})f, \psi_\xi \rangle)\|_{\ff^s_{pq}}
%= \|AS_{\tilde\psi}f\|_{\ff^s_{pq}}
%\le c\|S_{\tilde\psi}f\|_{\ff^s_{pq}}
%\le c\|f\|_{\FF^s_{pq}},
%\end{align*}
which verifies (\ref{multipl-est}).
\end{proof}

\begin{remark}
Several clarifying remarks about spectral multipliers are in order.

$(a)$ Spectral multipliers like the ones from Theorem~\ref{thm:multipliers} can be established for
the spaces $\tFF^{s}_{pq}$, where the condition $\kk>\cJ+d/2$ is replaced by $\kk>\cJ+d/2+|s|$.

$(b)$ Theorem~\ref{thm:multipliers} also holds for Besov spaces $\BB^{s}_{pq}$
and with the above replacement to $\tBB^{s}_{pq}$,
where $\cJ:=d/\min\{1,p\}$.

$(c)$ If we restrict the hypotheses of Theorem~\ref{thm:multipliers} to the case when
$(M,\rho,\mu)$ is an Ahlfors $d$-regular space, meaning that there exists a constant $c_4\ge 1$ such that
\begin{equation}\label{Ahlfors}
c_4^{-1}r^{d}\le|B(x,r)|\le c_4 r^{d}, \quad\forall x\in M,\; \forall r>0,
\end{equation}
then it suffices to assume that $\kk>\cJ$ as in \cite{FJ2} rather than $\kk>\cJ+d/2$.
$($The only difference in the proof is that when applying Theorem~\ref{thm:S-local-kernels} to $g$
we may use that $|B(x,b^{-j})|\sim |B(\xi,b^{-j})|$, which makes the difference.$)$
We omit the further details.
\end{remark}

\section{Atomic and molecular decompositions in the inhomogeneous case}\label{sec:inhomogeneous-case}

Inhomogeneous Besov and Triebel-Lizorkin spaces in the general setting, described in \S\ref{Introduction},
have been developed in \cite{KP}.
An advantage of the inhomogeneous spaces over the homogeneous spaces is that they are defined also in the case when
the set $M$ is compact, like the sphere, ball, and more general compact Riemannian manifolds.
On the other hand, this theory is more coherent in the homogeneous case.

We next briefly indicate how the atomic and molecular decompositions developed so far should be changed
in the inhomogeneous case.
Generally speaking in the inhomogeneous case the frequencies corresponding to eigenvalues
$0\le \lambda \le 1$ are grouped together.
Thus, in the definitions of inhomogeneous Besov and Triebel-Lizorkin spaces
(Definitions~\ref{def-B-spaces}-\ref{def-F-spaces}) the terms $\varphi_j(\sqrt{L})f$, $j\le 0$,
are replace by one term $\varphi_0(\sqrt{L})f$, where $\varphi_0\in C^\infty(\R_+)$ is such that
$\supp \varphi_0\subset [0, 2]$ and $|\varphi_0(\lambda)|>0$ on $[0, 2^{3/4}]$.
The frames are of the form
$\{\psi_\xi\}_{\xi\in\cX}$ and $\{\tilde\psi_\xi\}_{\xi\in\cX}$, where
$\cX = \cup_{j\ge 0} \cX_j$,
hence, in the definition of the inhomogeneous Besov and Triebel-Lizorkin sequence spaces spaces
sets $\cX_j$ are involved with $j\ge 0$.
It should be pointed out that the convergence in the inhomogeneous case is simpler than the one in the homogeneous case.
For more details, see \cite{KP}.

The definition of almost diagonal operators in the inhomogeneous case is the same as
in the homogeneous case, but $j\ge 0$.
Further, the definitions of smooth synthesis and analysis molecules
$m_\xi$ and $\tilde{m}_\xi$ for $\xi\in \cX_j$, $j\ge1$,
are the same as in the homogeneous case (\S\ref{subsec:molecules-BF}, \S\ref{subsec:molecules-tilde-BF}),
but for $\xi\in \cX_0$ (the zero level) Condition (iii) on $m_\xi$ and $\tilde{m}_\xi$ is dropped.
The same modification is applied for the definition of smooth atoms.
All theorems about molecular and atomic decompositions established in the homogeneous case in \S\ref{sec:molecules}
hold in the inhomogeneous case as well with almost identical proofs.
The spectral multipliers established in Theorem~\ref{thm:multipliers} are the same as in the inhomogeneous case.
We refrain from providing further details here.

\section{Appendix}\label{sec:appendix}

%-------------- Proof of Lemma 2.8

\subsection{Proof of Lemma~\ref{lem:est-discr-sum}}

The proof of Lemma~\ref{lem:est-discr-sum} relies on the following two lemmata:

\begin{lemma}\label{lem:est-num-net}
Let $\cX$ be a $\delta-$net on $M$ and $0 <\delta \le \delta^\star$. Then
\begin{equation}\label{num-net}
\#\{\cX\cap B(x,\delta^\star)\} \le c_{0}6^d\left(\delta^\star/\delta\right)^{d},
\quad \forall x\in M.
\end{equation}
Here $c_0$ is the constant from $(\ref{doubling-0})$.
\end{lemma}

\begin{proof}
It is easily seen that if $\xi\in\cX\cap B(x,\delta^{*})$,
then
$B(\xi,\delta)\subset B(x,2\delta^\star)\subset B(\xi,3\delta^\star)$.
Therefore, for every $\eta\in\cX\cap B(x,\delta^\star)$
$$
\sum_{\xi\in\cX\cap B(x,\delta^\star)}|B(\xi,\delta/2)|
\le |B(x,2\delta^\star)|
\le |B(\eta, 3\delta^\star)|
\le c_06^{d}\left(\delta^\star/\delta\right)^{d}|B(\eta,\delta/2)|,
$$
where for the last inequality we used (\ref{doubling}).
Summing up the above inequalities over all $\eta\in\cX\cap B(x,\delta^\star)$
leads to (\ref{num-net}).
\end{proof}

\begin{lemma}\label{lem:basic-est}
Suppose $\sigma >d$ and let $\cX$ be a $\delta-$net on $M$, $\delta>0$.
Then for any $x\in M$ and $\delta^\star\ge \delta$
\begin{equation}\label{basic-est}
\sum_{\xi\in\cX}\Big(1+\frac{\rho(x, \xi)}{\delta^\star}\Big)^{-\sigma}
\le \frac{c_06^d 2^\sigma}{1-2^{d-\sigma}}\Big(\frac{\delta^\star}{\delta}\Big)^d.
\end{equation}
\end{lemma}

\begin{proof}
Set $\Omega_0 :=\{\xi\in\cX: \rho(x,\xi) \le \delta^\star\}$ and
$$
\Omega_{j}:=\{\xi\in\cX: 2^{j-1}\delta^\star<\rho(x,\xi) \le 2^{j}\delta^\star\}, \quad j\ge 1.
$$
Then using Lemma~\ref{lem:est-num-net} we get
\begin{align*}
\sum_{\xi\in\cX}\Big(1+\frac{\rho(x, \xi)}{\delta^\star}\Big)^{-\sigma}
&\le \sum_{j\ge 0}\sum_{\xi\in\Omega_j}\Big(1+\frac{\rho(x, \xi)}{\delta^\star}\Big)^{-\sigma}
\\
&\le \sum_{j\geq0}\#\{\cX\cap B(x,2^j\delta^\star)\}2^{-(j-1)\sigma}
\\
&\le c_06^d 2^\sigma (\delta^\star/\delta)^d\sum_{j\ge 0}2^{-j(\sigma-d)}
\\
&\le \frac{c_{0}6^d 2^\sigma}{1-2^{d-\sigma}}\Big(\frac{\delta^\star}{\delta}\Big)^d,
\end{align*}
which confirms (\ref{basic-est}).
\end{proof}

\smallskip

\noindent
{\em Proof of Lemma~\ref{lem:est-discr-sum}.}
Under the hypotheses of Lemma~\ref{lem:est-discr-sum}, denote by $\Sigma$ the quantity
on the left in (\ref{est-discr-sum}) and set
$$
\cX':=\{\xi\in\cX: \rho(x, \xi)\ge \rho(x, y)/2\}
\quad\hbox{and}\quad
\cX'':=\{\xi\in\cX: \rho(y, \xi) > \rho(x, y)/2\}.
$$
Then
$\Sigma \le \sum_{\xi\in \cX'}\dots + \sum_{\xi\in \cX''}\dots =:\Sigma'+\Sigma''$.
To estimate $\Sigma''$ we use Lemma~\ref{lem:basic-est} and obtain
\begin{align*}
\Sigma'' \le \frac{c}{(1+\delta_2^{-1}\rho(x,y))^\sigma}
\sum_{\xi\in\cX}\frac{1}{(1+\delta_1^{-1}\rho(y,\xi))^\sigma}
\le \frac{c(\delta_1/\delta)^d}{(1+\delta_2^{-1}\rho(x,y))^\sigma}.
\end{align*}

To estimate $\Sigma'$ we consider two cases.

{\em Case 1:} $\delta_2^{-1}\rho(x, y) \ge 1$.
Just as above we obtained
\begin{align*}
\Sigma' \le \frac{c}{(1+\delta_1^{-1}\rho(x,y))^\sigma}
\sum_{\xi\in\cX}\frac{1}{(1+\delta_2^{-1}\rho(y,\xi))^\sigma}
\le \frac{c(\delta_2/\delta)^d}{(1+\delta_1^{-1}\rho(x,y))^\sigma}
\end{align*}
and using that $\delta_2^{-1}\rho(x, y) \ge 1$
\begin{align*}
\Sigma' &\le \frac{c(\delta_2/\delta_1)^d(\delta_1/\delta)^d}{(\delta_1^{-1}\rho(x,y))^\sigma}
= \frac{c(\delta_1/\delta)^d}{(\delta_2^{-1}\rho(x,y))^d(\delta_1^{-1}\rho(x,y))^{\sigma-d}}
\\
&\le \frac{c(\delta_1/\delta)^d}{(\delta_2^{-1}\rho(x,y))^\sigma}
\le \frac{c2^\sigma(\delta_1/\delta)^d}{(1+\delta_2^{-1}\rho(x,y))^\sigma}.
\end{align*}

{\em Case 2:} $\delta_2^{-1}\rho(x, y) < 1$.
We use Lemma~\ref{lem:basic-est} to obtain
\begin{align*}
\Sigma' \le \sum_{\xi\in\cX}\frac{1}{(1+\delta_1^{-1}\rho(x,\xi))^\sigma}
\le c(\delta_1/\delta)^d
\le \frac{c2^\sigma(\delta_1/\delta)^d}{(1+\delta_2^{-1}\rho(x,y))^\sigma}.
\end{align*}

Putting the above estimates together we arrive at (\ref{est-discr-sum}).
\qed

\subsection{Proof of Theorem \ref{thm:AlmDiag}}

We shall carry out the proof of this theorem only for the spaces $\ff^s_{pq}$ and $\tbb^s_{pq}$;
the proof in the case of the spaces $\tff^s_{pq}$ and $\bb^s_{pq}$ is similar and will be omitted.

We need two lemmata.

%%%%%%%% Lemma

\begin{lemma}\label{lem:mrax}
Let  $0<t\le 1$ and $\MM> d/t$. Then for any sequence of complex numbers
$\{h_{\eta}\}_{\eta\in \cX_{m} }$,   $m\in \bZ$,  we have for $x\in A_\xi$, $\xi\in\cX$,
$$
\sum_{\eta\in\cX_m} |h_{\eta}|\biggl(1+\frac{\rho(\xi,\eta)}
{\max \{\ell(\xi), \ell(\eta)\}}\biggr)^{-\MM}
\le c \max \left\{b^{(m-j)d/t}, 1 \right\} \cM_t\Big(\sum_{\eta\in \cX_m}
|h_{\eta}|\ONE_{A_\eta}\Big)(x),
$$
where the constant $c>0$ depends only on $t, \MM$, and the constant $c_0$ from $(\ref{doubling-0})$.
\end{lemma}

This is Lemma 7.1 in \cite{DKKP}.

We shall also need the well known Hardy inequalities, given in the following

\begin{lemma}\label{lem:hardy}
Let $\gamma >0$, $0<q<\infty$, $b>1$, and $a_m \ge 0$ for $m\in\bZ$. Then
\begin{equation}\label{hardy-1}
\Big(\sum_{j\in\bZ}\Big(\sum_{m\ge j} b^{-(m-j)\gamma}a_m\Big)^q\Big)^{1/q}
\le c\Big(\sum_{m\in\bZ} a_m^q\Big)^{1/q}
\end{equation}
and
\begin{equation}\label{hardy-2}
\Big(\sum_{j\in\bZ}\Big(\sum_{m\le j} b^{-(j-m)\gamma}a_m\Big)^q\Big)^{1/q}
\le c\Big(\sum_{m\in \bZ} a_m^q\Big)^{1/q},
\end{equation}
where the constant $c>0$ depends on $\gamma, q, b$.
\end{lemma}
The proof of the Hardy inequalities is standard and simple; we omit it.

\medskip

%------------------------------------------------------------------------------

Assume that the hypotheses of Theorem \ref{thm:AlmDiag} are valid for $\ff^s_{pq}$,
that is,
$A$ is an operator with matrix $(a_{\xi\eta})_{\xi,\eta\in\cX}$ such that
for some $\delta>0$
\begin{equation}\label{operator-A}
\|A\|_\delta :=\sup_{\xi, \eta\in\cX}\frac{|a_{\xi\eta}|}{\omega_{\xi\eta}(\delta)}\le c<\infty,
\end{equation}
where $\omega_{\xi\eta}(\delta)$ is defined in (\ref{def-omega-d-1}).
Also, let $h=\{h_\xi\}_{\xi\in\cX} \in\ff^s_{pq}$.
We next prove the estimate
\begin{equation}\label{F-AlmDiag}
\|\Aa h\|_{ \ff^s_{pq}}\le c\|\Aa\|_\delta \|h\|_{\ff^s_{pq}}.
\end{equation}
We only consider the case: $q<\infty$; the case when $q=\infty$ is easier and we omit it.

We have
$(Ah)_\xi= \sum_{\eta\in\cX} a_{\xi\eta}h_\eta$.
(By the proof below it follows that the series converges absolutely.)
Using this in the definition of $\|\cdot\|_{\ff^s_{pq}}$ in (\ref{def-f-space}), we have
\begin{align*}
\|\Aa h\|_{\ff^s_{pq}}
&:=\Big\|\Big(\sum_{\xi\in\cX}
\bigl[\ell(\xi)^{-s}|(\Aa h)_\xi|\tONE_{A_{\xi}}(\cdot)\bigr]^q\Big)^{1/q}\Big\|_{L^p}\\
&\le \Big\|\Big(\sum_{\xi\in\cX}\bigl[\ell(\xi)^{-s}\sum_{\eta\in\cX} |a_{\xi\eta}||h_{\eta}|
\tONE_{A_{\xi}}(\cdot)\bigr]^q\Big)^{1/q}\Big\|_{L^p}
\le c(\Sigma_1+\Sigma_2), \nonumber
\end{align*}
where
\begin{align*}
&\Sigma_1:=\Big\|\Big(\sum_{\xi\in\cX}\bigl[\ell(\xi)^{-s}
\sum_{\ell(\eta)\le \ell(\xi)}
|a_{\xi\eta}||h_{\eta}| \tONE_{A_{\xi}}(\cdot)\bigr]^q\Big)^{1/q}\Big\|_{L^p}
\quad
\mbox{and}\\
\quad
&\Sigma_2:=\Big\|\Big(\sum_{\xi\in\cX}\bigl[\ell(\xi)^{-s}
\sum_{\ell(\eta)>\ell(\xi)}
|a_{\xi\eta}||h_{\eta}|\tONE_{A_{\xi}}(\cdot)\bigr]^q\Big)^{1/q}\Big\|_{L^p}.
\end{align*}

%%%%%%%%%%%%%%%%% Estimate Sigma_1

To estimate $\Sigma_1$ we use (\ref{operator-A}).
By (\ref{def-omega-d-1}) we have whenever $\ell(\eta)\le \ell(\xi)$
$$
|a_{\xi\eta}|
\le c \|\Aa\|_\delta\Big(\frac{\ell(\eta)}{\ell(\xi)}\Big)^{\cJ+\delta-s}
\Big(\frac{|A_\xi|}{|A_\eta|}\Big)^{1/2}
\Big(1+\frac{\rho(\xi,\eta)}{\ell(\xi)}\Big)^{-\cJ-\delta}.
$$
Set $\LL_{\xi}(x):=\ell(\xi)^{-s}|A_\xi|^{-1/2}\ONE_{A_{\xi}}(x)$
and choose $t$ so that $0<t<\min\{1,p,q\}$ and $\cJ+ \delta -d/t>0$.
Then we have
\begin{align*}
&\|\Aa\|_\delta^{-1}\Sigma_1
%&\frac{\Sigma_1}{\|\Aa\|_\delta}
\le c  \Big\|\Big(\sum_{\xi\in \cX}\Big[\sum_{\ell(\eta)\le \ell(\xi)}
\Big( \frac{\ell(\eta)}{\ell(\xi)} \Big)^{\cJ+\delta -s}\Big(\frac{|A_\xi|}{|A_\eta|}\Big)^{1/2}\\
&\hspace{2.5in}\times
\Big(1+\frac{\rho(\xi,\eta)}{\ell(\xi) }\Big)^{-\cJ-\delta}
|h_{\eta}| \LL_{\xi}(\cdot)\Big]^q\Big)^{1/q}\Big\|_{L^p}\\
&= c  \Big\|\Big(\sum_{j\in\bZ}\sum_{\xi\in\cX_j}
\Big[\sum_{m\ge j}  b^{(j-m)
(\cJ+\delta )}\sum_{\eta\in\cX_m}\Big( \frac{\ell(\eta)}{\ell(\xi)} \Big)^{-s}\Big(\frac{|A_\xi|}{|A_\eta|}\Big)^{1/2}\\
&\hspace{2.5in}\times
|h_{\eta}|
 \bigr(1+b^j \rho(\xi,\eta) \bigr)^{-\cJ-\delta}\LL_{\xi}(\cdot)\Big]^{q}\Big)^{1/q}\Big\|_{L^p}.
\end{align*}
We now apply Lemma~\ref{lem:mrax}, the Hardy inequality (\ref{hardy-1}), and
the maximal inequality (\ref{max-ineq}) to obtain
\begin{align*}
\|\Aa\|_\delta^{-1}\Sigma_1
%\frac{\Sigma_1}{\|\Aa\|_\delta}
&\le c  \Big\|\Big(\sum_{j\in\bZ}\sum_{\xi\in\cX_j}
\Big[\sum_{m\ge j}  b^{(j-m)
(\cJ+\delta -d/t)}\\
& \hspace{0.7in} \times M_t\Big(\sum_{\eta\in\cX_m}
\Big( \frac{\ell(\eta)}{\ell(\xi)} \Big)^{-s}\Big(\frac{|A_\xi|}{|A_\eta|}\Big)^{1/2}|h_{\eta}|\ONE_{A_\eta}\Big)(\cdot)
\LL_{\xi}(\cdot)\Big]^q\Big)^{1/q}\Big\|_{L^p}\\
&\le c \Big\|\biggl(\sum_{j\in\bZ}
 \Big[\sum_{m\ge j} b^{(j-m)(\cJ+\delta -d/t)}
 M_t\Big(\sum_{\eta\in\cX_m}|h_{\eta}|\LL_{\eta}\Big)
 \Big]^q\Big)^{1/q}\Big\|_{L^p}\\
&\le c \Big\|\Big(\sum_{j\in\bZ}
\Big[  M_t\Big(\sum_{\xi\in\cX_j}|h_{\xi}|\LL_{\xi}\Big)\Big]^q\Big)^{1/q}\Big\|_{L^p}\\
&\le c \Big\|\Big(\sum_{j\in\bZ}
\Big[\sum_{\xi\in\cX_j}|h_{\xi}|\LL_{\xi}\Big]^q\Big)^{1/q}\Big\|_{L^p}
\le c\|h\|_{\ff^s_{pq}}.
\end{align*}

%%%%%%%%%%%%%%%%% Estimate Sigma_2

To estimate $\Sigma_2$ we use again Lemma~\ref{lem:mrax}, the Hardy inequality (\ref{hardy-2}) instead of (\ref{hardy-2}), and
the maximal inequality (\ref{max-ineq}).
The estimates for $\Sigma_1$ and $\Sigma_2$ imply (\ref{F-AlmDiag}).

\medskip

%-------------------------------- Besov spaces

We next proceed with the proof of the estimate
\begin{equation}\label{AlmDiag-tb}
\|\Aa h\|_{\tbb^s_{pq}}\le c\|\Aa\|_\delta \|h\|_{\tbb^s_{pq}}.
\end{equation}
Here we assume that $A$ is an operator with matrix $\{a_{\xi\eta}\}_{\xi\eta\in\cX}$
obeying (\ref{operator-A}) for some constants $\delta, c>0$,
where $\omega_{\xi\eta}(\delta)$ is defined in (\ref{def-omega-d-2}).
We also assume that $h=\{h_\xi\}_{\xi\in\cX} \in\tbb^s_{pq}$.
We consider the case when $p, q<\infty$;
the cases when $p=\infty$ or $q=\infty$ are easier to handle and we omit the details.

We know that
$|B(\xi,b^{-j})| \sim |A_\xi| \sim |B_\xi|$ for $\xi\in\cX_j$
and by Definition~\ref{def:b-spaces} it follows that
\begin{equation}\label{equiv-norms}
\|h\|_{\tbb^s_{pq}}
\sim \Big(\sum_{j\in\bZ}
\Big\|\sum_{\xi\in\cX_j}|A_\xi|^{-s/d}|h_\xi|\tONE_{A_{\xi}}(\cdot)\Big\|_{L^p}^q\Big)^{1/q},
\qquad \tONE_{A_{\xi}}:= |A_\xi|^{-1/2}\ONE_{A_{\xi}}.
\end{equation}
We have
$
(\Aa h)_{\xi}=\sum_{\eta\in\cX}a_{\xi\eta}h_{\eta}
$
and using (\ref{equiv-norms})
\begin{align*}\label{ssone-b}
\norm{\Aa h}_{\tbb^s_{pq}}
&\le c\Big(\sum_{j\in\bZ}
\Big\|\sum_{\xi\in\cX_j}|A_\xi|^{-s/d}|(\Aa h)_\xi|\tONE_{A_{\xi}}(\cdot)\Big\|_{L^p}^q\Big)^{1/q}\\
&\le c\Big(\sum_{j\in\bZ}
\Big\|\sum_{\xi\in\cX_j}|A_\xi|^{-s/d}\sum_{\eta\in\cX} |a_{\xi\eta}||h_{\eta}|
\tONE_{A_{\xi}}(\cdot)\Big\|_{L^p}^q\Big)^{1/q}
\le c(\Sigma_1+\Sigma_2),
\end{align*}
where
\begin{align*}
&\Sigma_1:= \Big(\sum_{j\in\bZ}
\Big\|\sum_{\xi\in\cX_j}|A_\xi|^{-s/d}\sum_{\ell(\eta)\le \ell(\xi)} |a_{\xi\eta}||h_{\eta}|
\tONE_{A_{\xi}}(\cdot)\Big\|_{L^p}^q\Big)^{1/q}
\quad
\mbox{and}\\
\quad
&\Sigma_2:= \Big(\sum_{j\in\bZ}
\Big\|\sum_{\xi\in\cX_j}|A_\xi|^{-s/d}\sum_{\ell(\eta) > \ell(\xi)} |a_{\xi\eta}||h_{\eta}|
\tONE_{A_{\xi}}(\cdot)\Big\|_{L^p}^q\Big)^{1/q}.
\end{align*}

%%%%%%%%%%%%%%%%% Estimate Sigma_1

We shall only estimate $\Sigma_1$.
%Suppose $\xi\in\cX_j$, $\eta\in\cX_m$, and $m\ge j$; hence $\ell(\eta)\le \ell(\xi)$.
Using that $\|\Aa\|_\delta<\infty$, see ({\ref{def-omega-d-2})-(\ref{AlmDiag-Norm}),
it readily follows that whenever $\ell(\eta)\le \ell(\xi)$
$$
|a_{\xi\eta}|
\le c\|\Aa\|_\delta\biggl(\frac{\ell(\eta)}{\ell(\xi)}\biggr)^{\cJ+\delta}
\biggl(\frac{|A_\xi|}{|A_\eta|} \biggr)^{s/d+1/2}
\bigg( 1+\frac{\rho(\xi,\eta)}{\ell(\xi)}\biggr)^{-\cJ-\delta}.
$$
Denote briefly $F_{\xi}:=|A_\xi|^{-s/d-1/2}\ONE_{A_{\xi}}(\cdot)$
and choose $t$ so that
$d/t=\cJ+\delta/2$.
Then $0<t<\min\{1,p\}$ and $\cJ+ \delta -d/t>0$.
We have
\begin{align*}
& \|A\|_\delta^{-1}\Big\|\sum_{\xi\in\cX_j}|A_\xi|^{-s/d}\sum_{\ell(\eta)\le \ell(\xi)} |a_{\xi\eta}||h_{\eta}|
\tONE_{A_{\xi}}(\cdot)\Big\|_{L^p}\\
&\le c\Big\|\sum_{\xi\in \cX_j}\sum_{\ell(\eta)\le \ell(\xi)}
\Big( \frac{\ell(\eta)}{\ell(\xi)} \Big)^{\cJ+\delta }\biggl(\frac{|A_\xi|}{|A_\eta|} \biggr)^{s/d+1/2}%\\
\Big(1+\frac{\rho(\xi,\eta)}{\ell(\xi) }\Big)^{-\cJ-\delta}
|h_{\eta}| F_{\xi}(\cdot)\Big\|_{L^p}\\
&= c\Big\|\sum_{\xi\in \cX_j}\sum_{m\ge j}
b^{(j-m)(\cJ+\delta)}\sum_{\eta\in\cX_m}\biggl(\frac{|A_\xi|}{|A_\eta|} \biggr)^{s/d+1/2}
|h_{\eta}|\bigr(1+b^j \rho(\xi,\eta) \bigr)^{-\cJ-\delta}F_{\xi}(\cdot)
\Big\|_{L^p}.
\end{align*}
We now apply Lemma~\ref{lem:mrax}
and the maximal inequality (\ref{max-ineq}) to obtain
\begin{align*}
& \|A\|_\delta^{-1}\Big\|\sum_{\xi\in\cX_j}|A_\xi|^{-s/d}\sum_{\ell(\eta)\le \ell(\xi)} |a_{\xi\eta}||h_{\eta}|
\tONE_{A_{\xi}}(\cdot)\Big\|_{L^p}\\
&\le c\Big\|\sum_{\xi\in \cX_j}\sum_{m\ge j}
b^{(j-m)(\cJ+\delta-d/t)}
M_t\biggl(\sum_{\eta\in\cX_m}
\biggl(\frac{|A_\xi|}{|A_\eta|} \biggr)^{s/d+1/2}|h_{\eta}|\ONE_{A_\eta}\biggr)(\cdot)
F_{\xi}(\cdot)
\Big\|_{L^p}\\
&= c\Big\|\sum_{m\ge j}
b^{(j-m)\delta/2}
M_t\biggl(\sum_{\eta\in\cX_m}|h_{\eta}|F_\eta\biggr)(\cdot)
\Big\|_{L^p}\\
&\le c\Big\|\sum_{m\ge j}
b^{(j-m)\delta/2}
\sum_{\eta\in\cX_m}|h_{\eta}|F_\eta(\cdot)
\Big\|_{L^p}.
\end{align*}
Consider the case when $p\ge 1$. Then applying the Hardy inequality (\ref{hardy-1}) we get
\begin{align*}
\|A\|_\delta^{-1}\Sigma_1
&=\Big(\sum_{j\in\bZ}\Big\|\sum_{\xi\in\cX_j}|A_\xi|^{-s/d}\sum_{\ell(\eta)\le \ell(\xi)} |a_{\xi\eta}||h_{\eta}|
\tONE_{A_{\xi}}(\cdot)\Big\|_{L^p}^q\Big)^{1/q}\\
&\le c\Big(\sum_{j\in\bZ}\Big\|\sum_{m\ge j}
b^{(j-m)\delta/2}
\sum_{\eta\in\cX_m}|h_{\eta}|F_\eta(\cdot)
\Big\|_{L^p}^q\Big)^{1/q}\\
&\le c\Big(\sum_{j\in\bZ}
\Big[\sum_{m\ge j} b^{(j-m)\delta/2}
\Big\|\sum_{\eta\in\cX_m}|h_{\eta}|F_\eta(\cdot)
\Big\|_{L^p}\Big]^q\Big)^{1/q}\\
&\le c\Big(\sum_{m\in\bZ}
\Big\|\sum_{\eta\in\cX_m}|h_{\eta}|F_\eta(\cdot)
\Big\|_{L^p}^q\Big)^{1/q}
\le c\|h\|_{\tbb^s_{pq}}.
\end{align*}
Now, let $0<p<1$. Then applying the $p$-triangle inequality and the Hardy inequality (\ref{hardy-1}) we get
\begin{align*}
\|A\|_\delta^{-1}\Sigma_1
&=\Big(\sum_{j\in\bZ}\Big\|\sum_{\xi\in\cX_j}|A_\xi|^{-s/d}\sum_{\ell(\eta)\le \ell(\xi)} |a_{\xi\eta}||h_{\eta}|
\tONE_{A_{\xi}}(\cdot)\Big\|_{L^p}^q\Big)^{1/q}\\
&\le c\Big(\sum_{j\in\bZ}\Big\|\sum_{m\ge j}
b^{(j-m)\delta/2}
\sum_{\eta\in\cX_m}|h_{\eta}|F_\eta(\cdot)
\Big\|_{L^p}^q\Big)^{1/q}\\
&\le c\Big(\sum_{j\in\bZ}
\Big[\sum_{m\ge j} b^{(j-m)p\delta/2}
\Big\|\sum_{\eta\in\cX_m}|h_{\eta}|F_\eta(\cdot)
\Big\|_{L^p}^p\Big]^{q/p}\Big)^{1/q}\\
&\le c\Big(\sum_{m\in\bZ}
\Big\|\sum_{\eta\in\cX_m}|h_{\eta}|F_\eta(\cdot)
\Big\|_{L^p}^q\Big)^{1/q}
\le c\|h\|_{\tbb^s_{pq}}.
\end{align*}

We similarly estimate $\Sigma_2$ and get the same bound.
The estimates for $\Sigma_1$ and $\Sigma_2$ yield (\ref{AlmDiag-tb}).
\qed

\end{document}